\documentclass[review]{elsarticle}
\pdfoutput=1
\usepackage{lineno,hyperref}
\modulolinenumbers[5]

\journal{Journal of Computational Physics}

\usepackage{calligra}
\usepackage[T1]{fontenc}

\usepackage{wrapfig}
\usepackage{enumerate}
\usepackage{amsmath,mathtools}
\usepackage{subcaption}
\usepackage{textcomp}
\usepackage{booktabs}
\usepackage{tikz}
\usetikzlibrary{arrows}
\usepackage[american]{circuitikz} 
\definecolor{darkgreen}{rgb}{0.412, 0.651, 0.388}
\usepackage{mathrsfs}

\usepackage{graphicx}
\usepackage{epstopdf}
\usepackage{amsmath,amssymb,amsthm,stmaryrd,textcomp}

\usepackage{listings}
\usepackage{tikz}
\allowdisplaybreaks

\newcommand{\Frac}[2]{\displaystyle \frac{#1}{#2}}
\newcommand{\deriv}[3]{\Frac{\partial^{#1} #2}{\partial {#3}^{#1}}}
\newcommand{\Deriv}[3]{\Frac{d^{#1} #2}{d {#3}^{#1}}}
\newcommand{\nabu}{{\boldsymbol \nabla}}
\newcommand{\grad}{\nabu}
\newcommand{\Div}{\grad \cdot}
\renewcommand{\vector}[1]{\boldsymbol{#1}}
\renewcommand{\matrix}[1]{\underline{\underline{\boldsymbol{#1}}}}
\newcommand{\dint}{\displaystyle\int}

\newcommand{\exP}{\mathscr{P}}
\newcommand{\exV}{\mathscr{V}}
\newcommand{\exC}{\mathscr{C}}

\usepackage{bm}

\newcommand{\revA}[1]{{\textcolor{black}{#1}}}
\newcommand{\revB}[1]{{\textcolor{black}{#1}}}

\newtheorem{remark}{Remark}
\newtheorem{theorem}{Theorem}

\begin{document}
\begin{frontmatter}

\title{Energy-based operator splitting approach for the time discretization of coupled systems of partial and ordinary differential equations for fluid flows: 
The Stokes case}

\author[LC]{Lucia Carichino}
\address[LC]{Department of Mathematical Sciences, Worcester Polytechnic Institute, \\
100 Institute Rd, Worcester, MA 01609, USA}

\author[GG]{Giovanna Guidoboni}
\address[GG]{Department of Electrical Engineering and Computer Science, University of Missouri,\\ 201 Naka Hall, Columbia, 65211, MO, USA}

\author[MS]{Marcela Szopos\corref{mycorrespondingauthor}}
\address[MS]{Universit\'e de Strasbourg, CNRS, IRMA UMR 7501, F-67000 Strasbourg, France}
\cortext[mycorrespondingauthor]{Corresponding author}
\ead{szopos@math.unistra.fr}

\begin{abstract}
The goal of this work is to develop a novel splitting approach for the numerical solution of multiscale problems involving the coupling between Stokes equations and ODE systems, 
as often encountered in blood flow modeling applications. The proposed algorithm is based on a semi-discretization in time based on operator splitting, whose design is guided by the 
rationale of ensuring that the physical energy balance is maintained at the discrete level. As a result, unconditional stability with respect to the time step choice is ensured by 
the implicit treatment of interface conditions within the Stokes substeps, whereas the coupling between Stokes and ODE substeps is enforced via appropriate initial conditions for 
each substep. Notably,  unconditional stability is attained without the need of subiterating between Stokes and ODE substeps. Stability and convergence properties of the proposed 
algorithm are tested on three specific examples for which analytical solutions are derived.

\end{abstract}


\end{frontmatter}

\section{Introduction}

Multiscale coupling between systems of partial differential equations (PDEs) and ordinary differential equations (ODEs) is often necessary when modeling complex problems 
arising in science, engineering and medicine. In particular, the present work is motivated by applications to blood flow modeling through the cardiovascular system, even 
though the resulting conceptual framework may be meaningful and applicable to a more general context of hydraulic networks.

The coupling between PDEs and ODEs in blood flow modeling has been utilized in different ways depending on the specific modeling needs. Many studies have used 
Windkessel-like models \cite{westerhof2009} as boundary conditions for three-dimensional (\textsc{3d}) blood flow simulations in regions of particular interest, as in 
\cite{vignon2006,Kim2009,grinberg2008,chabannes2015}. 
In addition, ODEs have been used to provide systemic descriptions of the cardiovascular system where \textsc{3d} regions are embedded,
 as in \cite{pennati1997,lagana2005,Lau2015,carichino2014}. In all these applications, the PDE/ODE coupling leads to 
interface conditions enforcing the continuity of mass and the balance of forces, which should also be preserved at the discrete level when solving the problem numerically. 

Many strategies have been proposed for the numerical solution of coupled PDE/ODE systems in the context of blood flow modeling. 
In particular, monolithic and splitting (or partitioned) schemes have been proposed, where the PDE and ODE systems are solved simultaneously or in separate substeps,
respectively. An extensive discussion about advantages and limitations of monolithic and splitting approaches can be found in \cite{bertoglio2013,quarteroni2016}.
The present contribution focuses on splitting techniques and on the properties of their modular structure. Among the many interesting contributions in this area, we mention
here those that are most closely related to our work. In \cite{quarteroni2001}, Quarteroni \textit{et al} consider the multiscale coupling between the Navier-Stokes 
equations in a rigid domain and a lumped parameter model. A splitting strategy based on subiterations between PDE and ODE solvers at each time step is proposed 
and assessed in different meaningful configurations. The splitting formulation was then used as an effective tool to prove the well-posedness of the coupled 
problem \cite{quarteroni2003}, in combination with appropriate fixed point results. In \cite{fouchet2015}, Fouchet-Incaux \textit{et al} compare the numerical stability 
of explicit and implicit coupling between the Stokes or Navier-Stokes equations and circuit-based models containing resistances and capacitances.
Unconditional stability was proved in the implicit case, whereas conditional stability was proved in the explicit case. 
In \cite{moghadam2013}, Moghadam \textit{et al} propose a time implicit approach to couple general lumped parameter models with a finite-element based solution of 
a Navier-Stokes problem in a \textsc{3d} domain. The algorithm combines the stability properties of monolithic approaches with the modularity of splitting algorithms.
The method is based on a Newton type iterative scheme, where data are exchanged between the two domains at each Newton iteration of the nonlinear Navier-Stokes solver to 
ensure convergence of both domains simultaneously. To the best of our knowledge, the splitting schemes that have been proposed for coupled PDE/ODE systems
in the context of fluid flow modeling so far, require subiterations between substeps, usually involving the values of pressure and flow rate at the multiscale interfaces, in order to 
achieve convergence of the overall algorithm. Depending on the mathematical properties of the models
in each substep and of the coupling between them, the convergence of such subiterations might become an issue, especially in the case of nonlinear problems.

The present study aims at providing a novel splitting scheme for coupled PDE/ODE systems for fluid flow that does not require subiterations between substeps to achieve stability for 
the overall algorithm. The scheme stability follows from ensuring that the physical energy balance is maintained at the discrete level
via a suitable application of operator splitting techniques \cite{marchuk1990,glowinski2003} to semi-discretize the problem in time, as in 
\cite{glowinski2006,guidoboni2009}. As a result, the proposed algorithm allows us to: \textit{(i)} solve in separate substeps potential nonlinearities within the systems 
of PDEs and/or ODEs; \textit{(ii)} maintain some flexibility in choosing the numerical method for the solution of each subproblem; \textit{(iii)} ensure unconditional 
stability without the need of subiterations between substeps. 
\revA{Indeed, the physical consistency of the coupling conditions at the discrete level is a major issue in multiscale numerical simulations, which also includes the coupling between 
PDEs of different types arising, for example, in the context of \textsc{3d}-\textsc{1d} modeling of blood flow~\cite{formaggia2013physical}.}

We remark that, in the present work, we aim at developing the main skeleton of the splitting algorithm, focusing on the scheme stability and performance with respect 
to the time step size in the norms dictated by the energy balance of the system. Thus, even though the mathematical framework will be presented in a general way,
here we will focus on the Stokes problem as PDE system and resistive connections between Stokes regions and lumped networks. In addition, we will present numerical
results obtained in the case of  two-dimensional (\textsc{2d}) Stokes problems coupled with zero-dimensional (\textsc{0d}) lumped circuit models. However, the basic skeleton presented in 
this
article could serve as a starting point for further extensions and improvements, including other PDE models, {\it e.g.} Navier-Stokes or porous media,
and numerical variants that provide increased accuracy in time by including suitable time-extrapolations of quantities of particular interest, in the same spirit 
as~\cite{Fernandez2013,Fernandez2015}, or by symmetrization of the splitting algorithm as discussed in~\cite{glowinski2003}. 

The paper is organized as follows. The mathematical framework is described in Section~\ref{sec:math_model} and the energy identity for the fully coupled system
is derived in Section~\ref{sec:energy_identity}. The energy-based operator splitting approach is presented in 
Section~\ref{sec:splitting}, where we also study its stability properties with respect to the choice of the time step. Section~\ref{sec:num_res} explores 
the properties of the proposed splitting algorithm by comparing analytical and numerical solutions in three particular examples. Conclusions and future 
perspectives are outlined in Section~\ref{sec:extensions}.

\section{Mathematical model}
\label{sec:math_model}

The present work focuses on the coupling between systems of PDEs, representing the motion of an incompressible viscous fluid in a bounded domain $\Omega\subset\mathbb R^d$, 
$d=2$ or 3, and systems of ODEs, representing lumped descriptions of the flow of a viscous fluid through a complex hydraulic network. In order to maintain the focus on the 
splitting strategy with respect to the coupling interface conditions, in this article we will assume that $\Omega$ is a non-deformable domain 
and that \revB{the} fluid is 
Newtonian. Extensions to 
deformable domains and non-Newtonian fluids are beyond the scope of this article, even though they are within reach, 
as outlined in Section~\ref{sec:extensions}. \\

\noindent {\bf Geometrical architecture of the coupled system.} A schematic representation of the geometrical coupling considered in this paper is provided in 
Figure~\ref{fig:full_geom}. It consists of: \textit{(i)} $L$ regions of space denoted by $\Omega_l\subset \mathbb R^d$, with  $l\in \mathcal L=\{1,\dots,L\}$ and 
$d=2$ or 3, where the fluid flow is described by the Stokes equations; \textit{(ii)} $M$ lumped hydraulic circuits denoted by $\Upsilon_m$, with 
$m\in\mathcal M = \{ 1,\dots,M\}$, where the fluid flow is described by the hydraulic analog of  Kirchoff laws of currents and voltages.
In order to clarify the geometrical setting of the coupling conditions, let us assume that the boundary of each domain $\Omega_l$, denoted by $\partial \Omega_l$, is 
the union of three portions, namely 
$
\partial \Omega_l = \Gamma_{l}\cup \Sigma_{l}\cup S_{l},
$
where different types of boundary and interface conditions are imposed. 
Specifically, Dirichlet conditions are imposed on $\Gamma_l$, Neumann conditions are imposed on $\Sigma_l$ and Stokes-circuit coupling 
conditions are imposed on $S_l$, as described below.
In particular, each region $\Omega_l$ may have $j_{\Omega_l} \geq 1$  Stokes-circuit connections, implying that each boundary 
portion $S_{l}$ may be written as 
$
S_{l} = \bigcup_{m\in\mathcal {M}_l} S_{lm},
$
with $l\in \mathcal L$. We remark that, for each $l\in\mathcal L$, the set $\mathcal M_l \subseteq \mathcal M$ identifies the circuits $\Upsilon_m$ that are connected 
to $\Omega_l$. Similarly, each circuit $\Upsilon_m$ may have $j_{\Upsilon_m}\geq 1$  Stokes-circuit connections and the set $\mathcal L_m \subseteq \mathcal L$ identifies 
the Stokes regions $\Omega_l$ that are connected to $\Upsilon_m$. 
For example, in the specific architecture depicted in 
Figure~\ref{fig:full_geom}, we have that $j_{\Omega_1}=1$ and $\mathcal{M}_1 =\{1\}$, $j_{\Omega_2}=4$ and $\mathcal{M}_2 =\{1,2,3\}$, 
$j_{\Upsilon_3}=3$ and $\mathcal{L}_3 =\{2,3\}$, $j_{\Upsilon_4}=1$ and $\mathcal{L}_4 =\{3\}$. It may also happen that the same Stokes region $\Omega_l$ 
and the same circuit $\Upsilon_m$ enjoy multiple connections, as for $\Omega_2$ and $\Upsilon_3$ in Figure~\ref{fig:full_geom}. Thus, an
additional subscript is introduced to distinguish between the various connections, so that we can write
$
S_{lm} = \bigcup_{k=1}^{j_{\Omega_l,\Upsilon_m}} S_{lm,k}, 
$
where $j_{\Omega_l,\Upsilon_m}$ is the total number of connections between $\Omega_l$ and $\Upsilon_m$. For example, for the architecture depicted in 
Figure~\ref{fig:full_geom}, we have $j_{\Omega_2,\Upsilon_3}=2$. Note that with these notations, we also have
$ j_{\Omega_l} = \sum _{m\in \mathcal{M}_l} j_{\Omega_l,\Upsilon_m}$ and $j_{\Upsilon_m} = \sum _{l\in \mathcal{L}_m} j_{\Omega_l,\Upsilon_m}$.

We remark that, for the cases considered in this study: 
\textit{(i)} Dirichlet conditions are imposed on a (at least) portion of the boundary of each domain $\Omega_l$, namely $\Gamma_l \neq \emptyset$; 
\textit{(ii)}
each domain $\Omega_l$ is connected to (at least) one circuit, namely $S_l \neq \emptyset$; and
\textit{(iii)}  Neumann conditions may not be imposed on the boundary of $\Omega_l$, namely $\Sigma_l = \emptyset$, as it happens in Example 3 described hereafter.\newline
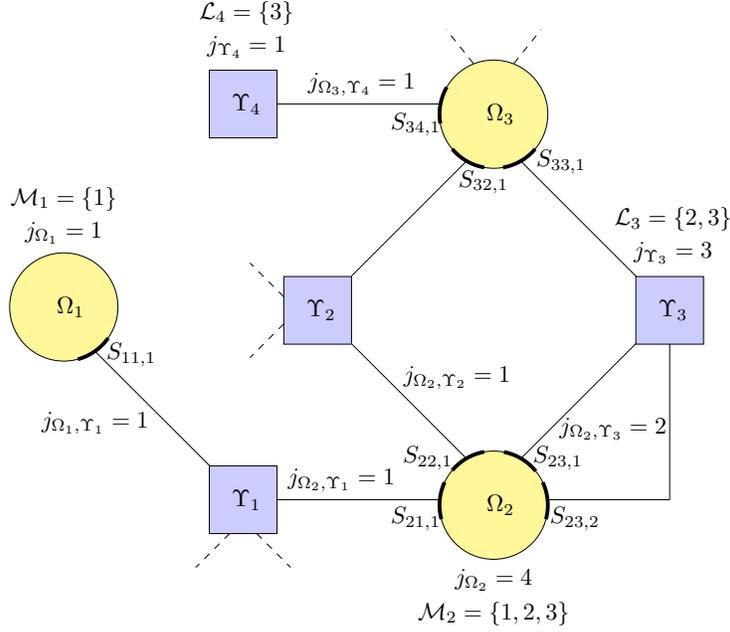
\begin{figure}[t]
\begin{center}
\scalebox{0.9}{
 \begin{tikzpicture}
\begin{scope}
\filldraw[fill=yellow!50!white, draw=black] (0,0) circle (.8cm);
 \node at (0.1,0) {$\Omega_1$};
  \draw [black,ultra thick,domain=-75:-35] plot ({.8*cos(\x)}, {.8*sin(\x)});
  \node at (1.,-0.7) {$S_{11,1}$};
  \node at (0.0,1.1) {$j_{\Omega_1}=1$};
  \node at (0.0,1.6) {$\mathcal{M}_1 =\{1\}$};  
 \end{scope}
 \begin{scope}[shift={(0.6,-1.5)},rotate = -45]
    \draw (-0.7, 0.5) -- (1.7,0.5); 
\coordinate [label=left:\textcolor{black}{$j_{\Omega_1,\Upsilon_1}=1 $}] (E) at (0.7,0.4);        
 \end{scope}  
\begin{scope}[shift={(2.15,-3.35)}]
\filldraw[fill=blue!20!white, draw=black] (0,0) rectangle (1,1);
 \node at (0.55,0.5) {$\Upsilon_1$};
 \draw[dashed] (0.3,0) -- (-0.2,-0.5);
  \draw[dashed] (0.7,0) -- (1.2,-0.5);
 \end{scope}
 \begin{scope}[shift={(3.85,-3.35)}]
    \draw (-0.7, 0.5) -- (1.7,0.5); 
      \coordinate [label=above:\textcolor{black}{$j_{\Omega_2,\Upsilon_1}=1 $}] (E) at (0.25,0.5);                  
 \end{scope} 
\begin{scope}[shift={(6.35,-2.93)}]
\filldraw[fill=yellow!50!white, draw=black] (0,0) circle (.8cm);
 \node at (0.1,0) {$\Omega_2$};
   \node at (0,-1.1) {$j_{\Omega_2}=4$};
   \node at (0,-1.6) {$\mathcal{M}_2 =\{1,2,3\}$};   
  \draw [black,ultra thick,domain=-15:25] plot ({.8*cos(\x)}, {.8*sin(\x)});
  \node at (1.2,-0.2) {$S_{23,2}$};      
   \draw [black,ultra thick,domain=39:79] plot ({.8*cos(\x)}, {.8*sin(\x)});
  \node at (0.95,0.7) {$S_{23,1}$};      
   \draw [black,ultra thick,domain=100:140] plot ({.8*cos(\x)}, {.8*sin(\x)}); 
  \node at (-0.97,0.7) {$S_{22,1}$};       
   \draw [black,ultra thick,domain=155:195] plot ({.8*cos(\x)}, {.8*sin(\x)}); 
  \node at (-1.15,-0.2) {$S_{21,1}$};         
 \end{scope}
 \begin{scope}[shift={(4.4,-1.4)},rotate = -45]
    \draw (-0.7, 0.5) -- (1.7,0.5); 
\coordinate [label=right:\textcolor{black}{$j_{\Omega_2,\Upsilon_2}=1 $}] (E) at (0.1,0.6);            
 \end{scope}
 \begin{scope}[shift={(7.6,-2.1)},rotate = 45]
    \draw (-0.7, 0.5) -- (1.7,0.5); 
 \end{scope}
\begin{scope}[shift={(3.25,-0.55)}]
\filldraw[fill=blue!20!white, draw=black](0,0) rectangle (1,1);
 \node at (0.55,0.5) {$\Upsilon_2$};
   \draw[dashed] (0,0.3) -- (-0.5,-0.2);
  \draw[dashed] (0,0.7) -- (-0.5,1.2);
 \end{scope}
\begin{scope}[shift={(8.45,-0.55)}]
\filldraw[fill=blue!20!white, draw=black] (0,0) rectangle (1,1);
 \node at (0.55,0.5) {$\Upsilon_3$};
  \node at (0.55,1.35) {$j_{\Upsilon_3}=3$};
  \node at (0.55,1.85) {$\mathcal{L}_3 =\{2,3\}$};
 \end{scope}
 \begin{scope}[shift={(5.1,0.6)},rotate = 45]
    \draw (-0.7, 0.5) -- (1.7,0.5); 
 \end{scope}
 \begin{scope}[shift={(6.9,1.3)},rotate = -45]
 \draw (-0.7, 0.5) -- (1.7,0.5);
 \end{scope}
\begin{scope}[shift={(6.35,2.85)}]
\filldraw[fill=yellow!50!white, draw=black] (0,0) circle (.8cm);
 \node at (0.1,0) {$\Omega_3$};
   \draw[dashed] (-0.3,0.75) -- (-0.7,1.25);
  \draw[dashed] (0.3,0.75) -- (0.7,1.25);
   \draw [black,ultra thick,domain=150:190] plot ({.8*cos(\x)}, {.8*sin(\x)}); 
  \node at (-1.15,-0.15) {$S_{34,1}$};   
    \draw [black,ultra thick,domain=-40:-80] plot ({.8*cos(\x)}, {.8*sin(\x)});
  \node at (1.,-0.7) {$S_{33,1}$};
    \draw [black,ultra thick,domain=-140:-100] plot ({.8*cos(\x)}, {.8*sin(\x)});
  \node at (-0.15,-1) {$S_{32,1}$};  
 \end{scope}
 \begin{scope}[shift={(7.85,-3.35)}]
 \draw (-0.7,0.5) -- (1.1,0.5); 
 \draw (1.1,0.5) -- (1.1,2.8);
 \coordinate [label=right:\textcolor{black}{$j_{\Omega_2,\Upsilon_3}=2 $}] (E) at (-0.65,1.55);                  
 \end{scope} 
 \begin{scope}[shift={(3.85,2.5)}]
 \draw (-0.7, 0.5) -- (1.7,0.5);
  \coordinate [label=above:\textcolor{black}{$j_{\Omega_3,\Upsilon_4}=1 $}] (E) at (0.55,0.5);                  
 \end{scope} 
\begin{scope}[shift={(2.15,2.5)}]
\filldraw[fill=blue!20!white, draw=black](0,0) rectangle (1,1);
 \node at (0.55,0.5) {$\Upsilon_4$};
  \node at (0.55,1.35) {$j_{\Upsilon_4}=1$};
  \node at (0.55,1.85) {$\mathcal{L}_4 =\{3\}$}; 
 \end{scope}
 \end{tikzpicture}
 }
\end{center}
\caption{Geometrical architecture of the coupled system. 
In each region of space $\Omega_l$, with $l\in \mathcal L=\{1,\dots,L\}$, the fluid flow is described by the Stokes equations. In each lumped compartment $\Upsilon_m$, 
with $m\in\mathcal M=\{1,\dots,M\}$, the fluid flow is described by the hydraulic analog of Kirchoff laws of currents and voltages. Multiple connections among Stokes domains and lumped 
circuits are allowed. Geometrical notations are reported on the graph for some components of the system.}
\label{fig:full_geom}
\end{figure}

\noindent {\bf Stokes problems in $\Omega_l$.} Let $\vector{v}_l=\vector{v}_l(\vector{x},t)$ and $p_l=p_l(\vector{x},t)$, for $l\in \mathcal L$, denote the velocity vector field and the pressure field, respectively, pertaining to the fluid in each 
domain $\Omega_l \times (0,T)$, with $\Omega_l\subset\mathbb R^d$, $d=2,3$ and $T>0$. Then, for $l\in \mathcal L$, we can write the Stokes equations as
\begin{align}
\Div \vector{v}_l& =0 & \mbox{in }\Omega_l \times (0,T), \label{eq:st_div}\\
\rho \deriv{}{\vector{v}_l}{t} &= -\nabla p_l + \mu \Delta \vector{v}_l + \rho \vector{f}_l  & \mbox{in }\Omega_l \times (0,T), \label{eq:st_lm}
\end{align}
where $\rho$ and $\mu$ are positive given constants representing the fluid density and dynamic viscosity, respectively, and $\vector{f}_l$ are given body forces per unit of mass. The system 
is equipped with the \textit{initial conditions}
\begin{align}
\vector{v}_l(\vector{x},t)&=\vector{v}_{l,0}(\vector{x}) & \mbox{ in }\Omega_l, \label{eq:st_ic}
\end{align}
and the \textit{boundary and interface conditions}
\begin{align}
&\vector{v}_l=\vector{0} & \mbox{on } \Gamma_{l}\times(0,T),  \label{eq:st_bc_wall}\\
&\Big(-p_l\matrix{I}+\mu\nabla \vector{v}_l\Big)\vector{n}_l= - \overline{p}_l\vector{n}_l & \mbox{on } \Sigma_{l}\times(0,T),  \label{eq:st_bc_in}\\
&\Big(-p_l\matrix{I}+\mu\nabla \vector{v}_l\Big)\vector{n}_{lm,k} =  \vector{g}_{lm,k} & \mbox{on } S_{lm,k}\times(0,T), \label{eq:st_ic_p}
\end{align}
where $\matrix{I}$ is the $d\times d$ identity tensor, $\vector{n}_l$ is the outward unit normal vector to $\Sigma_l$,
and $\overline{p}_l=\overline{p}_l(t)$ are given functions of time. For $m\in\mathcal M_l$ and $k=1, \ldots,j_{\Omega_l,\Upsilon_m}$, 
the vector $\vector{n}_{lm,k}$ denotes the outward unit normal vector to $S_{lm,k}$ and the functions $\vector{g}_{lm,k}$ are defined via the coupling 
conditions \eqref{eq:st_ic_p_coupling}.\\

\noindent {\bf Lumped hydraulic circuits in $\Upsilon_m$.} Let the dynamics in each lumped hydraulic circuit $\Upsilon_m$, for $m \in \mathcal M=\{1, \dots, M\}$, be 
described by the vector $\vector{y}_m = [y_{m1},y_{m2},\cdots,y_{m d_m}]^T$ of state variables satisfying the following system of (possibly nonlinear) ODEs
\begin{equation}
 \label{eq:circuit:m}
 \frac{ d\vector{y}_m}{dt} = \matrix{A}_m (\vector{y}_m,t) \vector{y}_m + \vector{r}_m(\vector{y}_m,t) 
\end{equation}
equipped with the initial conditions
\begin{equation}
 \label{eq:circuit:in_cond}
\vector{y}_m (t=0) = \vector{y}_{m,0},
\end{equation}
where  $\vector{y}_m$ and $\vector{r}_m$ are $d_m$-dimensional vector-valued functions and  $\matrix{A}_m$ is a $d_m\times d_m$ tensor. 
The tensor $\matrix{A}_m$ embodies topology and physics of the connections among the circuit nodes and the vector-valued function $\vector{r}_m$ comprises 
two contributions: \textit{(i)} sources and sinks within the circuit, including generators of current and voltage; and \textit{(ii)} connections with Stokes regions.
In this study, we will focus on lumped circuits involving resistive, capacitive and inductive elements, also known as RCL circuits. As a consequence, typical choices 
for the electrical state variables would be potential, voltage, charge, current or magnetic flux, which, in hydraulic terms, correspond to pressure, pressure difference, 
volume, volumetric flow rate or linear momentum flux. Since these state variables are characterized by different physical units, it follows that  the ODE system in \eqref{eq:circuit:m} is not homogeneous in terms of units.
\begin{remark} \label{rem:units}
The physical units of state variables and equations in system \eqref{eq:circuit:m} differ also depending on whether the circuits are 
coupled with \textsc{3d} or \textsc{2d}  Stokes regions, as summarized in Table~\ref{table:units_circuit}. In the \textsc{3d} case, namely when 
$\Omega_l\subset \mathbb R^d$ with $d=3$ for all $l\in\mathcal L$, the physical units of the state variables are kg~m$^{-1}$s$^{-2}$ for pressure and pressure difference, 
m$^3$ for volume, m$^3$s$^{-1}$ for flow rate and kg~m$^{-1}$s$^{-1}$ for linear momentum flux. Consequently, the physical units for the corresponding differential equations 
are kg~m$^{-1}$s$^{-3}$ if the state variable is a pressure or a pressure difference, m$^3$s$^{-1}$ if the state variable is a volume, m$^3$s$^{-2}$ if the state variable 
is a flow rate and kg~m$^{-1}$s$^{-2}$ if the state variable is a linear momentum flux. In the \textsc{2d} case, namely when $\Omega_l\subset \mathbb R^d$ with $d=2$ 
for all $l\in\mathcal L$, pressures and pressure differences are physical variables whose units are still kg~m$^{-1}$s$^{-2}$ and whose corresponding differential equations 
have the units of kg~m$^{-1}$s$^{-3}$. However, volumes and volumetric flow rates should be interpreted as quantities per unit of length, and therefore their units  in the \textsc{2d}  case are m$^2$ and m$^2$s$^{-1}$, respectively, and their corresponding differential equations have the units of m$^2$s$^{-1}$ and m$^2$s$^{-2}$, respectively. 
Note that the linear momentum flux is a quantity per unit of area, hence, in the \textsc{2d}  case its unit is still kg~m$^{-1}$s$^{-1}$ and the corresponding 
differential equations have the units of kg~m$^{-1}$s$^{-2}$. This remark is very important to fully understand the meaning of the coupling conditions detailed 
hereafter.\\
\end{remark}

\begin{table}
\centering
\scalebox{0.9}{
\begin{tabular}{ccccc}
  \toprule
{\it State variable}	&  \multicolumn{2}{c}{\it Units of state variable} & \multicolumn{2}{c}{\it Units of differential equation} \\
{\small (Coupling dimension)}	& $(\Omega_l\in\mathbb R^2)$ & $(\Omega_l\in\mathbb R^3)$ &  $(\Omega_l\in\mathbb R^2)$& $(\Omega_l\in\mathbb R^3)$ \\
  \midrule
pressure  & \multicolumn{2}{c}{kg m$^{-1}$s$^{-2}$} & \multicolumn{2}{c}{kg m$^{-1}$s$^{-3}$}  \\
pressure difference  & \multicolumn{2}{c}{kg m$^{-1}$s$^{-2}$} & \multicolumn{2}{c}{kg m$^{-1}$s$^{-3}$}  \\
volume  & m$^{2}$   & m$^{3}$ & m$^{2}$s$^{-1}$ & m$^{3}$s$^{-1}$\\
flow rate & m$^{2}$s$^{-1}$ & m$^{3}$s$^{-1}$ & m$^{2}$s$^{-2}$ & m$^{3}$s$^{-2}$\\
linear momentum flux & \multicolumn{2}{c}{kg m$^{-1}$s$^{-1}$} & \multicolumn{2}{c}{kg m$^{-1}$s$^{-2}$}\\
  \bottomrule
\end{tabular}
}
\caption{Physical units of state variables for the lumped hydraulic circuits and their corresponding differential equations, in the case where the circuits are 
coupled with two-dimensional or three-dimensional Stokes regions, namely $(\Omega_l\in\mathbb R^2)$ or $(\Omega_l\in\mathbb R^3)$, respectively.}\label{table:units_circuit}
\end{table}

\noindent {\bf Coupling conditions.} A domain $\Omega_l$ is connected to a lumped circuit $\Upsilon_m$ via the interfaces $S_{lm,k}$, with $k=1, \ldots,
j_{\Omega_l,\Upsilon_m}$ as indicated in Figure~\ref{fig:zoom_coupling_conditions}, where we impose the condition
\begin{equation}\label{eq:st_ic_p_coupling}
\vector{g}_{lm,k}(\vector{x},t) = - P_{lm,k}(t) \vector{n}_{lm,k} (\vector{x})  \quad \mbox{for}\quad \vector{x}\in S_{lm,k}  \quad \mbox{and} \quad t\in (0,T),
\end{equation}
where $P_{lm,k}$ is the pressure at the node of the circuit sitting on $S_{lm,k}$. Under some geometric assumptions on the domain, this condition corresponds to imposing 
that the average pressure on the interface $S_{lm,k}$ is equal to the nodal pressure, see \cite{fouchet2014,heywood1996,carichino2016thesis}. In addition, the  continuity 
of mass, and consequently flow rate, across $S_{lm,k}$ implies that 
\begin{equation} \label{eq:st_ic_q}
Q_{lm,k}(t)  = \dint_{S_{lm,k}}\vector{v}_l (\vector{x},t) \cdot \vector{n}_{lm,k} (\vector{x})dS_{lm,k} \quad \mbox{for} \quad t\in (0,T). 
\end{equation}
For $l\in \mathcal L_m$ and $k=1,\ldots,j_{\Omega_l,\Upsilon_m}$, each term $Q_{lm,k}$ contributes as source/sink for the circuit $\Upsilon_m$; thus, it is convenient 
to rewrite $\vector{r}_m$ in \eqref{eq:circuit:m} as 
\begin{equation}\label{eq:coupling_r_circuit}
 \vector{r}_m(\vector{y}_m,Q_{lm,k},P_{lm,k},t) = \vector{s}_m(\vector{y}_m,t) + \vector{b}_m(Q_{lm,k},P_{lm,k},t),
\end{equation}
where $\vector{s}_m(\vector{y}_m,t)$ represents the contribution of sources and sinks within the circuit (generators of current and voltage) and 
$\vector{b}_m(Q_{lm,k},P_{lm,k},t)$ gathers all contributions from the $j_{\Upsilon_m}$ Stokes-circuit connections. \\

\begin{figure}[htp]
\begin{center}
\scalebox{1}{
 \begin{tikzpicture}
\begin{scope}[shift={(-0.15,0)}]
\filldraw[fill=yellow!50!white, draw=black] (0,0) circle (.8cm);
 \node at (0.1,0) {$\Omega_l$};
 \draw [black,ultra thick,domain=-45:45] plot ({.8*cos(\x)}, {.8*sin(\x)});
  \node at (1.,-0.7) {$S_{lm,k}$};
 \end{scope}
\begin{scope}[shift={(-0.8,0)}]
\filldraw[fill=black] (1.65,0) circle(0.03);
\coordinate [label=above:\textcolor{black}{$P_{lm,k}$}] (E) at (1.9,0.0);  
\draw (1.65,0) -- (3.3,0);
\draw[ultra thick,->,red] (1.65,-0.2) -- (2.65,-0.2);
\coordinate [label=center:\textcolor{red}{$Q_{lm,k}$}] (E) at (2.85,-0.6);  
\end{scope}
\begin{scope}[shift={(2.5,-0.5)}]
\filldraw[fill=blue!20!white, draw=black] (0,0) rectangle (1,1);
 \node at (0.55,0.5) {$\Upsilon_m$};
 \end{scope}
 \end{tikzpicture}}
\end{center}
\caption{Schematic representation of the coupling between the Stokes region $\Omega_l$ and the lumped circuit $\Upsilon_m$. Coupling conditions for the pressure 
$P_{lm,k}$ and the flow rate $Q_{lm,k}$ should be imposed on the interface $S_{lm,k}$.}
\label{fig:zoom_coupling_conditions}
\end{figure}
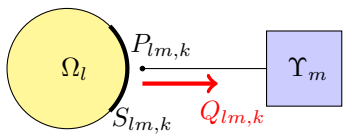
\noindent {\bf Fully coupled problem.} The fully coupled problem consists in finding $\vector{v}_l$, $p_l$, $P_{lm,k}$, 
$Q_{lm,k}$ and $\vector{y}_m$, for $l\in \mathcal{L},
\ m\in \mathcal{L}$ and $k=1, \ldots,j_{\Omega_l,\Upsilon_m}$ satisfying equations \eqref{eq:st_div}, \eqref{eq:st_lm} and \eqref{eq:circuit:m}, subject to the 
coupling conditions \eqref{eq:st_ic_p}, \eqref{eq:st_ic_p_coupling} and \eqref{eq:st_ic_q}, the boundary conditions \eqref{eq:st_bc_wall} and \eqref{eq:st_bc_in},
and the initial conditions \eqref{eq:st_ic} and \eqref{eq:circuit:in_cond}. Existence of solutions to the fully coupled problem has been proved in some 
particular cases. For example, in \cite{quarteroni2003}, Quarteroni \textit{et al} proved local in time existence of a solution when the connections are made through 
bridging regions. In \cite{baffico2010}, Baffico \textit{et al} proved the existence of a strong solution for small data when the Navier-Stokes equations are connected with a single resistor.

\section{Energy identity of the fully coupled problem} 
\label{sec:energy_identity}

The fully coupled problem satisfies an energy identity that embodies the main mechanisms governing the physics of the system. Let us begin by introducing the norms that will be 
used in the sequel. For $d=2$ or $3$, $d_z\in \mathbb{N}^*$, $T>0$ and $\Omega \subset \mathbb{R}^d$, given the vector valued function $\vector{z} \colon (0,T) \to \mathbb{R}^{d_z}$, 
the vector field $\vector{\Phi} \colon \Omega \times (0,T) \to \mathbb{R}^d$ and the tensor field $\matrix{K} \colon \Omega_l \times (0,T) \to \mathbb{R}^{d_z\times d_z}$, we introduce 
the following notations
\begin{equation}
\label{eq:norms}
\|\vector{z}\| = \sqrt{\sum_{i=1}^{d_z}z_i^2}\,,\ \ 
\| \vector{\Phi} \|_{L^2(\Omega)} = \sqrt{\dint_{\Omega} \vector{\Phi}\cdot \vector{\Phi}\, d\Omega} 
\ \  \mbox{and}\ \ 
\| \matrix{K} \|_{L^2(\Omega)} = \sqrt{\dint_{\Omega} \matrix{K}\colon \matrix{K}\, d\Omega} \, ,
\end{equation}
where the spaces $\mathbb{R}^d$ and $\mathbb{R}^{d_z\times d_z}$ are endowed with the usual Euclidean \revB{inner} products and, for the sake of clarity, the time dependence is omitted.
The notation $L^2(\Omega)$ denotes the space of square integrable real functions and the corresponding space of vector valued functions $[L^2(\Omega)]^d$ is denoted
by $\mathbf{L^2}(\Omega)$. By a slight abuse of notations, we will use in the sequel the same symbol $\|~\cdot~\|_{L^2(\Omega)}$ for the associated norms, see \eqref{eq:norms}. 
It is useful to recall that, given $\matrix{K}\in\mathbb R^{d_z\times d_z}$ symmetric and positive definite, then we can define the norm
$
\|\matrix K^{1/2}\vector{z}\| = \sqrt{\vector{z}^T\,\matrix{K}\,\vector{z}},
$
where the superscript $T$ means transpose.

Let us now proceed to derive the energy identity for the coupled system. For each $l\in\mathcal L$, let us multiply \eqref{eq:st_lm} by $\vector{v}_l$ in $\mathbf{L^2}(\Omega_l)$ and integrate over $\Omega_l$. 
After utilizing the divergence free condition \eqref{eq:st_div} and the boundary and interface conditions \eqref{eq:st_bc_wall}-\eqref{eq:st_ic_p}, we obtain:
\begin{align}
\frac{\rho}{2}\frac{d}{dt} \| {\vector{v}_l}\|^2_{L^2(\Omega_l)} 
+ \mu  \| \nabla {\vector{v}_l}\|^2_{L^2(\Omega_l)} = &
 \rho \int_{\Omega_l}\vector{f}_l \cdot \vector{v}_l\,d\Omega_l 
 - \overline{p}_l\int_{\Sigma_l} \vector{v}_l\cdot \vector{n}_l \,d\Sigma_l \\
& + \sum_{m\in\mathcal M_l} \sum_{k=1}^{j_{\Omega_l,\Upsilon_m}}\int_{S_{lm,k}} \vector{v}_l\cdot \vector{g}_{lm,k} \,dS_{lm,k} \,.\nonumber
\end{align}
The last term on the right hand side can be further manipulated using the coupling conditions \eqref{eq:st_ic_p_coupling} and \eqref{eq:st_ic_q} to write
\begin{align}
\int_{S_{lm,k}} \vector{v}_l\cdot \vector{g}_{lm,k} \,dS_{lm,k} =
- P_{lm,k}\int_{S_{lm,k}} \vector{v}_l\cdot \vector{n}_{lm,k} \,dS_{lm,k} =
- P_{lm,k} Q_{lm,k}\,.
\end{align}
Combining the above relationships, for each $l\in\mathcal L$ we obtain
\begin{align}\label{eq:en_stokes_l}
\frac{\rho}{2}\frac{d}{dt} \| {\vector{v}_l}\|^2_{L^2(\Omega_l)} 
+ \mu  \| \nabla {\vector{v}_l}\|^2_{L^2(\Omega_l)} = 
 \rho \int_{\Omega_l}\vector{f}_l \cdot \vector{v}_l\,d\Omega_l 
 - \overline{p}_l\int_{\Sigma_l} \vector{v}_l\cdot \vector{n}_l \,d\Sigma_l \\
 - \sum_{m\in\mathcal M_l} 
 \sum_{k=1}^{j_{\Omega_l,\Upsilon_m}} P_{lm,k} Q_{lm,k} \,.\nonumber
\end{align}

Let us now consider the system of differential equations in \eqref{eq:circuit:m} describing the dynamics of $\Upsilon_m$, for $m\in\mathcal M$. Since the equations might not be homogeneous in terms of physical units, see Remark~\ref{rem:units}, for each $m\in\mathcal M$ we need to perform the \revB{inner} product between \eqref{eq:circuit:m} and the vector valued function $\matrix{U}_m \vector{y}_m$, in such a way that each of the resulting scalar equations has the physical dimensions of a rate of change of energy, namely Kg~m$^{2}$ s$^{-3}$, in the case of \textsc{3d} Stokes regions, or that of a rate of change of energy per unit length, namely 
Kg~m~s$^{-3}$, in the case of \textsc{2d} Stokes regions.
Thus, the tensor $\matrix{U}_m$ is diagonal and its entries $U_{mj}$, with $j=1,\dots,d_m$, depend on the particular choice for the corresponding state variable. More precisely, for any $m\in\mathcal M$ and for $j=1,\dots,d_m$ 
\begin{itemize}
\item if $y_{mj}$ is a pressure or a pressure difference, then $U_{mj}$ is a capacitance;
\item if $y_{mj}$ is a volume, then $U_{mj}$ is the inverse of a capacitance;
\item if $y_{mj}$ is a flow rate, then $U_{mj}$ is an inductance;
\item if $y_{mj}$ is a linear momentum flux, then $U_{mj}$ is the inverse of an inductance.
\end{itemize}
The appropriate choices for $U_{mj}$ have been summarized in Table~\ref{table:units}, along with the corresponding physical units in the cases where the circuit is connected 
to \textsc{2d} or \textsc{3d} Stokes regions. We remark that the specific choice for capacitances and inductances appearing in $U_{mj}$ is determined by the corresponding 
circuit element pertaining to $y_{mj}$, as detailed in the examples hereafter. 
\begin{table}
\centering
\scalebox{0.9}{
\begin{tabular}{cccc}
 \toprule
$y_{mj}$ 			&  $U_{mj}$ & \multicolumn{2}{c} {\it physical units}   \\
& {\small (Coupling)}  & $(\Omega_l\in\mathbb R^2)$ &  $(\Omega_l\in\mathbb R^3)$  \\
 \midrule
pressure  & capacitance & kg$^{-1}$ m$^3$ s$^2$ & kg$^{-1}$ m$^4$ s$^2$\\
pressure difference  & capacitance & kg$^{-1}$ m$^3$ s$^2$ & kg$^{-1}$ m$^4$ s$^2$\\
volume &  inverse of a capacitance & kg m$^{-3}$ s$^{-2}$ & kg m$^{-4}$ s$^{-2}$\\
flow rate &  inductance & kg m$^{-3}$  & kg m$^{-4}$  \\
linear momentum flux & inverse of an inductance & kg$^{-1}$ m$^{3}$  & kg$^{-1}$ m$^{4}$  \\
\bottomrule
\end{tabular}
}
\caption{Summary of appropriate choices for the diagonal entries $U_{mj}$ of the tensor $\matrix{U}_m$ depending on the corresponding state variable $y_{mj}$. 
The physical units for $U_{mj}$ are listed for the cases where the circuit is coupled with two- or three-dimensional Stokes regions.}\label{table:units}
\end{table}
We also remark that, in general, $\matrix{U}_m = \matrix{U}_m(\vector{y}_m,t)$. Thus, performing the scalar product between \eqref{eq:circuit:m} and $\matrix{U}_m \vector{y}_m$ we obtain:
\begin{align}
\frac{1}{2}\frac{d}{dt} (\vector{y}_m^T \matrix{U}_m\vector{y}_m) + \vector{y}_m^T \matrix{B}_m \vector{y}_m = 
\vector{r}_m^T \matrix{U}_m \vector{y}_m,
\end{align}
where 
\begin{equation}
\label{eq:Bm}
\matrix{B}_m  = -\matrix{U}_m \matrix{A}_m - \frac{1}{2}\frac{d}{dt} \matrix{U}_m \,.
\end{equation}
Using the coupling conditions \eqref{eq:coupling_r_circuit} and the fact that $\matrix{U}_m$ is symmetric and positive definite, we finally obtain
\begin{align}\label{eq:en_circ_m}
\frac{1}{2}\frac{d}{dt} \| \matrix{U}_m^{1/2}\vector{y}_m\|^2 + \vector{y}_m^T \matrix{B}_m \vector{y}_m = 
\vector{b}_m^T \matrix{U}_m \vector{y}_m + \vector{s}_m^T \matrix{U}_m \vector{y}_m \,.
\end{align}
Now, summing \eqref{eq:en_stokes_l} over $l\in \mathcal{L}$, \eqref{eq:en_circ_m} over $m\in \mathcal{M}$, and adding the resulting equations, we obtain 
the following energy identity for the fully coupled system
\begin{equation}
\label{eq:full_energy}
\frac{d}{dt} \Big( \mathcal E_{\Omega} + \mathcal E_{\Upsilon}\Big) + \mathcal D_\Omega + \mathcal U_\Upsilon = \mathcal F_\Omega + \mathcal F_\Upsilon +  \mathcal G,
\end{equation}
where
\begin{align}
& \mathcal E_\Omega = \frac{1}{2} \sum_{l\in \mathcal L} \rho \| {\vector{v}_l}\|^2_{L^2(\Omega_l)}, 
\qquad \qquad \quad
\mathcal E_\Upsilon =\frac{1}{2} \sum_{m\in\mathcal M}\| \matrix{U}_m^{1/2}\vector{y}_m\|^2, \label{eq:E_OU}\\
&\mathcal D_\Omega =  \sum_{l\in \mathcal L} \mu \| \nabla {\vector{v}_l}\|^2_{L^2(\Omega_l)},
\qquad \qquad \quad
\mathcal U_\Upsilon= \sum_{m\in\mathcal M}\vector{y}_m^T \matrix{B}_m \vector{y}_m,\\
&
\mathcal F_\Omega =  \sum_{l\in \mathcal L}\left( \rho \int_{\Omega_l}\vector{f}_l \cdot \vector{v}_l\,d\Omega_l 
 - \overline{p}_l\int_{\Sigma_l} \vector{v}_l\cdot \vector{n}_l \,d\Sigma_l \right),\\
&
 \mathcal F_\Upsilon =  \sum_{m\in\mathcal M} \vector{s}_m^T \matrix{U}_m \vector{y}_m,\\
&
\mathcal G = - \sum_{\substack{m\in\mathcal M_l\\l\in\mathcal L}} \sum_{k=1}^{j_{\Omega_l,\Upsilon_m}}
  P_{lm,k}Q_{lm,k} + \sum_{m\in\mathcal M}\vector{b}_m^T \matrix{U}_m \vector{y}_m. \label{eq:energy_g}
\end{align}
$\mathcal E_\Omega$ represents the total kinetic energy in the Stokes regions, $\mathcal E_\Upsilon$
represents the total energy (kinetic + potential) characterizing the lumped circuits, $\mathcal D_\Omega $
represents the viscous dissipation in the Stokes regions, $\mathcal U_\Upsilon$
represents all the contributions from resistive, capacitive and inductive elements in the lumped circuits, 
$\mathcal F_\Omega$ represents the forcing on the system due to body forces  and external pressures acting on the Stokes regions,
$\mathcal F_\Upsilon$ represents the forcing on the system due to generators of current and voltage within the lumped circuits, and $\mathcal G$
represents the contribution from the Stokes-circuit connections.

We emphasize that $\mathcal E_\Omega(t) \geq0$, $\mathcal E_\Upsilon(t) \geq 0$ and $\mathcal D_\Omega(t)\geq0$ for all $t$, whereas the sign of $\mathcal U_\Upsilon (t)$ 
depends on the properties of the tensor $\matrix{B}_m$. The functionals $\mathcal F_\Omega$ and $\mathcal F_\Upsilon$ do not have a definite sign since they depend on the external forcing. 
Lastly, the functional form of $\mathcal G$ depends on the type of lumped elements involved in the  Stokes-circuit connections. \\

\noindent {\bf The case of resistive connections.} In order to clearly elucidate the main rationale behind the proposed splitting algorithm, in this article we will focus on resistive 
connections between Stokes regions and lumped circuits. Resistive connections are among the most common in blood flow modeling \cite{quarteroni2016}. In some applications, though, 
capacitive and inductive elements might be needed. Capacitive connections are used when the fluid pressure in $\Omega_l$ influences the fluid flow in $\Upsilon_m$ and, simultaneously, the 
fluid pressure in $\Upsilon_m$ influence the fluid flow in  $\Omega_l$, without having actual fluid flow between  $\Omega_l$ and $\Upsilon_m$. Thus, some portions 
of the boundary of $\Omega_l$ must be deformable, leading to a fluid-structure interaction problem that goes beyond the scope of this article and might be considered 
as future research direction, as outlined in Section~\ref{sec:extensions}. Inductive connections are used when the regime of interest is such that inertial effects become important. 
Since in the present paper we are neglecting inertial effects by adopting the Stokes equations in each $\Omega_l$, we consistently neglect inertial effects in the connections between 
the Stokes regions and the lumped circuits. We remark that the particular elements allowed in the connections might lead to different initial problems, as pointed out in \cite{moghadam2013}. 
We also remark that resistive, inductive and capacitive elements may all be present in the lumped circuit $\Upsilon_m$.

Let us then consider the case where a resistor connects a Stokes region $\Omega_l$ with the circuit $\Upsilon_m$, therefore allowing us to choose pressures as state variables
at both ends. For the sake of simplicity, let us also assume that the resistor node within the circuit is set to ground via a capacitor, in the same spirit as~\cite{quarteroni2003},
as shown in Figure~\ref{fig:zoom_coupling_RC}. 
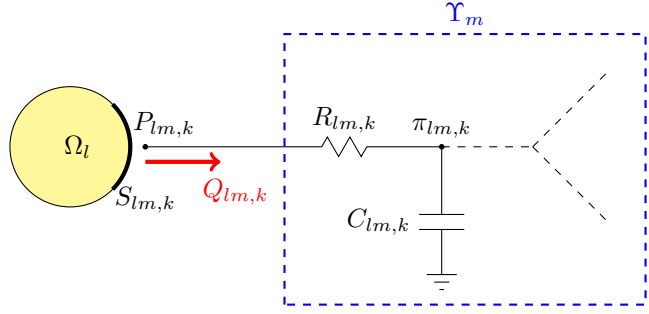
\begin{figure}[htb]
\begin{center}
\scalebox{1}{
 \begin{tikzpicture}[scale=1]
\begin{scope}[shift={(-0.15,0)}]
\filldraw[fill=yellow!50!white, draw=black] (0,0) circle (.8cm);
 \node at (0.1,0) {$\Omega_l$};
 \draw [black,ultra thick,domain=-45:45] plot ({.8*cos(\x)}, {.8*sin(\x)});
  \node at (1.,-0.7) {$S_{lm,k}$};
 \end{scope}
\begin{scope}[shift={(-0.8,0)}]
\filldraw[fill=black] (1.65,0) circle(0.03);
\coordinate [label=above:\textcolor{black}{$P_{lm,k}$}] (E) at (1.9,0.0);  
\draw (1.65,0) -- (4,0);
\draw[ultra thick,->,red] (1.65,-0.2) -- (2.65,-0.2);
\coordinate [label=center:\textcolor{red}{$Q_{lm,k}$}] (E) at (2.85,-0.6);  
\end{scope}

\begin{scope}[xscale=0.15,yscale=0.15,shift={(21.3,0)}, rotate=90]
  \draw[color=black,rotate=-90] (0,0) --(0.5,0.866025404) -- (1.5,-0.866025404) -- (2.5,0.866025404) -- (3.5,-0.866025404) -- (4,0);
\end{scope}
\draw[] (3.79,0) -- (4.79,0);
\filldraw[fill=black] (4.79,0) circle(0.03);
\coordinate [label=above:\textcolor{black}{$R_{lm,k}$}] (E) at (3.5,0.07);  
\coordinate [label=above:\textcolor{black}{$\pi_{lm,k}$}] (E) at (4.79,0.0);  

\draw[] (4.79,0) -- (4.79,-0.9);
\draw[] (4.49,-0.9) -- (5.09,-0.9);
\draw[] (4.49,-1.1) -- (5.09,-1.1);
\draw[] (4.79,-1.1) -- (4.79,-1.7);
\draw[] (4.59,-1.7) -- (4.99,-1.7);
\draw[] (4.69,-1.8) -- (4.89,-1.8);
\draw[] (4.76,-1.9) -- (4.82,-1.9);
\coordinate [label=left:\textcolor{black}{$C_{lm,k}$}] (E) at (4.5,-1.0);  

\draw[dashed] (4.79,0) -- (6,0);
\draw[dashed] (6,0) -- (7,1);
\draw[dashed] (6,0) -- (7,-1);
\coordinate [label=above:\textcolor{blue}{$\Upsilon_{m}$}] (E) at (5.1,1.5); 
\draw[dashed,thick,blue] (2.7,-2.1) rectangle (7.5,1.5);
 \end{tikzpicture}
 }
\end{center}
\caption{Schematic representation of the particular type of coupling between a Stokes region $\Omega_l$ and a circuit $\Upsilon_m$ considered in this work. A resistive connections between $\Omega_l$ and $\Upsilon_m$, in addition to the capacitive connection to the ground at the circuit side of the resistor, allows us to adopt pressures as state variables at both ends.}
\label{fig:zoom_coupling_RC}
\end{figure}
We denote the resistance and capacitance in the connection by $R_{lm,k}$ and $C_{lm,k}$, respectively, and the pressures at the Stokes and 
circuit ends by $P_{lm,k}$ and $\pi_{lm,k}$, respectively. Thus, for each $m\in \mathcal{M}$, we can rewrite the vector of state variables $\vector{y}_m$ as
$\vector{y}_m=[\vector{\pi}_m,\vector{\omega}_m]^T$, where the $j_{\Upsilon_m}$-dimensional column vector 
$\left[\pi_{lm,k} \right]^T$, with $l\in \mathcal{L}_m$ and $k=1,\ldots, j_{\Omega_l,\Upsilon_m}$,
gathers all pressures at the circuit end of the connecting
resistors, whereas the 
$(d_m - j_{\Upsilon_m})$-dimensional column vector $\left[\omega_{lm,k} \right]^T$, with $l\in \mathcal{L}_m$ and $k=1,\ldots, j_{\Omega_l,\Upsilon_m}$,
gathers the remaining state variables. Then, for each $m\in \mathcal{M}$, we can rewrite system \eqref{eq:circuit:m} as
\begin{equation}
 \label{eq:circuit:m:new}
 \frac{d}{dt} \left[
              \begin{array}{c}
              \vector{\pi}_m \\ \vector{\omega}_m
              \end{array}
              \right] 
  = \matrix{A}_m \left[
              \begin{array}{c}
              \vector{\pi}_m \\ \vector{\omega}_m
              \end{array}
              \right] 
   + \vector{s}_m  +  \vector{b}_m .          
\end{equation}
Recalling that $\vector{b}_m$ gathers the contributions due to the Stokes-circuit connections, we can write $\vector{b}_m = [ $\bm{\mathcal{b}}$_m,0]^T$,
where \bm{\mathcal{b}}$_m = \left[ \beta_{lm,k} \right]^T$, with $\beta_{lm,k} = \Frac{Q_{lm,k}}{C_{lm,k}}$
for $l\in \mathcal{L}_m$ and $k=1,\ldots, j_{\Omega_l,\Upsilon_m}$. 
Thus, denoting by $\mathcal G_{RC}$ the functional defined in equation~\eqref{eq:energy_g} in the case of the Stokes-circuit connections depicted in Figure~\ref{fig:zoom_coupling_RC},
and noticing that the entries of $\matrix{U}_m$ corresponding to $\pi_{lm,k}$ are simply given by $C_{lm,k}$, we can write 
\begin{equation}\label{eq:grc}
\mathcal G_{RC} = - \sum_{\substack{m\in\mathcal M_l\\l\in\mathcal L}} \sum_{k=1}^{j_{\Omega_l,\Upsilon_m}} P_{lm,k}Q_{lm,k}
                + \sum_{\substack{m\in\mathcal M_l\\l\in\mathcal L}} \sum_{k=1}^{j_{\Omega_l,\Upsilon_m}} \beta_{lm,k} C_{lm,k} \pi_{lm,k}.
\end{equation}     
Furthermore, since $Q_{lm,k}$ is the flow rate through the resistor of resistance $R_{lm,k}$, we can write $Q_{lm,k} =(P_{lm,k}-\pi_{lm,k})/R_{lm,k}$, thus obtaining
\begin{equation}\label{eq:drc}
 \mathcal G_{RC} = -  \sum_{\substack{m\in\mathcal M_l\\l\in\mathcal L}}\sum_{k=1}^{j_{\Omega_l,\Upsilon_m}} R_{lm,k} (Q_{lm,k})^2 = - \mathcal D_{RC}\le0 \qquad \forall t\ge0\, ,
\end{equation}
showing that $\mathcal G_{RC}$ contributes to the energy dissipation of the fully coupled system, consistently with the physics of a resistive connection. 
As a consequence, in the case of Stokes-circuit connections illustrated in Figure~\ref{fig:zoom_coupling_RC} and considered hereafter, the fully coupled system  satisfies the energy 
identity
\begin{equation}
\label{eq:full_energy_new}
\frac{d}{dt} \Big( \mathcal E_{\Omega} + \mathcal E_{\Upsilon}\Big) + \mathcal D_\Omega + \mathcal D_{RC} + \mathcal U_\Upsilon
= \mathcal F_\Omega + \mathcal F_\Upsilon\, .
\end{equation}

To better understand the mathematical formulation of the problem and the various contributions in the energy identity, particularly those coming from  $\mathcal G_{RC}$, we examine 
three illustrative examples that differ by: \textit{(i)} number of Stokes regions and lumped circuits; \textit{(ii)} number of connections between Stokes regions and lumped circuits; and \textit{(iii)}
type of elements within the lumped circuits. In all these examples, each domain $\Omega_l\in \mathbb R^d$, with $d=2$ and $l\in\mathcal L$, is defined as the rectangle $(0,L) \times
(-H/2,H/2)$, with $H,L>0$ given. 

\noindent {\bf Example 1.} 
In Example 1, see Figure~\ref{fig:ex1}, the \textsc{2d} Stokes region $\Omega_1$ is connected to the lumped circuit $\Upsilon_1$; as a consequence, here we have 
$l=m=j_{\Omega_1}=j_{\Upsilon_1}=j_{\Omega_1,\Upsilon_1}=1$ and 
$\mathcal L = \mathcal L_1 = \mathcal M =  \mathcal M_1 =\{1\}$. 
\begin{figure}[tp]
\begin{center}
\scalebox{0.8}{
\begin{tikzpicture}[scale=0.8]

\coordinate [label=above:\textcolor{black}{$\overline{p}_1$}] (E) at (0.7,-0.4);  

\begin{scope}[xscale=0.65,shift={(-1,0)}]
\filldraw[thick,yellow!50!white,draw=black] (2.5,-1.4) rectangle (10.0,1.4);
\coordinate [label=above:\textcolor{black}{$\Omega _1$}] (E) at (6.5,-0.15); 
\draw [black,ultra thick] (10.0,-1.4) -- (10.0,1.4);
\coordinate [label=above:\textcolor{black}{$\Sigma_{1}$}] (E) at (2.1,0.9);  
\coordinate [label=above:\textcolor{black}{$\Gamma_{1}$}] (E) at (6.5,1.4);  
\coordinate [label=below:\textcolor{black}{$\Gamma_{1}$}] (E) at (6.5,-1.4);  
\coordinate [label=above:\textcolor{black}{$S_{11,1}$}] (E) at (10.85,0.9);  
\coordinate [label=above:\textcolor{black}{$P_{11,1}$}] (E) at (10.85,0.0);  
\end{scope}

\draw[thick, red, ->] (6.15,-1) -- (7.15,-1);
\coordinate [label=below:\textcolor{red}{$Q_{11,1}$}] (E) at (6.65,-1);  

\filldraw[fill=black] (5.86,0) circle(0.05);
\draw[thick] (5.85,0) -- (7.5,0);
\begin{scope}[xscale=0.15,yscale=0.15,shift={(50,0)}, rotate=90]
  \draw[thick,color=black,rotate=-90] (0,0) --(0.5,0.866025404) -- (1.5,-0.866025404) -- (2.5,0.866025404) -- (3.5,-0.866025404) -- (4,0);
\end{scope}
\draw[thick] (8.1,0) -- (9.1,0);
\filldraw[fill=black] (9.1,0) circle(0.05);

\coordinate [label=below:\textcolor{black}{$R_{11,1}$}] (E) at (7.9,-0.07);  
\coordinate [label=above:\textcolor{black}{$\pi_{11,1}$}] (E) at (9.1,0.0);  

\draw[thick] (9.1,0) -- (9.1,-0.9);
\draw[thick] (8.8,-0.9) -- (9.4,-0.9);
\draw[thick] (8.8,-1.1) -- (9.4,-1.1);
\draw[thick] (9.1,-1.1) -- (9.1,-1.7);
\draw[thick] (8.9,-1.7) -- (9.3,-1.7);
\draw[thick] (9.0,-1.8) -- (9.2,-1.8);
\draw[thick] (9.07,-1.9) -- (9.13,-1.9);

\coordinate [label=left:\textcolor{black}{$C_{11,1}$}] (E) at (8.9,-1.0);  

\coordinate [label=center:\textcolor{blue}{$\Upsilon_{1}$}] (E) at (12.3,2.3); 
\draw[dashed,thick,blue] (7.3,2.0) rectangle (17.3,-2);

\draw[thick] (9.1,0) -- (11.1,0);
\begin{scope}[shift={(1.0,0)}]

\begin{scope}[xscale=0.15,yscale=0.15,shift={(67.3,0)}, rotate=90]
  \draw[thick,color=black,rotate=-90] (0,0) --(0.5,0.866025404) -- (1.5,-0.866025404) -- (2.5,0.866025404) -- (3.5,-0.866025404) -- (4,0);
\end{scope}
\draw[thick] (10.7,0) -- (12.7,0);
\draw[thick,->] (10.1,-0.2) -- (10.7,0.2);
\filldraw[fill=black] (12.7,0) circle(0.05);

\coordinate [label=below:\textcolor{black}{$R_{a}(\vector{y}_1,t)$}] (E) at (10.4,-0.07);  
\coordinate [label=above:\textcolor{black}{$\omega_{11}$}] (E) at (12.7,0.0);  

\begin{scope}[shift={(1.0,0)}]
 
\draw[thick] (11.7,0) -- (11.7,-0.9);
\draw[thick] (11.4,-0.9) -- (12.0,-0.9);
\draw[thick] (11.4,-1.1) -- (12.0,-1.1);
\draw[thick] (11.7,-1.1) -- (11.7,-1.7);
\draw[thick] (11.5,-1.7) -- (11.9,-1.7);
\draw[thick] (11.6,-1.8) -- (11.8,-1.8);
\draw[thick] (11.67,-1.9) -- (11.73,-1.9);
\draw[thick,->] (11.4,-1.3) -- (12.0,-0.7);

\coordinate [label=left:\textcolor{black}{$C_{a}(\vector{y}_1,t)$}] (E) at (11.5,-1.0);  

\end{scope}

\begin{scope}[shift={(1.5,0)}]
 
\draw[thick] (11.2,0) -- (12.2,0);
\begin{scope}[xscale=0.15,yscale=0.15,shift={(81.3,0)}, rotate=90]
  \draw[thick,color=black,rotate=-90] (0,0) --(0.5,0.866025404) -- (1.5,-0.866025404) -- (2.5,0.866025404) -- (3.5,-0.866025404) -- (4,0);
\end{scope}
\draw[thick] (12.8,0) -- (13.8,0);
\filldraw[fill=black] (13.8,0) circle(0.05);

\coordinate [label=below:\textcolor{black}{$R_{b}$}] (E) at (12.5,-0.07);  

\coordinate [label=above:\textcolor{black}{$\widetilde p$}] (E) at (13.8,0.0);  

\end{scope}

\begin{scope}[shift={(1.5,0)}]

\draw[thick] (13.8,0) -- (13.8,-0.75);
\draw[thick] (13.8,-1.0) circle (.25cm);
\draw[thick] (13.8,-1.25) -- (13.8,-1.7);
\draw[thick] (13.6,-1.7) -- (14.0,-1.7);
\draw[thick] (13.7,-1.8) -- (13.9,-1.8);
\draw[thick] (13.77,-1.9) -- (13.83,-1.9);
 \path[->,thick,every node/.style={font=\sffamily\small}]
    (14.1,-1.7) edge[bend right] node [right] {} (14.1,0);   
\end{scope}

\end{scope}

\end{tikzpicture}
}
\end{center}
\caption{\textit{Example 1.} The two-dimensional Stokes region $\Omega_1$ is connected to the lumped circuit $\Upsilon_1$ via a resistive element  with resistance $R_{11,1}$.}
\label{fig:ex1}
\end{figure}
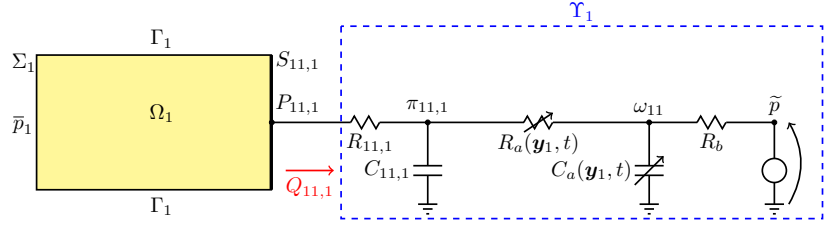
The circuit $\Upsilon_1$ is described by the vector of state variables $\vector{y}_1=[\pi_{11,1},\omega_{11}]^T$ whose dimension is $d_1=2$. 
This circuit includes a voltage generator, resistances and capacitances, some of which nonlinear, as ofter needed in blood flow modeling applications \cite{ursino1997,guidoboni2014}.
Due to the particular structure of the circuit, it is natural to choose $\pi_{11,1}$ as being the nodal pressure and $\omega_{11}$ the nodal volume.
In this example, in the ODE system~\eqref{eq:circuit:m}, we have
\begin{equation}\label{eq:test1_A}
\matrix{A}_1(\vector{y}_1,t) = \left[
\begin{array}{cc}
- \frac{1}{R_{a}(\vector{y}_1,t)C_{11,1}} & \frac{1}{R_{a}(\vector{y}_1,t)C_{11,1}C_{a}(\vector{y}_1,t)} \\[0.15in]
\frac{1}{R_{a}(\vector{y}_1,t)}  & -\frac{1}{R_{a}(\vector{y}_1,t)C_{a}(\vector{y}_1,t)}- \frac{1}{R_{b} C_{a}(\vector{y}_1,t)}
\end{array}\right],
\end{equation}
\begin{align}\label{eq:test1_r} 
\vector{r}_1(\vector{y}_1,Q_{11,1},P_{11,1},t) & = \vector{s}_1(t) + \vector{b}_1(Q_{11,1},t) \\ \nonumber
& = \left[0, \frac{\widetilde p(t)}{R_{b}} \right]^T + \left[\frac{Q_{11,1}(t)}{C_{11,1}}, 0\right]^T, 
\end{align}
with 
\begin{equation}
\label{eq:flow_rate_11}
Q_{11,1}(t)=\frac{P_{11,1}(t)-\pi_{11,1}(t)}{R_{11,1}}.
\end{equation}
Moreover, we have
\begin{equation}\label{eq:test1_U}
\matrix{U}_1(\vector{y}_1,t) = \left[\begin{array}{cc}
C_{11,1} & 0 \\
0 & \frac{1}{C_{a}(\vector{y}_1,t)}
\end{array}\right],
\end{equation}
and we \revB{observe} that, we chose $C_{11,1}$ for the entry in $U_{11}$ as it refers to the variable $\pi_{11,1}$ and, similarly, we chose $1/C_{a}$ 
for the entry $U_{12}$ as it refers to the variable $\omega_{11}$. 

Using the definition of $\matrix{U}_1$ given in~\eqref{eq:test1_U}, 
the functionals $\mathcal E_\Upsilon$ and $\mathcal F_\Upsilon$ can be written as 
\begin{equation}
  \mathcal E_\Upsilon =\frac{1}{2}\left( C_{11,1} \pi_{11,1}^2+ \frac{\omega_{11}^2}{C_{a}} \right),\qquad \qquad 
  \mathcal F_\Upsilon = \frac{\widetilde p}{R_{b}C_{a}}\omega_{11}, 
\end{equation}
respectively. The former functional represents the fluid potential energy stored in the capacitors of the lumped circuit, whereas the latter 
accounts for the forcing on the system due to the generator of voltage. Similarly, using the definition of $\matrix{B}_1$ given in~\eqref{eq:Bm}, 
the functional $\mathcal U_\Upsilon$ can be written as 
\begin{align}
\mathcal U_\Upsilon &= \frac{\pi_{11,1}^2}{R_{a}} - 2\frac{\pi_{11,1}\omega_{11}}{R_{a}C_{a}} 
+  \left( \frac{1}{R_{a}}+\frac{1}{R_{b}}+\frac{1}{2}\frac{dC_{a}}{dt}\right) \frac{\omega_{11}^2}{C_{a}^2}.
\end{align}
 Clearly, if $C_{a}$ is constant, then $\mathcal U_\Upsilon\geq0$, thereby providing energy dissipation for the system. 
Using the definition of $\vector{b}_1$ given in~\eqref{eq:test1_r} and the definition of $Q_{11,1}$ given in~\eqref{eq:flow_rate_11}, we can write
the functional $\mathcal G_{RC}$ as
\begin{equation}
\mathcal G_{RC} = - R_{11,1} \left(Q_{11,1}\right)^2 := - \mathcal D_{RC} \leq 0 ,\mbox{ for all} \quad t\in (0,T),
\end{equation}
which shows that $\mathcal G_{RC}$ contributes via $\mathcal D_{RC}$ to the mechanisms of energy dissipation in the coupled system.\\
\noindent {\bf Example 2.} 
In Example 2, see Figure~\ref{fig:ex2}, the \textsc{2d} Stokes regions $\Omega_1$ and $\Omega_2$ are connected to the lumped circuit $\Upsilon_1$; as a consequence, here we have $l=2$, $m=1$, 
$j_{\Omega_1}=j_{\Omega_2}=j_{\Omega_1,\Upsilon_1}=j_{\Omega_2,\Upsilon_1}=1$, $j_{\Upsilon_1}=2$, and $\mathcal{M} =\mathcal M_l =\{1\}$ for $l=1,2$, $\mathcal L = {\mathcal L}_1 = \{1,2\}$. 
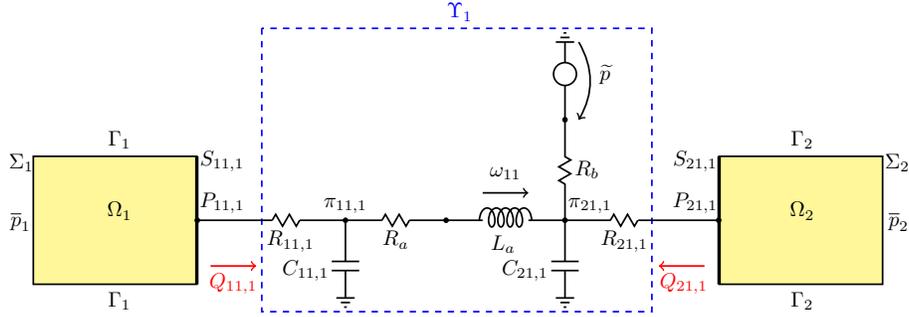
\begin{figure}[htp]
\begin{center}
\scalebox{0.8}{
\begin{tikzpicture}[scale=0.76]

\begin{scope}[shift={(-3.5,0)}]

\coordinate [label=above:\textcolor{black}{$\overline{p}_1$}] (E) at (2.0,-0.4);  

\begin{scope}[xscale=0.65,shift={(-1,0)}]
\filldraw[thick,yellow!50!white,draw=black] (4.5,-1.4) rectangle (10.0,1.4);
\coordinate [label=above:\textcolor{black}{$\Omega _1$}] (E) at (7.4,-0.15); 
\draw [black,ultra thick] (10.0,-1.4) -- (10.0,1.4);
\coordinate [label=above:\textcolor{black}{$\Sigma_{1}$}] (E) at (4.1,0.9);  
\coordinate [label=above:\textcolor{black}{$\Gamma_{1}$}] (E) at (7.4,1.4);  
\coordinate [label=below:\textcolor{black}{$\Gamma_{1}$}] (E) at (7.4,-1.4);  
\coordinate [label=above:\textcolor{black}{$S_{11,1}$}] (E) at (10.85,0.9);  
\coordinate [label=above:\textcolor{black}{$P_{11,1}$}] (E) at (10.85,0.0);  
\end{scope}

\coordinate [label=above:\textcolor{blue}{$\Upsilon_{1}$}] (E) at (11.6,4.2); 
\draw[dashed,thick,blue] (7.27,4.2) rectangle (15.8,-2);

\draw[thick, red, ->] (6.15,-1) -- (7.15,-1);
\coordinate [label=below:\textcolor{red}{$Q_{11,1}$}] (E) at (6.65,-1);  

\filldraw[fill=black] (5.86,0) circle(0.05);
\draw[thick] (5.85,0) -- (7.5,0);
\begin{scope}[xscale=0.15,yscale=0.15,shift={(50,0)}, rotate=90]
  \draw[thick,color=black,rotate=-90] (0,0) --(0.5,0.866025404) -- (1.5,-0.866025404) -- (2.5,0.866025404) -- (3.5,-0.866025404) -- (4,0);
\end{scope}
\draw[thick] (8.1,0) -- (9.1,0);
\filldraw[fill=black] (9.1,0) circle(0.05);

\coordinate [label=below:\textcolor{black}{$R_{11,1}$}] (E) at (7.9,-0.07);  
\coordinate [label=above:\textcolor{black}{$\pi_{11,1}$}] (E) at (9.1,0.0);  

\draw[thick] (9.1,0) -- (9.1,-0.9);
\draw[thick] (8.8,-0.9) -- (9.4,-0.9);
\draw[thick] (8.8,-1.1) -- (9.4,-1.1);
\draw[thick] (9.1,-1.1) -- (9.1,-1.7);
\draw[thick] (8.9,-1.7) -- (9.3,-1.7);
\draw[thick] (9.0,-1.8) -- (9.2,-1.8);
\draw[thick] (9.07,-1.9) -- (9.13,-1.9);

\coordinate [label=left:\textcolor{black}{$C_{11,1}$}] (E) at (8.9,-1.1);  

\draw[thick] (9.1,0) -- (9.9,0);
\begin{scope}[xscale=0.15,yscale=0.15,shift={(66,0)}, rotate=90]
  \draw[thick,color=black,rotate=-90] (0,0) --(0.5,0.866025404) -- (1.5,-0.866025404) -- (2.5,0.866025404) -- (3.5,-0.866025404) -- (4,0);
\end{scope}
\draw[thick] (10.5,0) -- (11.3,0);
\filldraw[fill=black] (11.3,0) circle(0.05);

\coordinate [label=below:\textcolor{black}{$R_{a}$}] (E) at (10.2,-0.07);  

\draw[thick] (11.3,0) -- (12.1,0);
\draw [] (12.05,0) to[cute inductor] (13.15,0) ; 
    \node [below=5pt] at (12.55,0.07) {$L_{a}$};
\draw[thick] (13.1,0) -- (13.9,0);
\draw[thick,->] (12.1,0.6) -- (13.1,0.6);
\node [above=3pt] at (12.6,0.6) {$\omega_{11}$};

\filldraw[fill=black] (13.9,0) circle(0.05);
 \coordinate [label=above:\textcolor{black}{$\pi_{21,1}$}] (E) at (14.45,0.0);  

\begin{scope}[shift={(2.2,0)}]
\draw[thick] (11.7,0) -- (11.7,-0.9);
\draw[thick] (11.4,-0.9) -- (12.0,-0.9);
\draw[thick] (11.4,-1.1) -- (12.0,-1.1);
\draw[thick] (11.7,-1.1) -- (11.7,-1.7);
\draw[thick] (11.5,-1.7) -- (11.9,-1.7);
\draw[thick] (11.6,-1.8) -- (11.8,-1.8);
\draw[thick] (11.67,-1.9) -- (11.73,-1.9);

\coordinate [label=left:\textcolor{black}{$C_{21,1}$}] (E) at (11.5,-1.1);  
\end{scope}

\draw[thick] (13.9,0) -- (14.9,0);
\begin{scope}[xscale=0.15,yscale=0.15,shift={(99.3,0)}, rotate=90]
  \draw[thick,color=black,rotate=-90] (0,0) --(0.5,0.866025404) -- (1.5,-0.866025404) -- (2.5,0.866025404) -- (3.5,-0.866025404) -- (4,0);
\end{scope}
\draw[thick] (15.5,0) -- (17.5,0);
\coordinate [label=below:\textcolor{black}{$R_{21,1}$}] (E) at (15.2,-0.07);  

\draw[thick] (13.9,0) -- (13.9,0.805);
\begin{scope}[xscale=0.15,yscale=0.15,shift={(92.65,9.3)}, rotate=0]
  \draw[thick,color=black,rotate=-90] (0,0) --(0.5,0.866025404) -- (1.5,-0.866025404) -- (2.5,0.866025404) -- (3.5,-0.866025404) -- (4,0);
\end{scope}
\draw[thick] (13.9,1.38) -- (13.9,2.2);
\coordinate [label=right:\textcolor{black}{$R_{b}$}] (E) at (13.95,1.1);  
\filldraw[fill=black] (13.9,2.2) circle(0.05);

\begin{scope}[shift={(13.9,2.2)},rotate=180]
\draw[thick] (0,0) -- (0,-0.75);
\draw[thick] (0,-1.0) circle (.25cm);
\draw[thick] (0,-1.25) -- (0,-1.7);
\draw[thick] (-0.2,-1.7) -- (0.2,-1.7);
\draw[thick] (-0.1,-1.8) -- (0.1,-1.8);
\draw[thick] (-0.03,-1.9) -- (0.03,-1.9);
 \path[->,thick,every node/.style={font=\sffamily\small}]
    (-0.3,-1.7) edge[bend left] node [right] {} (-0.3,0);   
\coordinate [label=right:\textcolor{black}{$\widetilde{p}$}] (E) at (-0.6,-1.0);      
\end{scope}

\end{scope}

\draw[thick, red, <-] (12.45,-1) -- (13.45,-1);
\coordinate [label=below:\textcolor{red}{$Q_{21,1}$}] (E) at (13.,-1);  

\begin{scope}[shift={(-0.40,0)}]
\begin{scope}[xscale=0.65]
\filldraw[thick,yellow!50!white,draw=black] (21.8,-1.4) rectangle (27.3,1.4);
 \coordinate [label=above:\textcolor{black}{$\Omega _2$}] (E) at (24.6,-0.15); 
 \draw [black,ultra thick] (21.8,-1.4) -- (21.8,1.4);
 \coordinate [label=above:\textcolor{black}{$ S_{21,1}$}] (E) at (21.0,0.9);  
 \coordinate [label=above:\textcolor{black}{$\Gamma_{2}$}] (E) at (24.6,1.4);  
  \coordinate [label=below:\textcolor{black}{$\Gamma_{2}$}] (E) at (24.6,-1.4);  
 \coordinate [label=above:\textcolor{black}{$\Sigma_{2}$}] (E) at (27.8,0.9);  
 \coordinate [label=above:\textcolor{black}{$P_{21,1}$}] (E) at (21.0,0.0);  
\end{scope}
  \filldraw[fill=black] (14.16,0) circle(0.05);
\coordinate [label=above:\textcolor{black}{$\overline{p}_2$}] (E) at (18.1,-0.4); 
\end{scope}

\end{tikzpicture}
}
\end{center}
\caption{\textit{Example 2.} The two-dimensional Stokes regions $\Omega_1$ and $\Omega_2$ are connected to the lumped circuit $\Upsilon_1$ via resistive elements with 
resistance $R_{11,1}$ and $R_{21,1}$, respectively. }
\label{fig:ex2}
\end{figure}
The circuit $\Upsilon_1$ is described by the vector of state variables $\vector{y}_1=[\pi_{11,1},\pi_{21,1},\omega_{11}]^T$ whose dimension is $d_1=3$, and it includes also 
an inductive element, in addition to a voltage generator, resistances and capacitances as occurring in systemic modeling of blood flow, see for instance \cite{avanzolini1988}.
Due to the particular structure of the circuit, it is natural to choose $\pi_{11,1}$ and $\pi_{21,1}$ as the nodal pressures and $\omega_{11}$ as the flow rate, respectively.
In this example, in the ODE system~\eqref{eq:circuit:m}, we have
\begin{equation}\label{eq:test2_A}
\matrix{A}_1(\vector{y}_1,t) = \left[
\begin{array}{ccc}
0  & 0 & - \frac{1}{C_{11,1}} \\[0.15in]
0 & -\frac{1}{C_{21,1}R_b} & \frac{1}{C_{21,1}}  \\[0.15in]
\frac{1}{L_{a}} & -\frac{1}{L_{a}} & -\frac{R_{a}}{L_{a}} 
\end{array}\right],
\end{equation}
\begin{align}\label{eq:test2_r}
\vector{r}_1(\vector{y}_1,t) & = \vector{s}_1(t)+ \vector{b}_1(Q_{11,1},Q_{21,1},t) \\ \nonumber
& = \left[0, \frac{\widetilde{p}(t)}{C_{21,1}R_b}, 0\right]^T + \left[\frac{Q_{11,1}(t)}{C_{11,1}},\frac{Q_{21,1}(t)}{C_{21,1}},0\right]^T,
\end{align}
where 
\begin{equation}
Q_{l1,1}(t) =\frac{P_{l1,1}(t)-\pi_{l1,1}(t)}{R_{l1,1}} \qquad l=1,2.
\label{eq:test2_flow_rate}
\end{equation}
Moreover, in this case
\begin{equation}
\matrix{U}_1 (\vector{y}_1,t) = \left[\begin{array}{ccc}
C_{11,1} & 0 & 0\\
0 & C_{21,1} & 0 \\
0 & 0 & L_{a}
\end{array}\right].
\end{equation}
Proceeding as in the case of Example 1, we can write explicitly the functionals $\mathcal E_\Upsilon$, $\mathcal U_\Upsilon$ and $\mathcal F_\Upsilon$ as
\begin{align}
&\mathcal E_\Upsilon = \frac{1}{2}\left( C_{11,1} \pi_{11,1}^2+ L_{a}\omega_{11}^2 + C_{21,1}\pi_{21,1}^2 \right),\qquad \; \\
&\mathcal U_\Upsilon = R_{a} \omega_{11}^2+ \frac{1}{R_{b}} \pi_{21,1}^2, \qquad
\mathcal F_\Upsilon = \frac{\widetilde{p}}{R_{b}}\pi_{21,1},
\end{align}
where we can identify the contributions of both potential and kinetic energy in $\mathcal E_\Upsilon$. In addition, the term $\mathcal U_\Upsilon (t)\geq 0$ for all $t$, 
thereby contributing to the overall energy dissipation. Using the definition of $\vector{b}_1$ given in~\eqref{eq:test2_r}, $Q_{11,1}$
and $Q_{21,1}$ given in~\eqref{eq:test2_flow_rate}, we conclude that
\begin{eqnarray}
\mathcal G_{RC} = - R_{11,1} \left(Q_{11,1}\right)^2 - R_{21,1}\left(Q_{21,1}\right)^2:= - \mathcal D_{RC} \leq 0 ,\mbox{ for all} \quad t\in (0,T).
\end{eqnarray}

\noindent {\bf Example 3.} 
In Example 3, see Figure~\ref{fig:ex3}, the \textsc{2d} Stokes region $\Omega_1$ is connected to the closed lumped circuit $\Upsilon_1$; as a consequence, here we have $l=m=1$, 
$j_{\Omega_1}=j_{\Upsilon_1}= j_{\Omega_1,\Upsilon_1} =2$ and $\mathcal L = \mathcal L_1 = \mathcal M =  \mathcal M_1 =\{1\}$. 
\begin{figure}[tb]
\begin{center}
\scalebox{0.75}{
\begin{tikzpicture}[scale=0.89]

\begin{scope}[xscale=0.65]
\filldraw[thick,yellow!50!white,draw=black] (2.5,-1.4) rectangle (10.0,1.4);
\coordinate [label=above:\textcolor{black}{$\Omega _1$}] (E) at (6.5,-0.15); 
\draw [black,ultra thick] (10.0,-1.4) -- (10.0,1.4);
\draw [black,ultra thick] (2.5,-1.4) -- (2.5,1.4);
\filldraw[fill=black] (10.03,0) circle(0.05);
\filldraw[fill=black] (2.47,0) circle(0.05);
\coordinate [label=above:\textcolor{black}{$S_{11,2}$}] (E) at (1.75,0.9);  
\coordinate [label=above:\textcolor{black}{$\Gamma_{1}$}] (E) at (6.5,1.4);  
\coordinate [label=below:\textcolor{black}{$\Gamma_{1}$}] (E) at (6.5,-1.4);  
\coordinate [label=above:\textcolor{black}{$S_{11,1}$}] (E) at (10.85,0.9);  
\coordinate [label=above:\textcolor{black}{$P_{11,1}$}] (E) at (10.85,0.0);  
\coordinate [label=above:\textcolor{black}{$P_{11,2}$}] (E) at (1.75,0.0);  
\end{scope}

\draw[thick, red, ->] (6.7,-1) -- (7.7,-1);
\coordinate [label=below:\textcolor{red}{$Q_{11,1}$}] (E) at (7.2,-1);  

\draw[thick] (6.5,0) -- (8.257,0);
\begin{scope}[xscale=0.15,yscale=0.15,shift={(55,0)}, rotate=90]
  \draw[thick,color=black,rotate=-90] (0,0) --(0.5,0.866025404) -- (1.5,-0.866025404) -- (2.5,0.866025404) -- (3.5,-0.866025404) -- (4,0);
\end{scope}
\draw[thick] (8.84,0) -- (9.8,0);
\filldraw[fill=black] (9.8,0) circle(0.05);

\coordinate [label=below:\textcolor{black}{$R_{11,1}$}] (E) at (8.5,-0.07);  
\coordinate [label=left:\textcolor{black}{$\pi_{11,1}$}] (E) at (9.9,0.3);  

\draw[thick] (9.8,0) -- (9.8,-0.9);
\draw[thick] (9.5,-0.9) -- (10.1,-0.9);
\draw[thick] (9.5,-1.1) -- (10.1,-1.1);
\draw[thick] (9.8,-1.1) -- (9.8,-1.7);
\draw[thick] (9.5,-1.7) -- (10.1,-1.7);
\draw[thick] (9.6,-1.8) -- (10,-1.8);
\draw[thick] (9.77,-1.9) -- (9.83,-1.9);

\coordinate [label=left:\textcolor{black}{$C_{11,1}$}] (E) at (9.6,-1.0);  

\draw[thick] (9.8,0) -- (9.8,0.805);
\begin{scope}[xscale=0.15,yscale=0.15,shift={(65.3,9.3)}, rotate=0]
  \draw[thick,color=black,rotate=-90] (0,0) --(0.5,0.866025404) -- (1.5,-0.866025404) -- (2.5,0.866025404) -- (3.5,-0.866025404) -- (4,0);
\end{scope}
\draw[thick] (9.8,1.38) -- (9.8,2.2);
\coordinate [label=right:\textcolor{black}{$R_{a}$}] (E) at (9.9,1.1);  
\filldraw[fill=black] (9.8,2.2) circle(0.05);

\begin{scope}[shift={(9.8,2.2)},rotate=180]
\draw[thick] (0,0) -- (0,-0.75);
\draw[thick] (0,-1.0) circle (.25cm);
\draw[thick] (0,-1.25) -- (0,-1.7);
\draw[thick] (-0.2,-1.7) -- (0.2,-1.7);
\draw[thick] (-0.1,-1.8) -- (0.1,-1.8);
\draw[thick] (-0.03,-1.9) -- (0.03,-1.9);
 \path[->,thick,every node/.style={font=\sffamily\small}]
    (-0.3,-1.7) edge[bend left] node [right] {} (-0.3,0);   
\coordinate [label=right:\textcolor{black}{$\widetilde{p}_a$}] (E) at (-0.6,-1.0);      
\end{scope}

\draw[thick] (9.8,0) -- (10.8,0);
\draw[thick] (10.8,0) -- (10.8,-3.5);
\draw[thick] (6.3,-3.5) -- (10.8,-3.5);

\begin{scope}[xscale=-1,xshift=-231]
\draw[thick, red, ->] (6.7,-1) -- (7.7,-1);
\coordinate [label=below:\textcolor{red}{$Q_{11,2}$}] (E) at (7.2,-1);  

\draw[thick] (6.5,0) -- (8.257,0);
\begin{scope}[xscale=0.15,yscale=0.15,shift={(55,0)}, rotate=90]
  \draw[thick,color=black,rotate=-90] (0,0) --(0.5,0.866025404) -- (1.5,-0.866025404) -- (2.5,0.866025404) -- (3.5,-0.866025404) -- (4,0);
\end{scope}
\draw[thick] (8.84,0) -- (9.8,0);
\filldraw[fill=black] (9.8,0) circle(0.05);

\coordinate [label=below:\textcolor{black}{$R_{11,2}$}] (E) at (8.5,-0.07);  
\coordinate [label=left:\textcolor{black}{$\pi_{11,2}$}] (E) at (9.7,0.3);  

\draw[thick] (9.8,0) -- (9.8,-0.9);
\draw[thick] (9.5,-0.9) -- (10.1,-0.9);
\draw[thick] (9.5,-1.1) -- (10.1,-1.1);
\draw[thick] (9.8,-1.1) -- (9.8,-1.7);
\draw[thick] (9.5,-1.7) -- (10.1,-1.7);
\draw[thick] (9.6,-1.8) -- (10,-1.8);
\draw[thick] (9.77,-1.9) -- (9.83,-1.9);

\coordinate [label=right:\textcolor{black}{$C_{11,2}$}] (E) at (9.55,-1.0);  

\draw[thick] (9.8,0) -- (9.8,0.805);
\begin{scope}[xscale=0.15,yscale=0.15,shift={(65.3,9.3)}, rotate=0]
  \draw[thick,color=black,rotate=-90] (0,0) --(0.5,0.866025404) -- (1.5,-0.866025404) -- (2.5,0.866025404) -- (3.5,-0.866025404) -- (4,0);
\end{scope}
\draw[thick] (9.8,1.38) -- (9.8,2.2);
\coordinate [label=right:\textcolor{black}{$R_{b}$}] (E) at (9.7,1.1);  
\filldraw[fill=black] (9.8,2.2) circle(0.05);

\begin{scope}[shift={(9.8,2.2)},rotate=180]
\draw[thick] (0,0) -- (0,-0.75);
\draw[thick] (0,-1.0) circle (.25cm);
\draw[thick] (0,-1.25) -- (0,-1.7);
\draw[thick] (-0.2,-1.7) -- (0.2,-1.7);
\draw[thick] (-0.1,-1.8) -- (0.1,-1.8);
\draw[thick] (-0.03,-1.9) -- (0.03,-1.9);
 \path[->,thick,every node/.style={font=\sffamily\small}]
    (-0.3,-1.7) edge[bend left] node [right] {} (-0.3,0);   
\coordinate [label=left:\textcolor{black}{$\widetilde{p}_b$}] (E) at (-0.6,-1.0);      
\end{scope}

\draw[thick] (9.8,0) -- (10.8,0);
\draw[thick] (10.8,0) -- (10.8,-3.5);
\draw[thick] (6.3,-3.5) -- (10.8,-3.5);

\end{scope}

\begin{scope}[shift={(-7.4,-3.5)}]
\draw[thick] (9.1,0) -- (9.9,0);
\begin{scope}[xscale=0.15,yscale=0.15,shift={(66,0)}, rotate=90]
  \draw[thick,color=black,rotate=-90] (0,0) --(0.5,0.866025404) -- (1.5,-0.866025404) -- (2.5,0.866025404) -- (3.5,-0.866025404) -- (4,0);
\end{scope}
\draw[thick] (10.5,0) -- (11.3,0);
\filldraw[fill=black] (11.3,0) circle(0.05);

\coordinate [label=below:\textcolor{black}{$R_{c}$}] (E) at (10.2,-0.07);  

\draw[thick] (11.3,0) -- (12.1,0);
\draw [] (12.05,0) to[cute inductor] (13.15,0) ; 
    \node [below=5pt] at (12.55,0.07) {$L_{c}$};
\draw[thick] (13.1,0) -- (13.9,0);
\draw[thick,<-] (12.1,0.6) -- (13.1,0.6);
\node [above=3pt] at (12.6,0.6) {$\omega_{11}$};

\end{scope}

\coordinate [label=above:\textcolor{blue}{$\Upsilon_{1}$}] (E) at (10,4.5); 
\draw[dashed,thick,blue] (4,-2.2) -- (8,-2.2);
\draw[dashed,thick,blue] (8,-2.2) -- (8,4.5);
\draw[dashed,thick,blue] (8,4.5) -- (12,4.5);
\draw[dashed,thick,blue] (12,4.5) -- (12,-4.5);
\draw[dashed,thick,blue] (4,-4.5) -- (12,-4.5);
\begin{scope}[xscale=-1,xshift=-230]
\draw[dashed,thick,blue] (3.5,-2.2) -- (8,-2.2);
\draw[dashed,thick,blue] (8,-2.2) -- (8,4.5);
\draw[dashed,thick,blue] (8,4.5) -- (12,4.5);
\draw[dashed,thick,blue] (12,4.5) -- (12,-4.5);
\draw[dashed,thick,blue] (3.5,-4.5) -- (12,-4.5);
\end{scope}

\end{tikzpicture}
}
\end{center}
\caption{\textit{Example 3.} The two-dimensional Stokes region $\Omega_1$ is connected to the lumped circuit $\Upsilon_1$ via two resistive elements with 
resistance $R_{11,1}$ and $R_{11,2}$, respectively.}
\label{fig:ex3}
\end{figure}
The circuit $\Upsilon_1$ is described by the vector of state variables $\vector{y}_1=[\pi_{11,1},\pi_{11,2},\omega_{11}]^T$ whose dimension is $d_1=3$. 
In addition to including resistive, capacitive and inductive elements, the main feature here is that the circuit $\Upsilon_1$ is closed, as the 
circulatory system. Due to the particular structure of the circuit, it is natural to choose $\pi_{11,1}$ and $\pi_{11,2}$ as being the nodal pressures 
and $\omega_{11}$ as being the flow rate, respectively.
In this example, in the ODE system~\eqref{eq:circuit:m}, we have
\begin{equation}\label{eq:test3_A}
\matrix{A}_1(\vector{y}_1,t) = \left[
\begin{array}{ccc}
- \frac{1}{R_{a}C_{11,1}} & 0 & -\frac{1}{C_{11,1}} \\[0.15in]
0 & - \frac{1}{R_{b}C_{11,2}} & \frac{1}{C_{11,2}} \\[0.15in]
\frac{1}{L_{c}}  & -\frac{1}{L_{c}} & - \frac{R_{c}}{L_{c}}
\end{array}\right],
\end{equation}
\begin{align}\label{eq:test3_r}
& \vector{r}_1(\vector{y}_1,Q_{11,1},P_{11,1},Q_{11,2},P_{11,2},t) = \vector{s}_1(t) + \vector{b}_1(Q_{11,1},Q_{11,2},t) \\ \nonumber
& = \left[ \frac{\widetilde p_a(t)}{R_{a}C_{11,1}},\frac{\widetilde p_b(t)}{R_{b}C_{11,2}},0 \right]^T +
    \left[\frac{Q_{11,1}(t)}{C_{11,1}},\frac{Q_{11,2}(t)}{C_{11,2}},0 \right]^T,
\end{align}
where
\begin{equation}
Q_{11,k}(t)=\frac{P_{11,k}(t)-\pi_{11,k}(t)}{R_{11,k}} \qquad k=1,2.
\label{eq:test3_flow_rate}
\end{equation}
Moreover, the tensor $\matrix{U}_1$ is given by
\begin{equation}
\matrix{U}_1 (\vector{y}_1,t) = \left[\begin{array}{ccc}
C_{11,1} & 0 & 0\\
0 & C_{11,2} & 0 \\
0 & 0 & L_{c}
\end{array}\right].
\end{equation}

The functionals $\mathcal E_\Upsilon$, $\mathcal U_\Upsilon$ and $\mathcal F_\Upsilon$ are given by
\begin{align}
& \mathcal E_\Upsilon = \frac{1}{2}\left(C_{11,1} \pi_{11,1}^2+ C_{11,2}\pi_{11,2}^2 + L_{c}\omega_{11}^2\right), \qquad \; \\ 
& \mathcal U_\Upsilon = \frac{1}{R_{a}} \pi_{11,1}^2+ \frac{1}{R_{b}} \pi_{11,2}^2 + R_{c}\omega_{11}^2, \qquad
\mathcal F_\Upsilon = \frac{\widetilde{p}_a}{R_{a}}\pi_{11,1} + \frac{\widetilde{p}_b}{R_{b}}\pi_{11,2}.
\end{align}
We can identify the contributions of both potential and kinetic energy in $\mathcal E_\Upsilon$. Moreover, it follows that $\mathcal U_\Upsilon (t)\geq 0$ for all $t$, 
thereby contributing to the overall energy dissipation. Using the definition of $\vector{b}_1$ given in~\eqref{eq:test3_r}, $Q_{11,1}$
and $Q_{11,2}$ given in~\eqref{eq:test3_flow_rate}, we conclude that
\begin{eqnarray}
\mathcal G_{RC} = - R_{11,1} (Q_{11,1})^2 - R_{11,2} (Q_{11,2})^2:= - \mathcal D_{RC} \leq 0 ,\mbox{ for all} \ t\in (0,T).
\end{eqnarray}

\section{Energy-based operator splitting approach}
\label{sec:splitting}

The above examples showed that resistive Stokes-circuit connections contribute to the energy dissipation of the fully coupled system, namely $\mathcal G_{RC}=-\mathcal D_{RC}$, with $\mathcal D_{RC}(t) \geq 0 $ for all $t$.
Thus, the energy identity \eqref{eq:full_energy_new} holds.

In particular, if all forcing terms are zero, namely $\mathcal F_\Omega(t) = \mathcal F_\Upsilon(t) =0$ for $t\geq0$, and the circuit properties are such that $\mathcal U_\Upsilon(t)\geq 0$ for $t\geq0$, then from \eqref{eq:full_energy_new} we obtain that $\mathcal E = \mathcal E_\Omega+\mathcal E_\Upsilon$ is a decreasing function of time, namely
\begin{equation}\label{eq:norm_bound_coupled}
\mathcal E(t) \leq \mathcal E(0) \quad \forall t\geq0.
\end{equation}
This essential mathematical and physical property must be preserved at the discrete level, and this provides the main rationale for our splitting scheme. 
We begin by adopting an operator splitting technique, see {\it e.g.} \cite[Chap. II]{glowinski2003}, to perform  a semi-discretization in time to 
solve sequentially in separate substeps the PDE systems associated with the Stokes regions and the ODE systems associated with the lumped hydraulic circuits. 
The most important feature of our scheme is that the substeps are designed so that the energy at the semi-discrete level mirrors the behavior of the energy of the fully coupled system, 
thereby providing unconditional stability to the proposed splitting method via an upper bound in the norms of the solution similar to that provided by~\eqref{eq:norm_bound_coupled}. 
The version of the method detailed below yields, at most, a first-order accuracy in time, since it includes only two substeps; however, the scheme can be generalized to attain 
second-order accuracy using symmetrization techniques \cite[Chap. VI]{glowinski2003}.\\

\noindent{\bf First-order splitting algorithm.}
For the sake of simplicity, let $\Delta t$ denote a fixed time step, let $t^n = n \Delta t$ and let \revB{$\varphi^n = \varphi(t^n)$} for any general expression $\varphi$. 
Let $\vector{v}_l^0=\vector{v}_{l,0}$ for all $l\in \mathcal L$ and $\vector{y}_m^0=\vector{y}_{m,0}$ for all $m\in \mathcal M$. 
Then, for any $n\geq0$ solve
\begin{description}
\item[Step 1] For each $l\in \mathcal L$, $m\in\mathcal M$ and $k=1,\ldots,j_{\Omega_l,\Upsilon_m}$, given $\vector{v}_l^{n}$ and $\vector{y}_m^{n}$,
find $\vector{v}_l$ and $\vector{y}_m$ such that
\begin{align}
\Div \vector{v}_l& =0 & \mbox{in }\Omega_l \times (t^n,t^{n+1}),  \label{eq:step1_div}\\
\rho \deriv{}{\vector{v}_l}{t} &= -\nabla p_l + \mu \Delta \vector{v}_l + \rho \vector{f}_l  & \mbox{in }\Omega_l \times (t^n,t^{n+1}), \label{eq:step1_lm}\\
\frac{d\vector{y}_m}{dt} &=  \vector{b}_m (Q_{lm,k},P_{lm,k},t) & \mbox{in } (t^n,t^{n+1}),  \label{eq:step1_circ}
\end{align}
with the initial conditions
\begin{align}
\vector{v}_l(\vector{x},t^n)&=\vector{v}_l^{n}(\vector{x}) & \mbox{in }\Omega_l,\\
\vector{y}_m(t^n)&=\vector{y}_m^{n},
\end{align}
and the boundary conditions
\begin{align}
&\vector{v}_l=\vector{0} & \mbox{on } \Gamma_{l}\times (t^n,t^{n+1}),   \label{eq:step1_bc_wall}\\
&\Big(-p_l\matrix{I}+\mu\nabla \vector{v}_l\Big)\vector{n}_l= - \overline{p}_l\vector{n}_l & \mbox{on } \Sigma_{l}\times (t^n,t^{n+1}),  \label{eq:step1_bc_in}\\
&\Big(-p_l\matrix{I}+\mu\nabla \vector{v}_l\Big)\vector{n}_{lm,k}= - P_{lm,k}\vector{n}_{lm,k} & \mbox{on } S_{lm,k}\times (t^n,t^{n+1}),   \label{eq:step1_ic_p}
\end{align}
with
\begin{equation}\label{eq:step1_q}
\displaystyle\int_{S_{lm,k}}\vector{v}_l(\vector{x},t) \cdot \vector{n}_{lm,k}(\vector{x},t)\, dS_{lm,k} = Q_{lm,k}(t) \quad \mbox{in } (t^n,t^{n+1}),
\end{equation}
and then set
\begin{equation}\label{eq:end_Step1}
\vector{v}_l^{n + \frac{1}{2}} = \vector{v}_l(\vector{x},t^{n+1}), \; p_l^{n+1}=p_l(\vector{x},t^{n+1}) \; \mbox{and}
\; \vector{y}_m^{n + \frac{1}{2}} = \vector{y}_m(t^{n+1}).
\end{equation}
\item[Step 2] For each $l\in \mathcal L$, $m\in\mathcal M$ and $k=1,\ldots,j_{\Omega_l,\Upsilon_m}$, given $\vector{v}_l^{n+\frac{1}{2}}$ and $\vector{y}_m^{n+\frac{1}{2}}$,
find $\vector{v}_l$ and $\vector{y}_m$ such that
\begin{align}
\rho \deriv{}{\vector{v}_l}{t} &= \vector{0} & \mbox{in }\Omega_l \times  (t^n,t^{n+1}),  \label{eq:step2_vel}\\
\frac{d\vector{y}_m}{dt} &= \matrix{A}_m(\vector{y}_m,t) \,\vector{y}_m + \vector{s}_m(\vector{y}_m,t)  & \mbox{in }(t^n,t^{n+1}),
\end{align}
with the initial conditions
\begin{align}
\vector{v}_l(\vector{x},t^n)&=\vector{v}_l^{n + \frac{1}{2}}(\vector{x}) &  \mbox{in }\Omega_l, \\
\vector{y}_m(t^n)&=\vector{y}_m^{n + \frac{1}{2}}, &
\end{align}
and set
\begin{equation}\label{eq:split_end}
\vector{v}_l^{n + 1} = \vector{v}_l(\vector{x},t^{n+1}) \quad \mbox{and} 
\quad \vector{y}_m^{n + 1} = \vector{y}_m(t^{n+1}).
\end{equation}
\end{description}
\begin{remark}
In the particular case of resistive connections described in Section~\ref{sec:energy_identity} we have that~\eqref{eq:step1_circ} reduces to 
$\Deriv{}{\pi_{lm,k}}{t} = \Frac{Q_{lm,k}}{C_{lm,k}} $ and $\Deriv{}{\vector{\omega}_{m}}{t} =\vector{0}$,
 for $ l\in \mathcal L$, $m\in\mathcal{M}_l$, $k=1,\ldots,j_{\Omega_l,\Upsilon_m}$.
\end{remark}
Some of the main features of the proposed algorithm are the following: 
\begin{enumerate}
\item The solution at time $t^{n+1}$ is obtained from the solution at time $t^n$ after solving sequentially Step 1 and Step 2, without the need of subiterations between the two steps. The scheme is stable even without subiterations because each substep satisfies an energy identity similar to that of the fully coupled problem, as discussed later in this Section.
\item Pressure and flow rate at the coupling interfaces, namely $P_{lm,k}$ and $Q_{lm,k}$, are solved for simultaneously, thereby implicitly, even though the algorithm is partitioned. This is not the case in other splitting schemes, see for example Quarteroni \textit{et al}~\cite{quarteroni2001}, where pressures computed from the lumped circuits are used as inputs for the Stokes problem and flow rates computed in the Stokes problem are used as inputs for the lumped circuits.
\item Steps 1 and 2 communicate via the initial conditions. In particular, the state variables $\vector{y}_m$, with $m\in\mathcal M$, are updated in Step 1 and their value at $t^{n+1}$ provides the initial conditions for Step 2. 
\item Steps 1 and 2 are defined on the discrete time interval $(t^n,t^{n+1})$, but the differential operators have yet to be fully discretized in time and space. Even though specific choices will have to be made for the time and space discretization of Steps 1 and 2, the overall splitting scheme described above is independent of these choices, keeping in mind that first-order convergence in time can be achieved only if all the substeps are solved with numerical methods that are at least first-order in time.
\item \revA{By treating a subset of the ODE systems jointly with the Stokes problems in Step 1, see Equations~\eqref{eq:step1_ic_p} and \eqref{eq:step1_q}, the interface conditions can be dealt with implicitly and ensure proper energy balance at the multiscale interfaces. On the other hand, by treating the main part of the ODE systems in Step 2, the splitting algorithm provides the flexibility to: \textit{(i)} modify only one step in the scheme should the modeling application require the integration of new ODE models; and \textit{(ii)} adopt a numerical scheme that best addresses potential nonlinearities affecting the ODE system.}
\item \revA{Equation~\eqref{eq:step2_vel} effectively means that the velocity vector fields $\vector{v}_l$, for $l\in \mathcal L$, are not updated in Step 2, thereby implying that $\vector{v}_l^{n + 1} = \vector{v}_l^{n + 1/2}$.
}
\end{enumerate}

\noindent{\bf Stability analysis for the first-order splitting algorithm.}
The stability analysis will be performed on a simplified problem that preserves the main difficulties associated with the PDE/ODE coupling  considered in this work.
\begin{theorem}\label{th:uncond_stab}
Consider the Stokes-circuit coupled problem described in Section~\ref{sec:math_model}, in the case of resistive connections considered in Section~\ref{sec:energy_identity}. 
Under the assumptions that
\begin{itemize}
\item[(i)] the circuit properties are such that the tensor $\matrix{A}_m$ is constant and $\matrix{B}_m$ is positive definite;
\item[(ii)] the are no forcing terms, namely $\mathcal F_\Omega(t) = \mathcal F_\Upsilon(t) =0$ for all $t$;
\end{itemize}
the first-order splitting algorithm~\eqref{eq:step1_div}-\eqref{eq:split_end} is unconditionally stable.
\end{theorem}
\begin{proof}
Let $\Delta t=t^{n+1}-t^{n}$ and let us begin by considering Step 1. Under assumptions \textit{(i)-(ii)}, using an implicit Euler scheme for the time-discretization and following a similar procedure to that detailed in Section~\ref{sec:energy_identity}, we can obtain the following energy identity at the time-discrete level for Step 1
\begin{equation}
\label{eq:energy_Step1_timedisc}
\begin{split}
\Frac{1}{\Delta t} \mathcal E^{n+\frac{1}{2}}_{I}&+ \mathcal D_{\Omega,I}^{n+\frac{1}{2}} +\mathcal D_{RC,I}^{n+\frac{1}{2}} = \\
&\Frac{1}{\Delta t}\left( \sum_{l\in \mathcal L}   \rho \int_{\Omega_l} \vector{v}_l^{n} \cdot \vector{v}_l^{n + \frac{1}{2}}  \,d\Omega_l  +  \sum_{m\in\mathcal M} \left(\vector{y}^{n}_m\right)^T \matrix{U}_m \vector{y}^{n + \frac{1}{2}}_m \right),
\end{split}
\end{equation}
where $ \mathcal E^{n+\frac{1}{2}}_{I} =  \mathcal E^{n+\frac{1}{2}}_{\Omega,I} +  \mathcal E^{n+\frac{1}{2}}_{\Upsilon,I}$ and
\begin{align}
 \mathcal E^{n+\frac{1}{2}}_{\Omega,I} &= \sum_{l\in \mathcal L}\rho\| {\vector{v}_l^{n + \frac{1}{2}}}\|^2_{L^2(\Omega_l)}, 
&  \mathcal E^{n+\frac{1}{2}}_{\Upsilon,I}&= \sum_{m\in\mathcal M} \|\matrix{U}_m^{1/2} \vector{y}^{n + \frac{1}{2}}_m \|^2, \\
\mathcal D_{\Omega,I}^{n+\frac{1}{2}} &= \sum_{l\in \mathcal L} \mu \| \nabla {\vector{v}_l^{n + \frac{1}{2}}}\|^2_{L^2(\Omega_l)},
& \mathcal D_{RC,I}^{n+\frac{1}{2}} &= \sum_{\substack{m\in\mathcal M_l\\l\in\mathcal L}}\sum_{k=1}^{j_{\Omega_l,\Upsilon_m}} R_{lm,k} \left(Q_{lm,k}^{n+\frac{1}{2}}\right)^2.
\end{align}
Using Young's inequality on the first term on the right hand side of~\eqref{eq:energy_Step1_timedisc}, yields
\begin{equation}\label{eq:estimate_energy2D_step1}
\begin{aligned}
 \sum_{l\in \mathcal L}  \rho \int_{\Omega_l} \vector{v}_l^{n} \cdot \vector{v}_l^{n + \frac{1}{2}}  \,d\Omega_l &\leq  \sum_{l\in \mathcal L}  \Frac{\rho}{2} \left( \| {\vector{v}^{n} _l}\|^2_{L^2(\Omega_l)}  + \| {\vector{v}^{n+\frac{1}{2}} _l}\|^2_{L^2(\Omega_l)}\right)\\
 &=\Frac{1}{2} \mathcal E^{n}_{\Omega,I} + \Frac{1}{2} \mathcal E^{n+\frac{1}{2}}_{\Omega,I}.
 \end{aligned}
\end{equation}
To estimate the second term on the right hand side of~\eqref{eq:energy_Step1_timedisc}, we first use Cauchy-Schwarz inequality and then  Young's inequality to obtain
\begin{equation}\label{eq:estimate_energy0D_step1}
\begin{aligned}
\sum_{m\in\mathcal M} \left(\vector{y}^{n}_m\right)^T \matrix{U}_m \vector{y}^{n + \frac{1}{2}}_m  &\leq \sum_{m\in\mathcal M} \left( \left(\vector{y}^{n}_m\right)^T \matrix{U}_m \vector{y}^{n}_m \right)  \left( \left(\vector{y}^{n+\frac{1}{2}}_m\right)^T \matrix{U}_m \vector{y}^{n+\frac{1}{2}}_m \right)\\
& \leq \Frac{1}{2}\sum_{m\in\mathcal M} \left ( \left(\vector{y}^{n}_m\right)^T \matrix{U}_m \vector{y}^{n}_m + \left(\vector{y}^{n+\frac{1}{2}}_m\right)^T \matrix{U}_m \vector{y}^{n + \frac{1}{2}}_m \right)\\
 &=\Frac{1}{2} \mathcal E^{n}_{\Upsilon,I} + \Frac{1}{2} \mathcal E^{n+\frac{1}{2}}_{\Upsilon,I}.
\end{aligned}
\end{equation}
Note that, Eq.\eqref{eq:estimate_energy0D_step1} holds since $\matrix{U}_m$ is symmetric and positive definite. 
Combining~\eqref{eq:energy_Step1_timedisc},~\eqref{eq:estimate_energy2D_step1} and~\eqref{eq:estimate_energy0D_step1}, we obtain
\begin{equation}
\frac{1}{2 \Delta t}\mathcal E_I^{n+\frac{1}{2}} + \mathcal D_{\Omega,I}^{n+\frac{1}{2}} + \mathcal D_{RC,I}^{n+\frac{1}{2}}\leq  \frac{1}{2 \Delta t}\mathcal E_I^{n},
\end{equation}
for which it follows that 
$
\mathcal E_I^{n+\frac{1}{2}} - \mathcal E_I^{n} \leq -2 \Delta t \left(\mathcal D_{\Omega,I}^{n+\frac{1}{2}} + \mathcal D_{RC,I}^{n+\frac{1}{2}} \right) \leq 0,
$
and, finally, we can conclude that
\begin{equation}\label{eq:disc_eneregy_indeq_step1}
\mathcal E_I^{n+\frac{1}{2}} \leq \mathcal E_I^{n}.
\end{equation}
Let us now consider the energy identity for Step 2.
Using an implicit Euler scheme for the time-discretization, we can obtain the following energy identity at the time-discrete level for Step 2
\begin{equation}
\label{eq:energy_Step2_timedisc}
\begin{split}
\Frac{1}{\Delta t} \mathcal E^{n+1}_{II}&+ \mathcal U_{\Upsilon,II}^{n+1}= \\
& \Frac{1}{\Delta t} \left( \sum_{l\in \mathcal L} \rho\int_{\Omega_l} \vector{v}_l^{n+\frac{1}{2}} \cdot \vector{v}_l^{n + 1}  \,d\Omega_l  +  \sum_{m\in\mathcal M} \left(\vector{y}^{n+\frac{1}{2}}_m\right)^T \matrix{U}_m \vector{y}^{n + 1}_m \right),
\end{split}
\end{equation}
where $ \mathcal E^{n+1}_{II} =  \mathcal E^{n+1}_{\Omega,II} +  \mathcal E^{n+1}_{\Upsilon,II}$ and
\begin{align}
\mathcal E^{n+1}_{\Omega,II} &= \sum_{l\in \mathcal L}\rho\| {\vector{v}_l^{n + 1}}\|^2_{L^2(\Omega_l)}, 
& \mathcal E^{n+1}_{\Upsilon,II}&= \sum_{m\in\mathcal M} \|\matrix{U}_m^{1/2} \vector{y}^{n +1}_m \|^2, \\
\mathcal U_{\Upsilon,II}^{n+1}&=\sum_{m\in\mathcal M}\left( \vector{y}^{n+1}_m\right)^T \matrix{B}_m \vector{y}^{n+1}_m.
\end{align}
Now, to estimate the right hand side of~\eqref{eq:energy_Step2_timedisc}, we use the same procedure utilized for Step 1 in~\eqref{eq:estimate_energy2D_step1}-\eqref{eq:estimate_energy0D_step1}, and we obtain the following inequality 
\begin{equation}\label{eq:energy_estimate_Step2}
\begin{aligned}
\Frac{1}{2\Delta t} \mathcal E^{n+1}_{II} + \mathcal U_{\Upsilon,II}^{n+1} &\leq  \Frac{1}{2\Delta t} \left(\sum_{l\in \mathcal L} \rho \| \vector{v}_l^{n + \frac{1}{2}}\|^2_{L^2(\Omega_l)} +  \sum_{m\in\mathcal M} \left(\vector{y}^{n +  \frac{1}{2}}_m\right)^T \matrix{U}_m \vector{y}^{n +  \frac{1}{2}}_m \right)\\
& = \Frac{1}{2\Delta t}  \mathcal E^{n+\frac{1}{2}}_{II}.
\end{aligned}
\end{equation}
We remark that assumption \textit{(i)} guarantees that $\mathcal U_{\Upsilon,II}^{n+1}\geq0$.
Thus~\eqref{eq:energy_estimate_Step2} implies 
$
\mathcal E_{II}^{n+1} -  \mathcal E_{II}^{n+\frac{1}{2}} \leq - 2 \Delta t \ \mathcal U_{\Upsilon,II}^{n+1} \leq0,
$
which leads to
\begin{equation}\label{eq:disc_eneregy_indeq_step2}
\mathcal E_{II}^{n+1} \leq \mathcal E_{II}^{n+\frac{1}{2}}.
\end{equation}
Thanks to the fact that the initial conditions for Step 2 coincide with the final solution of Step 1 as stated in~\eqref{eq:end_Step1}, it follows that $\mathcal E_{II}^{n+\frac{1}{2}} = \mathcal E_I^{n+\frac{1}{2}}$. Thus, combining \eqref{eq:disc_eneregy_indeq_step1} and \eqref{eq:disc_eneregy_indeq_step2}, we obtain the following inequality
\begin{equation}\label{eq:final_energy_thm}
\mathcal E_{II}^{n+1} \leq \mathcal E_{II}^{n+\frac{1}{2}} = \mathcal E_I^{n+\frac{1}{2}} \leq \mathcal E_I^{n} \quad \mbox{for }n\geq1,
\end{equation}
which provides an upper bound for the norm of the solution regardless of the time step size, thereby ensuring unconditional stability of the algorithm.
\end{proof}
\begin{remark}
Unconditional stability is a direct consequence of treating implicitly in Step1
the contributions from the $j_{\Upsilon_m}$ Stokes-circuit connections, represented by $\vector{b}_m(Q_{l,m},P_{l,m},t)$, see~\eqref{eq:step1_circ}.
This splitting choice allows us to evaluate at the same time instant  all the quantities in $\mathcal{G}_{RC}$~\eqref{eq:grc}, ensuring that  $\mathcal{G}_{RC}$ can be expressed as a dissipation $\mathcal D_{RC}~$\eqref{eq:drc} even at the discrete level. As a consequence, the proposed splitting algorithm does not introduce uncontrolled artificial terms in the energy~\cite{bertoglio2013,fouchet2015}, ensuring numerical stability with respect to the time step.
\end{remark}

\section{Numerical results}
\label{sec:num_res}
 
In this section we evaluate the performance of the proposed splitting method by comparing the numerical solutions
of the three illustrative examples constructed
in Section~\ref{sec:energy_identity} with their exact solutions reported in the Appendix. 
In particular, we assess the convergence properties of the method for different choices of time step  and we show that the expected  first-order convergence in time is actually achieved. 
Table~\ref{tab:param_test123} lists the parameter values utilized in each example, whose order of magnitudes as similar to those arising in blood flow applications and are reported in the cgs system for ease of reference. 
A common feature of the exact solutions for the three examples considered here is their periodicity in time with given period $\tau = 2\pi/\omega$. The global time step $\Delta t$ is determined by  the number of intervals in each time period, denoted $N_\tau$,
according to the formula $\Delta t = \tau/N_\tau$. Next, each substep is discretized with different time steps, denoted $\Delta t_1$ and $\Delta t_2$, respectively,
using an implicit Euler scheme.
In order to check that the computed numerical solution is periodic of period $\tau$,  we introduce the superscript $per$ to denote the computed solution over each time interval $\big((n-1)\tau,n\tau\big)$, for $n\ge 1$
, so that, for a general expression $\varphi$, we define its discretization $\varphi^{[per]}$ over one period of time   as
\begin{equation}\label{eq:phi:per}
\varphi^{[per]} = \left[
                                 \varphi^{(per-1)*N_\tau},\;
                                 \varphi^{(per-1)*N_\tau + 1},\;
                                 \dots ,\;
                                 \varphi^{per*N_\tau}  
                                 \right]^T \quad \mbox{ for } per\geq 1.
\end{equation}
Finally, we use the following criterion:
\begin{equation} 
\max_{\substack{m\in\mathcal M\\l\in\mathcal L}} \left\{
\begin{array}{c}
\displaystyle
\frac{\|\vector{v}_l^{[per]} - \vector{v}_l^{[per-1]}\|^2_{L^2(\Omega_l)}}{\| \vector{v}_l^{[per-1]}\|^2_{L^2(\Omega_l)}}, \;
\frac{\|p_l^{[per]} - p_l^{[per-1]}\|^2_{L^2(\Omega_l)}}{\| p_l^{[per-1]}\|^2_{L^2(\Omega_l)}}, \;  \\[0.15in]
\displaystyle\frac{\|\vector{y}_m^{[per]} - \vector{y}_m^{[per-1]}\|^2}{\| \vector{y}_m^{[per-1]}\|^2}
 \end{array}
 \right\}
 < \varepsilon _{per},
\end{equation}
to identify the numerical quantities to be compared with the exact solution over one time period. The results reported in this section have been obtained for $\varepsilon_{per} =10^{-6}$ .
Once numerical periodicity is reached, results are extracted and the following normalized errors \cite[Sec. 15.8]{quarteroni2010} are computed
   \begin{align}
   & Err_{\vector{v}} = \left ( \Delta t \sum _{n=(per-1)*N_\tau} ^ {per*N_\tau} \sum _{l\in \mathcal{L}}
   \frac{\| \vector{v}_l ^ n  - \vector{v}_{l,ex} (t^n) \| _{L^2(\Omega_l)} ^2}{\| \vector{v}_{l,ex} (t^n) \| _{L^2(\Omega_l)} ^2} \right)^{1/2}, \label{eq:err_comput_v}\\
   & Err_{p} = \left ( \Delta t \sum _{n=(per-1)*N_\tau} ^ {per*N_\tau} \sum _{l\in \mathcal{L}}
   \frac{\| p_l ^ n  - p_{l,ex} (t^n) \| _{L^2(\Omega_l)} ^2}{\| p_{l,ex} (t^n) \| _{L^2(\Omega_l)} ^2} \right)^{1/2},\\
   & Err_{\vector{y}} = \left ( \Delta t \sum _{n=(per-1)*N_\tau} ^ {per*N_\tau} \sum _{m\in \mathcal{M}}
   \frac{\| \matrix{U}_m^{1/2}(t^n)\vector{y}_m ^ n  -  \matrix{U}_{m,ex}^{1/2}(t^n)\vector{y}_{m,ex} (t^n) \| ^2}
   {\| \matrix{U}_{m,ex}^{1/2}(t^n)\vector{y}_{m,ex} (t^n)\|^2} \right)^{1/2},      \label{eq:err_comput_y}       
   \end{align}
  where the subscript $ex$ denotes the exact solution. We remark that the norms appearing in the definitions for the errors are dictated by the norms in the energy functionals \eqref{eq:E_OU} associated with  the Stokes domains  and the circuits, 
  respectively, since the energy estimates provide the main rationale behind  the design of the numerical algorithm.
  
Numerical results are obtained in the case $d=2$, since the aim of the present work is to set the basis of a methodological approach for the time-discretization
of geometrical multiscale fluid flow models. However, the method applies also to the case $d=3$ in a straightforward manner. The spatial discretization of the Stokes
problem is handled via a triangular uniform mesh of $4000$ elements for each domain $\Omega_l$ and an inf-sup stable finite element pair, namely  (Taylor-Hood) 
$\mathbb P^2 / \mathbb P^1$ elements \cite[Chap. 15]{quarteroni2010}. 
The comparison between numerical approximations and exact solutions is performed over one time period, once periodicity has been reached, for three different global time steps, namely $\Delta t= 0.01, 0.005, 0.001$. Results are obtained with the time subsets $\Delta t _1= \Delta t$ and $\Delta t_2 = \Delta t/s$, with $s=5$ for Example 1 and $s=10$ for Examples 2 and 3.
The computational framework relies on the finite element library Freefem \cite{hecht2012}.

{\bf Example 1.}
The results presented here refer to the case of  nonlinear variable resistance $R_{a}$ and capacitance $C_{a}$, whose analytical expression is reported in the Appendix. We also tested the method in the case of constant $R_{a}$ and $C_{a}$ against the exact solution and obtained similar performances as in the nonlinear case. 
Figure~\ref{fig:num_sol_test1} displays a comparison between the exact solution and different numerical approximations of physical quantities of 
interest, namely pressure $P_{11,1}$ and flow rate $Q_{11,1}$ at the Stokes-circuit interface $S_{11,1}$ (upper panel) and pressure $\pi_{11,1}$ and volume $\omega_{11}$ in the \textsc{0d} circuit (lower panel). We remark that the nonlinearities in $R_{a}$ and $C_{a}$ have been treated explicitly, which means that $R_{a}$ and $C_{a}$ are evaluated at the previous time step.
\begin{figure}[htb]
\centering
\scalebox{0.88}{
  \begin{tikzpicture}[xscale=1]
  \coordinate [label=above:\textcolor{black}{\bf Interface quantities}] (E) at (3.7,2.2);  
\node(label) at (0,0) {\includegraphics[width=0.54\linewidth]{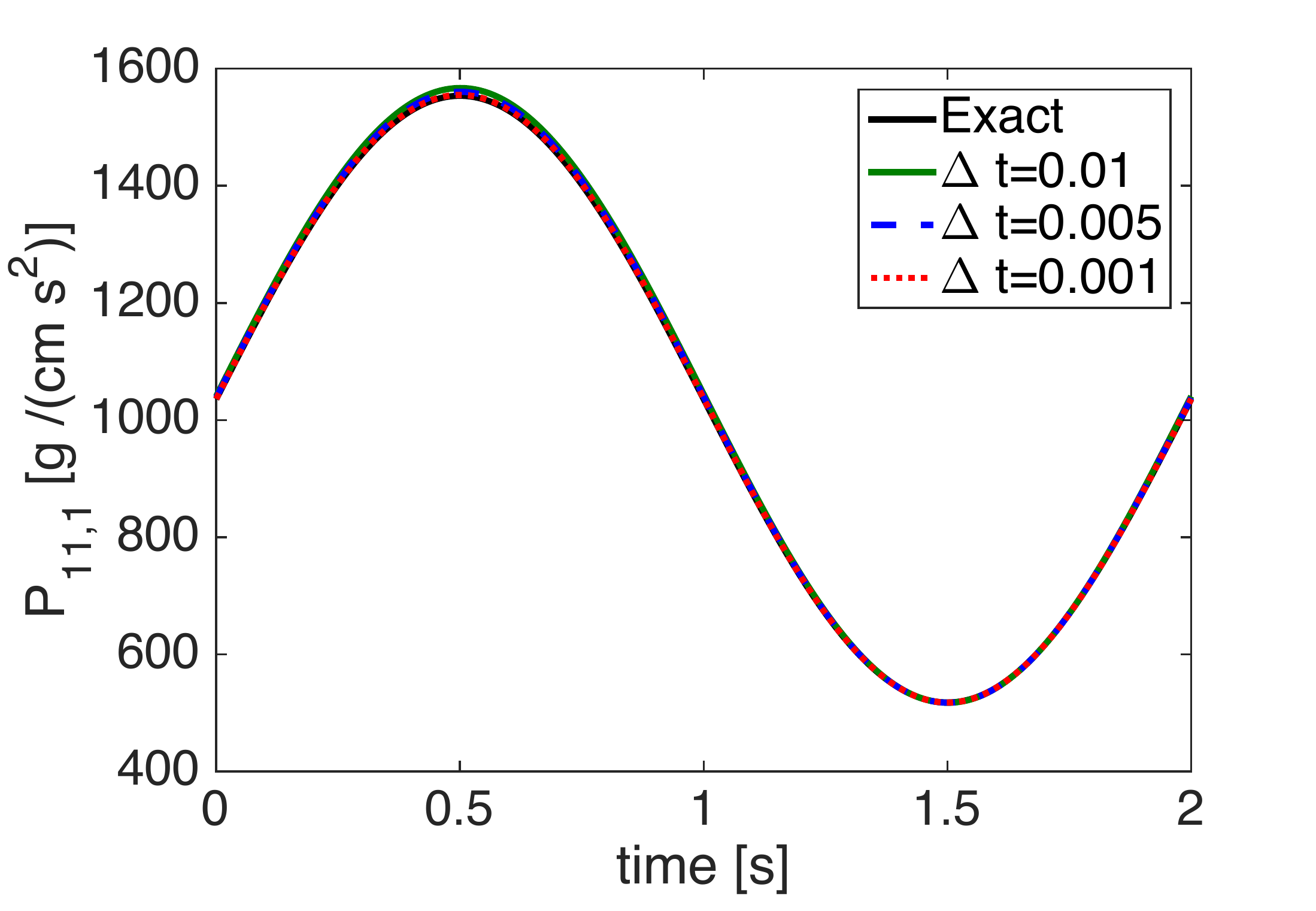}};
\node(label) at (-1.05,-0.8) {\includegraphics[width=0.16\linewidth]{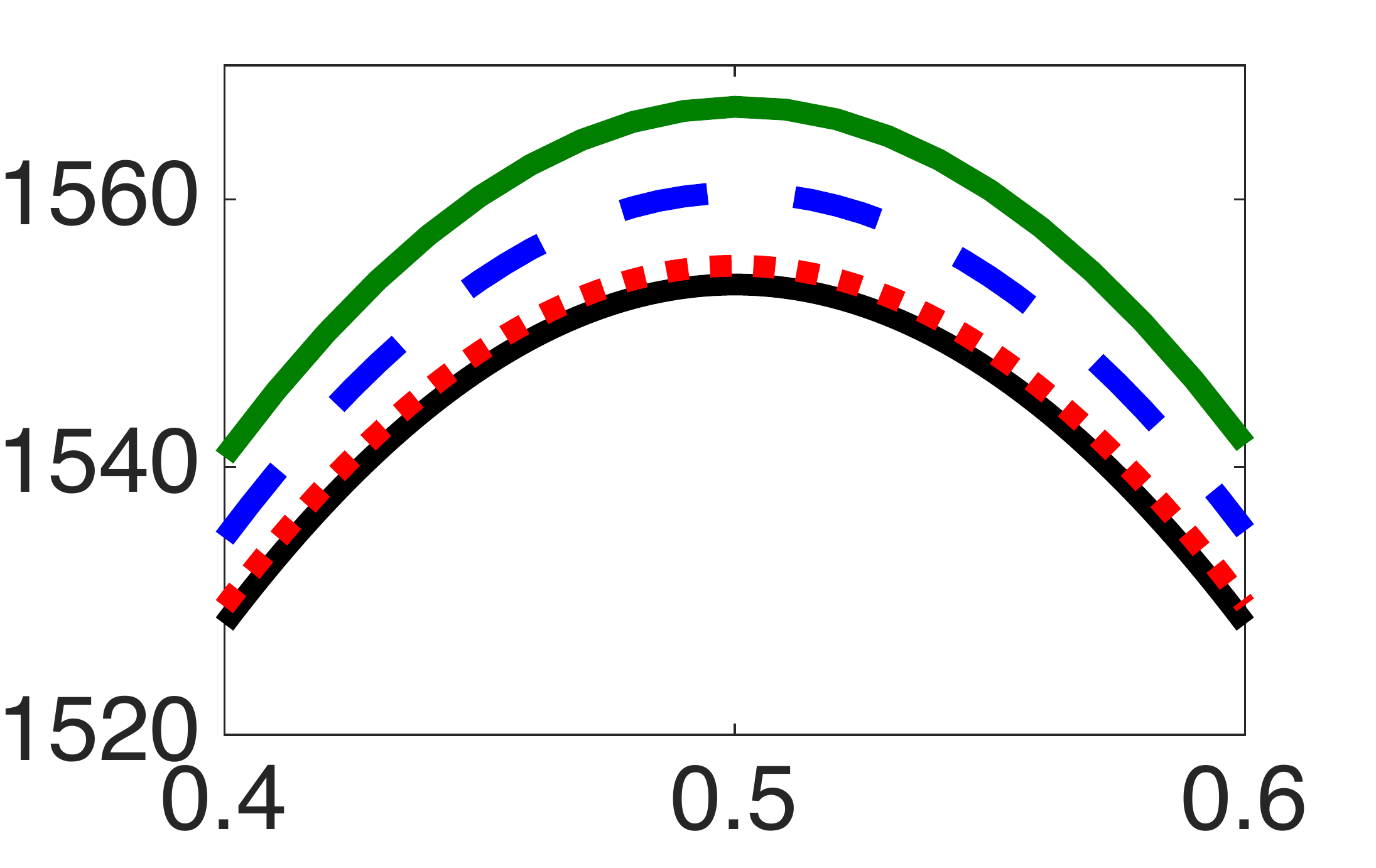}};
\begin{scope}[shift={(-1.25,1.55)},yscale=1.1]
\draw[draw=gray] (0,0) rectangle (0.5,0.4);
 \draw[gray] (0,0) -- (-0.45,-1.663);
  \draw[gray] (0.5,0) -- (0.98,-1.663);
\end{scope}
\node(label) at (7,0) {\includegraphics[width=0.54\linewidth]{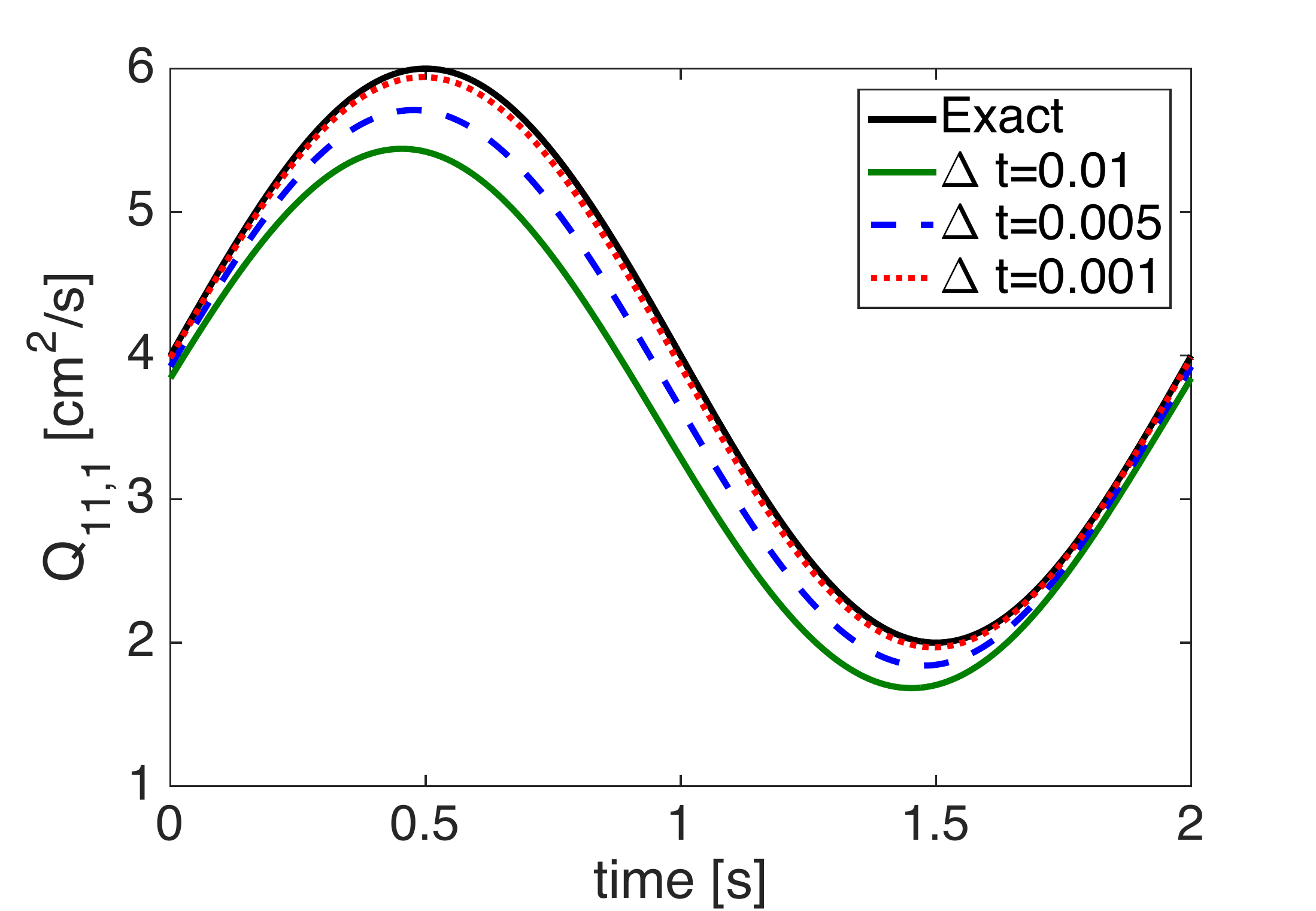}};
\begin{scope}[shift={(0,0.)}]
\coordinate [label=above:\textcolor{black}{\bf \textsc{0d} unknowns}] (E) at (3.7,-2.8);  
\node(label) at (0,-5) {\includegraphics[width=0.54\linewidth]{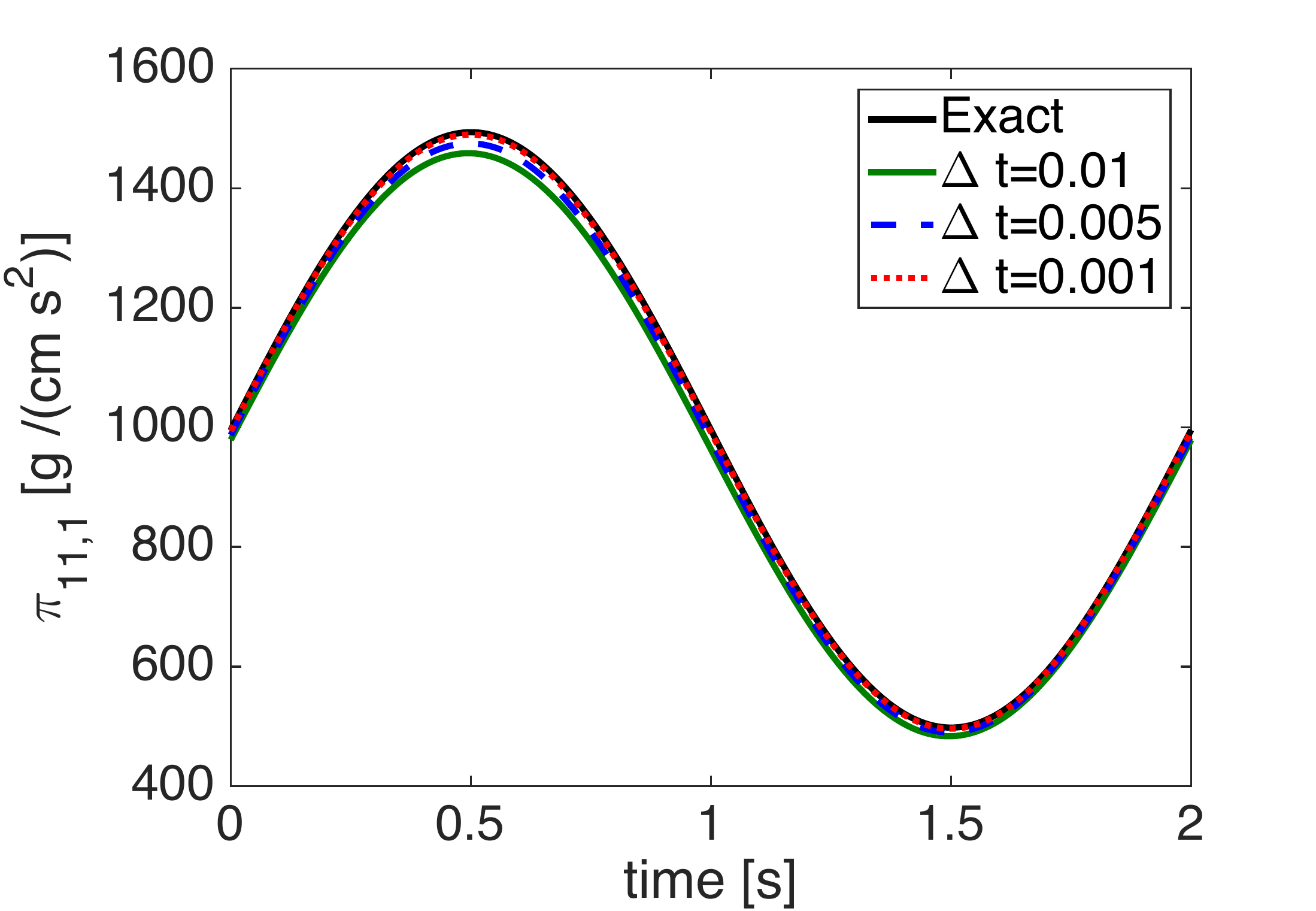}};
\node(label) at (-1.0,-5.8) {\includegraphics[width=0.16\linewidth]{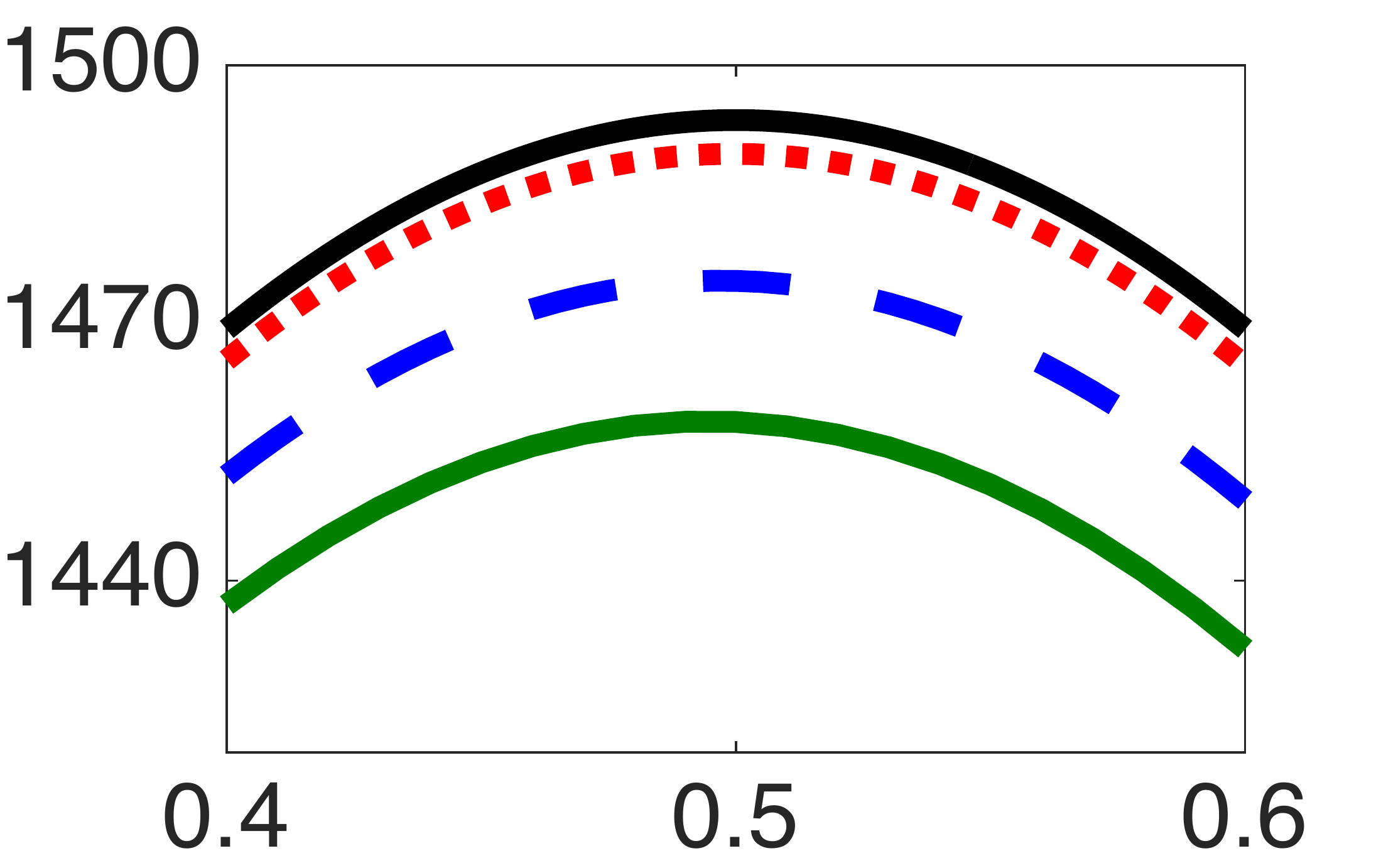}};
\begin{scope}[shift={(-1.2,-3.7)},yscale=1.1]
\draw[draw=gray] (0,0) rectangle (0.5,0.4);
 \draw[gray] (0,0) -- (-0.45,-1.44);
  \draw[gray] (0.5,0) -- (0.98,-1.44);
\end{scope}
\node(label) at (7,-5) {\includegraphics[width=0.54\linewidth]{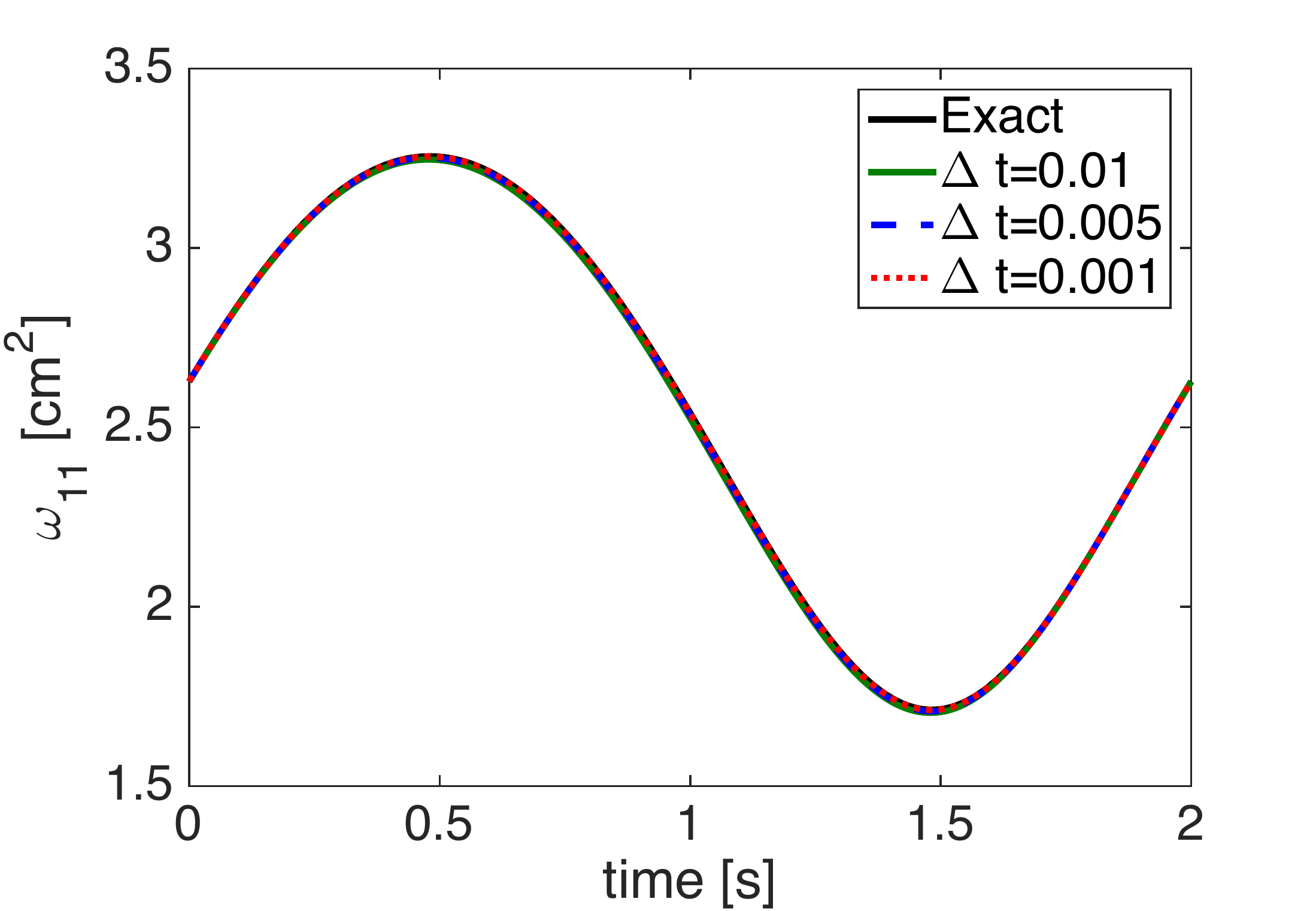}};
\node(label) at (5.87,-5.8) {\includegraphics[width=0.16\linewidth]{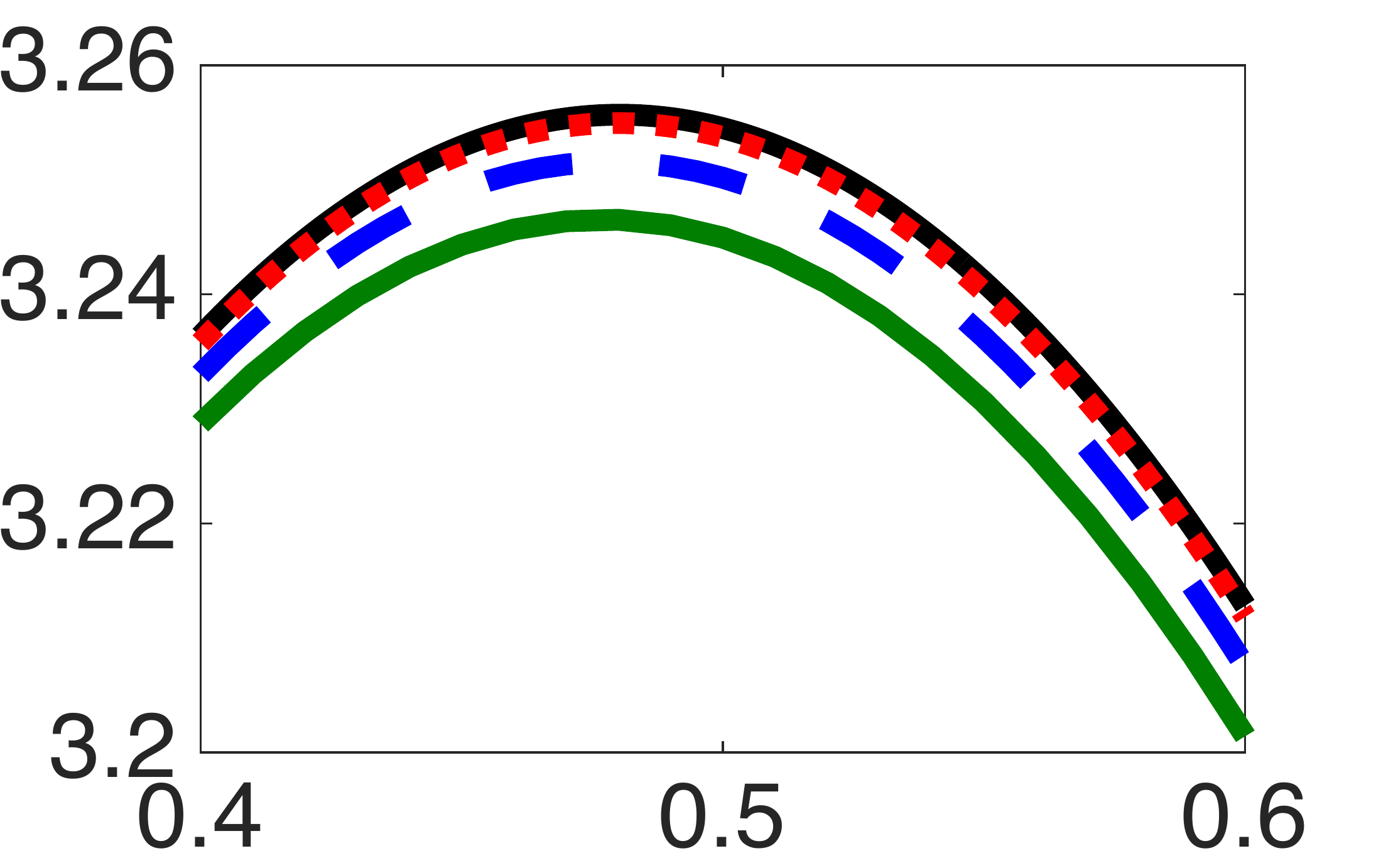}};
\begin{scope}[shift={(5.65,-3.72)},yscale=1.1]
\draw[draw=gray] (0,0) rectangle (0.5,0.4);
 \draw[gray] (0,0) -- (-0.47,-1.42);
  \draw[gray] (0.5,0) -- (1.01,-1.42);
\end{scope}
\end{scope}
\end{tikzpicture}}
    \caption{\textit{Example 1.} Comparison between the exact solution and the corresponding numerical approximation for 
    interface quantities and \textsc{0d} unknowns, for three time steps $\Delta t=0.01,\ 0.005,\ 0.001$, over one time period once the numerical  periodicity has been reached.}
  \label{fig:num_sol_test1}
\end{figure}
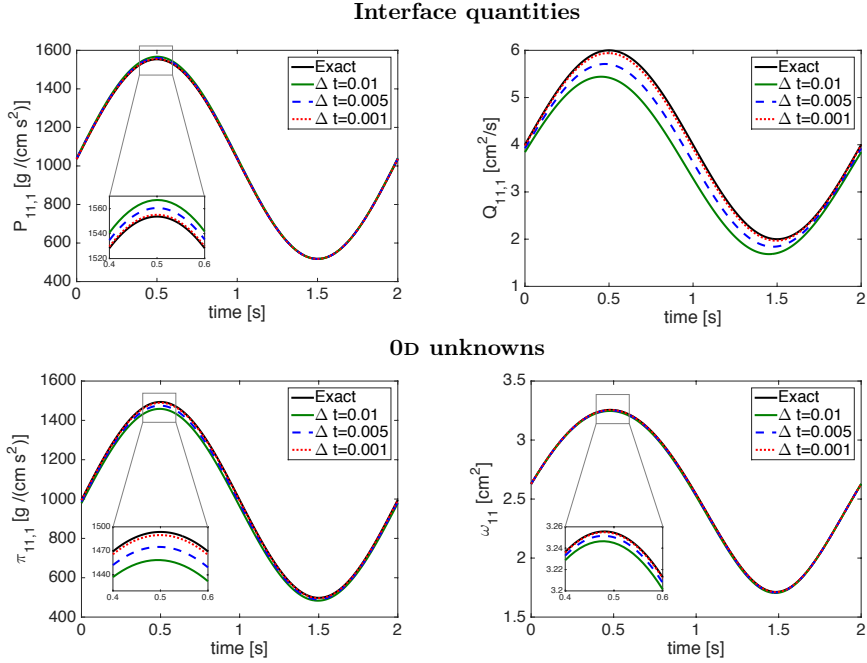
A good agreement is obtained for $P_{11,1}$,  $\pi_{11,1}$ and $\omega_{11}$  for $\Delta t=0.01$, and the error decreases as $\Delta t$ is reduced. 
The numerical approximation of $Q_{11,1}$ captures the periodicity of the solution even for $\Delta t=0.01$, but the peaks are lower than those exhibited by the exact solution. As $\Delta t$ is the reduced, the computed $Q_{11,1}$ approaches the exact peaks of the interface flow rate, thereby capturing the full dynamics of the problem.
Figure~\ref{fig:num_sol_test1} also shows that the numerical solution is not affected by spurious oscillations or instabilities, even for the largest time step. These findings confirm that the choice of the time step affects the accuracy of the computed solution but not the stability of the numerical scheme, thereby supporting the unconditional stability result proved in Theorem~\ref{th:uncond_stab}.

{\bf Example 2.}
Figure~\ref{fig:num_result_ex2} displays a comparison between the exact solution and
 pressures and flow rates at the Stokes-circuit interfaces, namely  $P_{11,1}$ and $Q_{11,1}$ at the interface $S_{11,1}$ and $P_{21,1}$ and $Q_{21,1}$ at the interface $S_{21,1}$ (upper panel) and the state variables $\pi_{11,1}$, $\pi_{21,1}$ and $\omega_{11}$ pertaining to the \textsc{0d} circuit $\Upsilon_1$ (lower panel).
Similarly to Example 1, the comparison shows very good agreement in the pressures $P_{11,1}$, $P_{21,1}$, $\pi_{11,1}$ and $\pi_{21,1}$, and a satisfactory approximation in the flow rates $Q_{11,1}$, $Q_{21,1}$ and $\omega_{11}$, especially when the time step is small. 
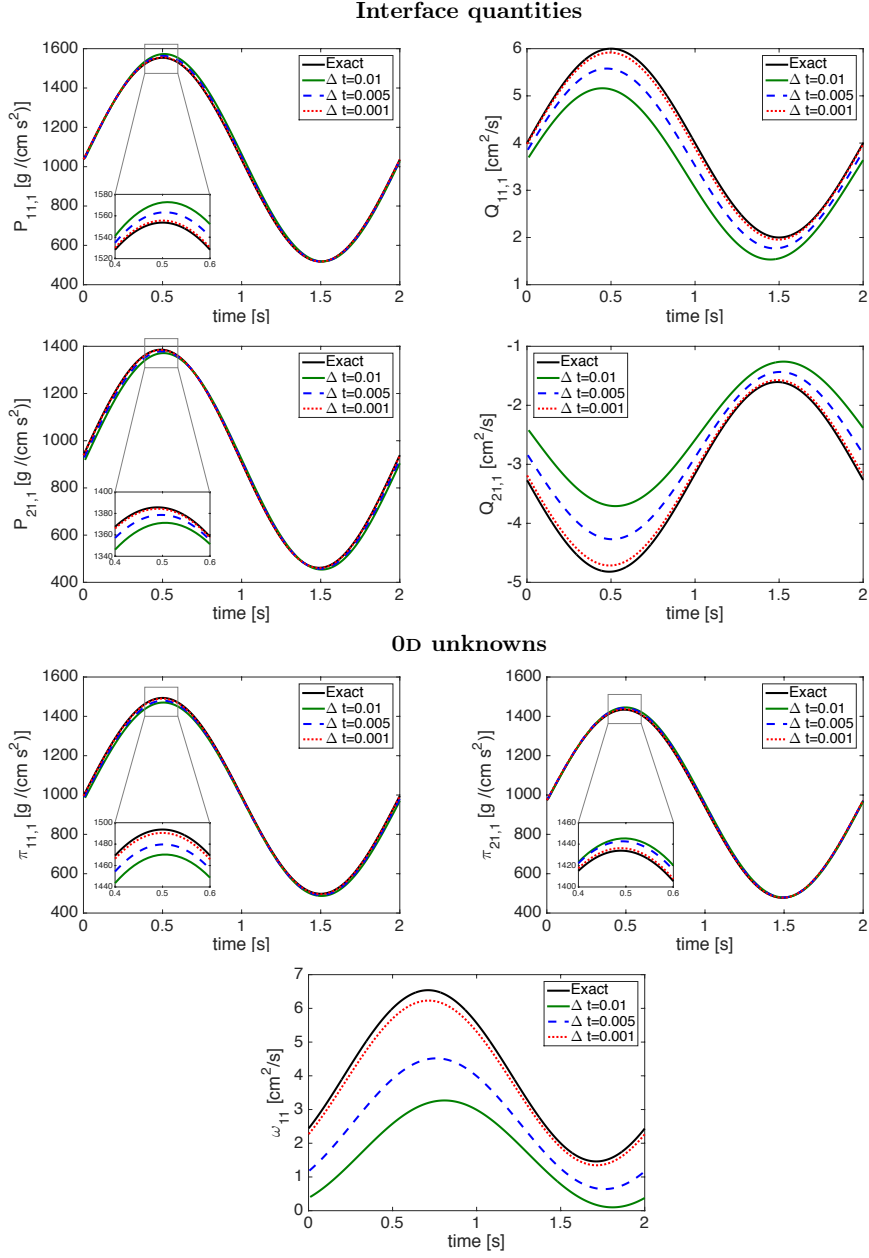
\begin{figure}[htbp]
\centering
\scalebox{0.88}{
  \begin{tikzpicture}[xscale=1]
  \coordinate [label=above:\textcolor{black}{\bf Interface quantities}] (E) at (3.7,2.2);  
\node(label) at (0,0) {\includegraphics[width=0.54\linewidth]{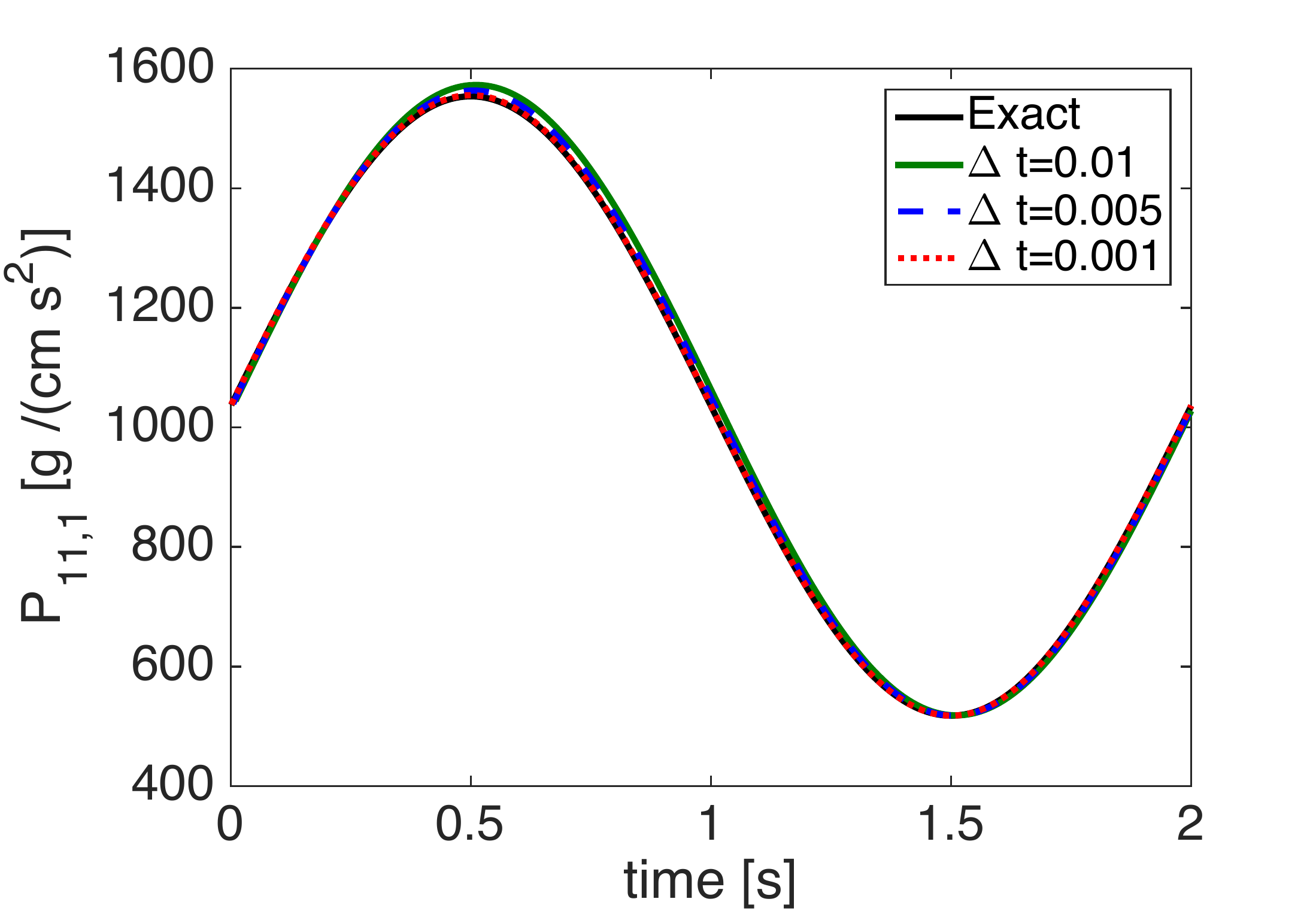}};
\node(label) at (-1.0,-0.8) {\includegraphics[width=0.16\linewidth]{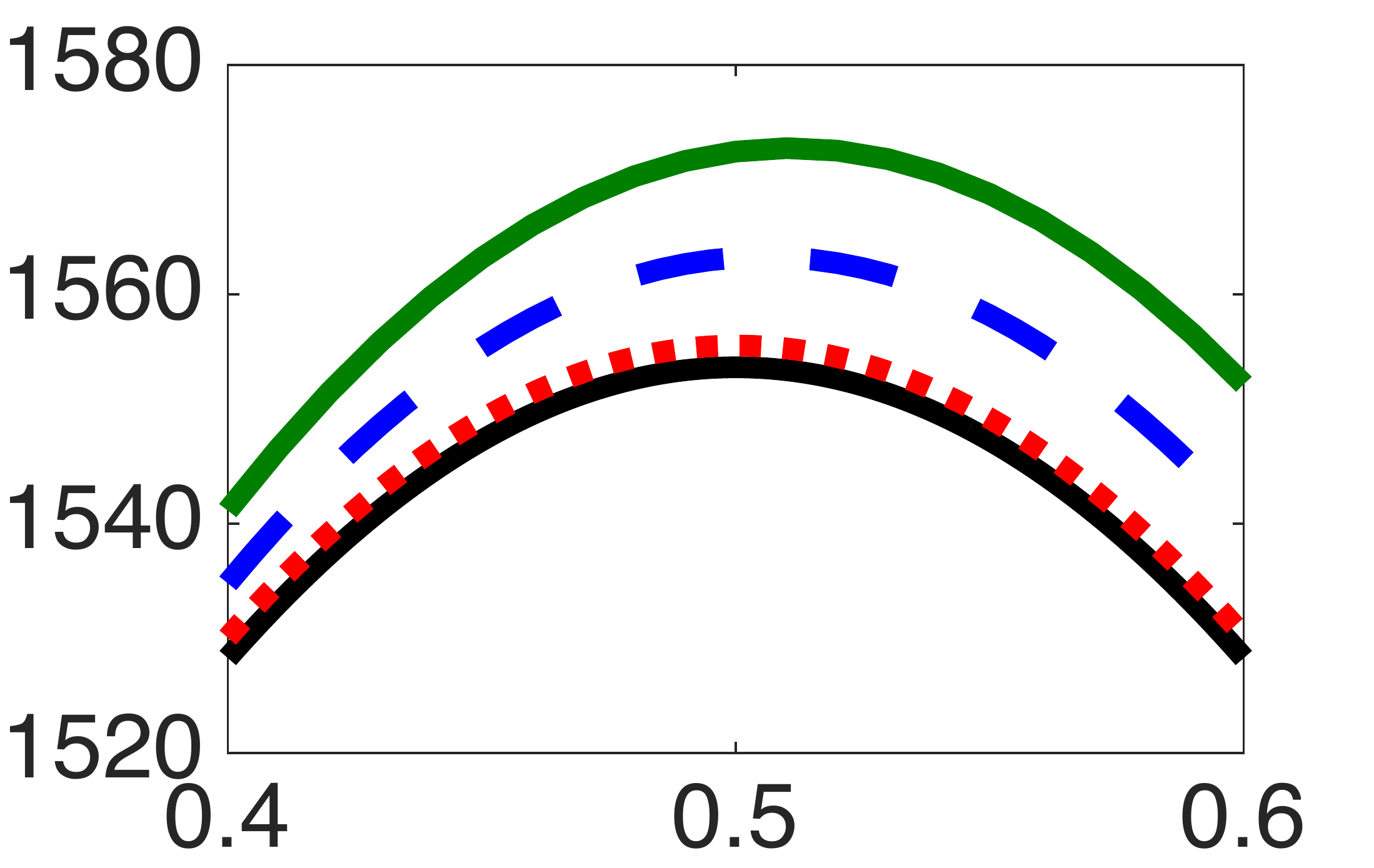}};
\begin{scope}[shift={(-1.2,1.55)},yscale=1.1]
\draw[draw=gray] (0,0) rectangle (0.5,0.4);
 \draw[gray] (0,0) -- (-0.45,-1.663);
  \draw[gray] (0.5,0) -- (0.98,-1.663);
\end{scope}
\node(label) at (7,0) {\includegraphics[width=0.54\linewidth]{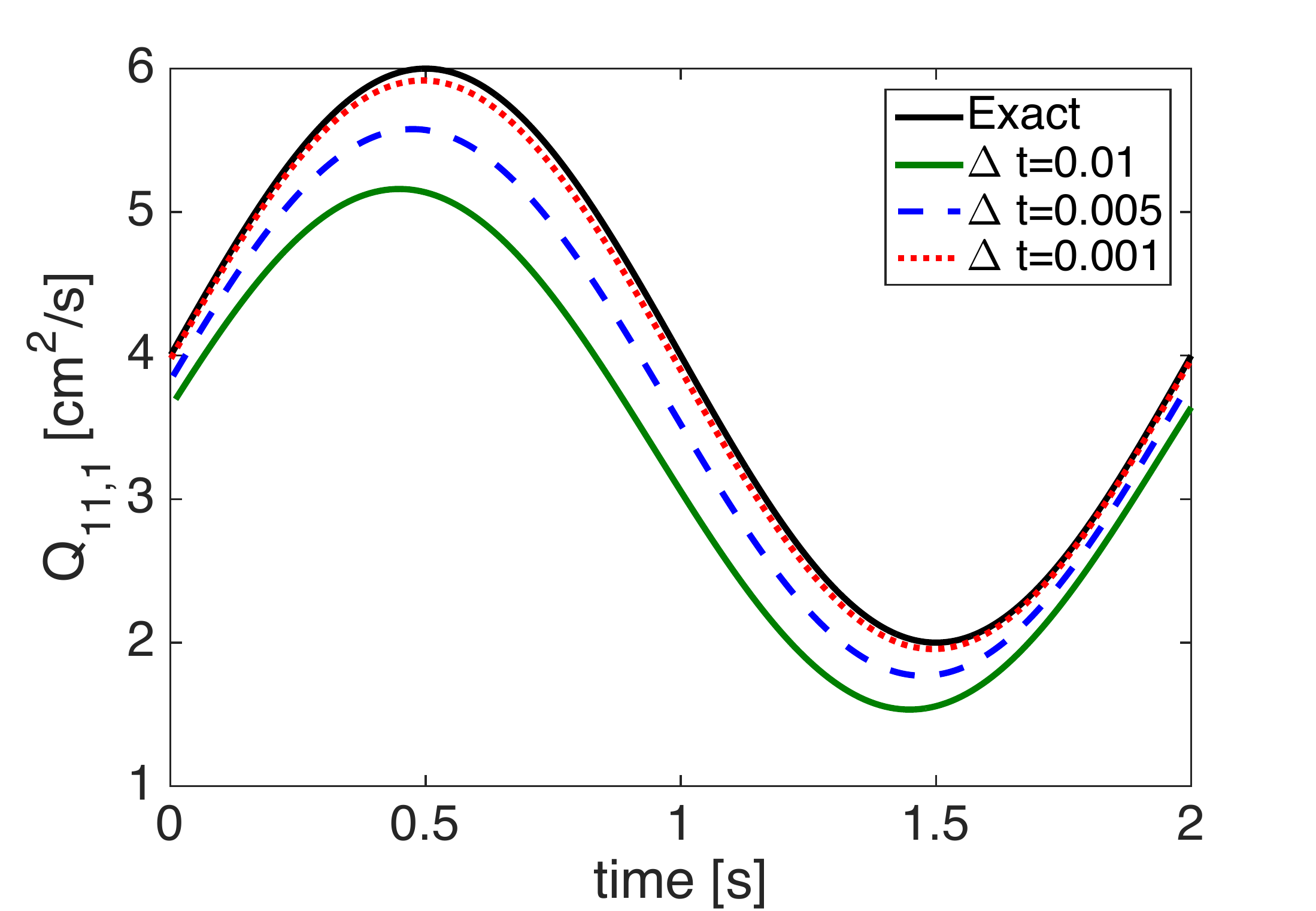}};    

\begin{scope}[shift={(0,0.5)}]
\node(label) at (0,-5) {\includegraphics[width=0.54\linewidth]{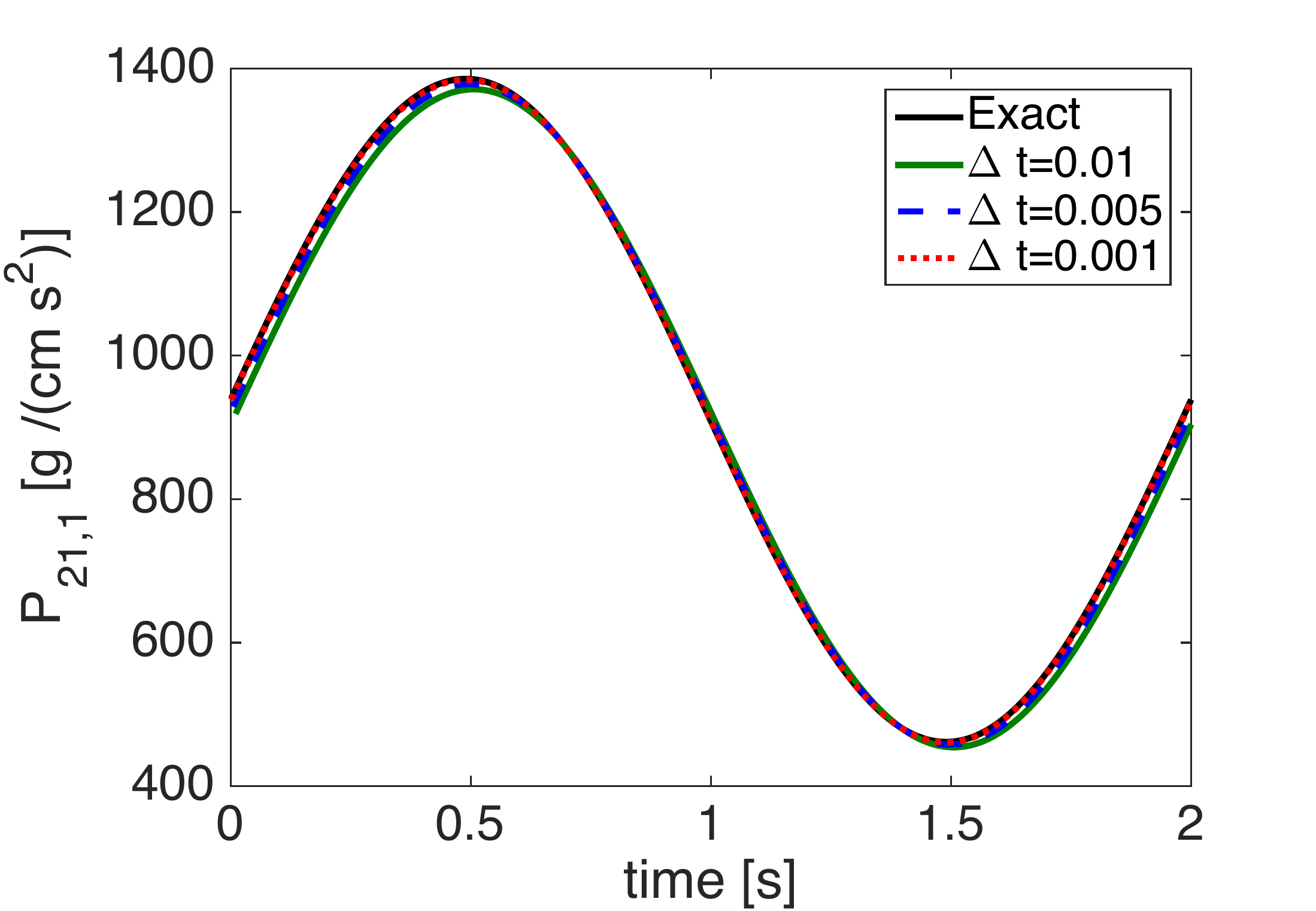}};
\node(label) at (-1.0,-5.8) {\includegraphics[width=0.16\linewidth]{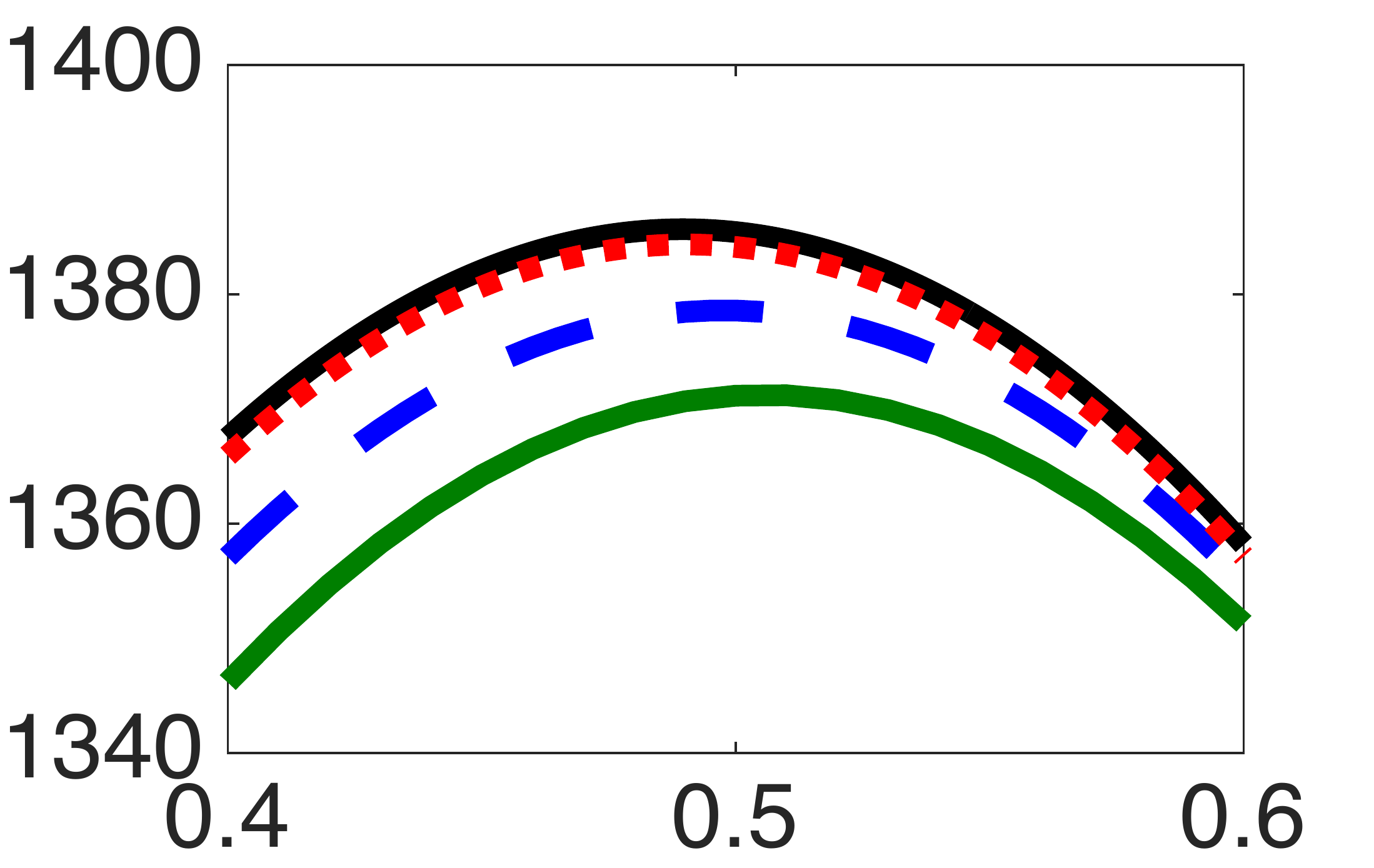}};
\begin{scope}[shift={(-1.2,-3.4)},yscale=1.1]
\draw[draw=gray] (0,0) rectangle (0.5,0.4);
 \draw[gray] (0,0) -- (-0.45,-1.71);
  \draw[gray] (0.5,0) -- (0.98,-1.71);
\end{scope}
\node(label) at (7,-5) {\includegraphics[width=0.54\linewidth]{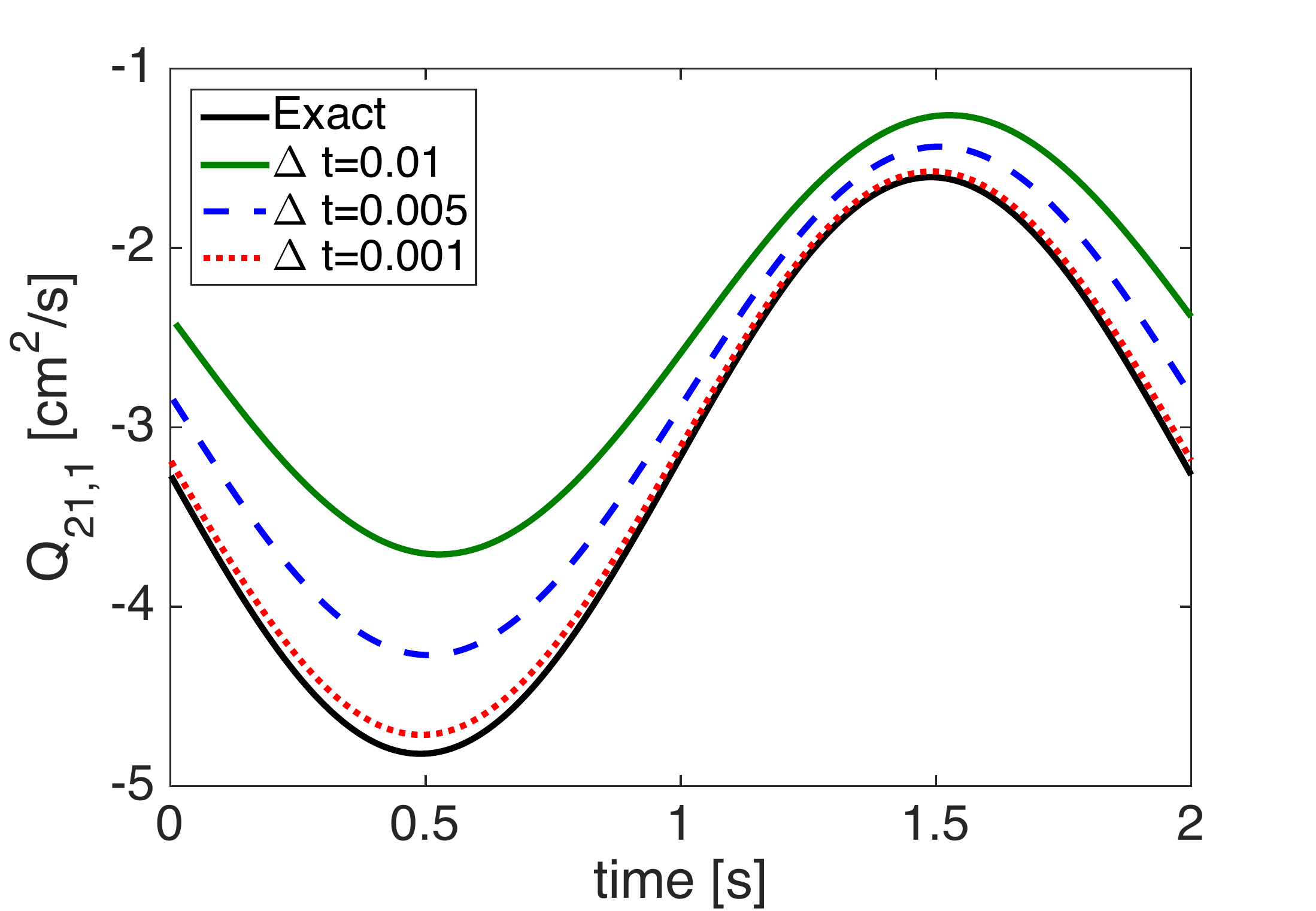}};
\end{scope}

\begin{scope}[shift={(0,-9.5)}]
  \coordinate [label=above:\textcolor{black}{\bf \textsc{0d} unknowns}] (E) at (3.7,2.2);  
\node(label) at (0,0) {\includegraphics[width=0.54\linewidth]{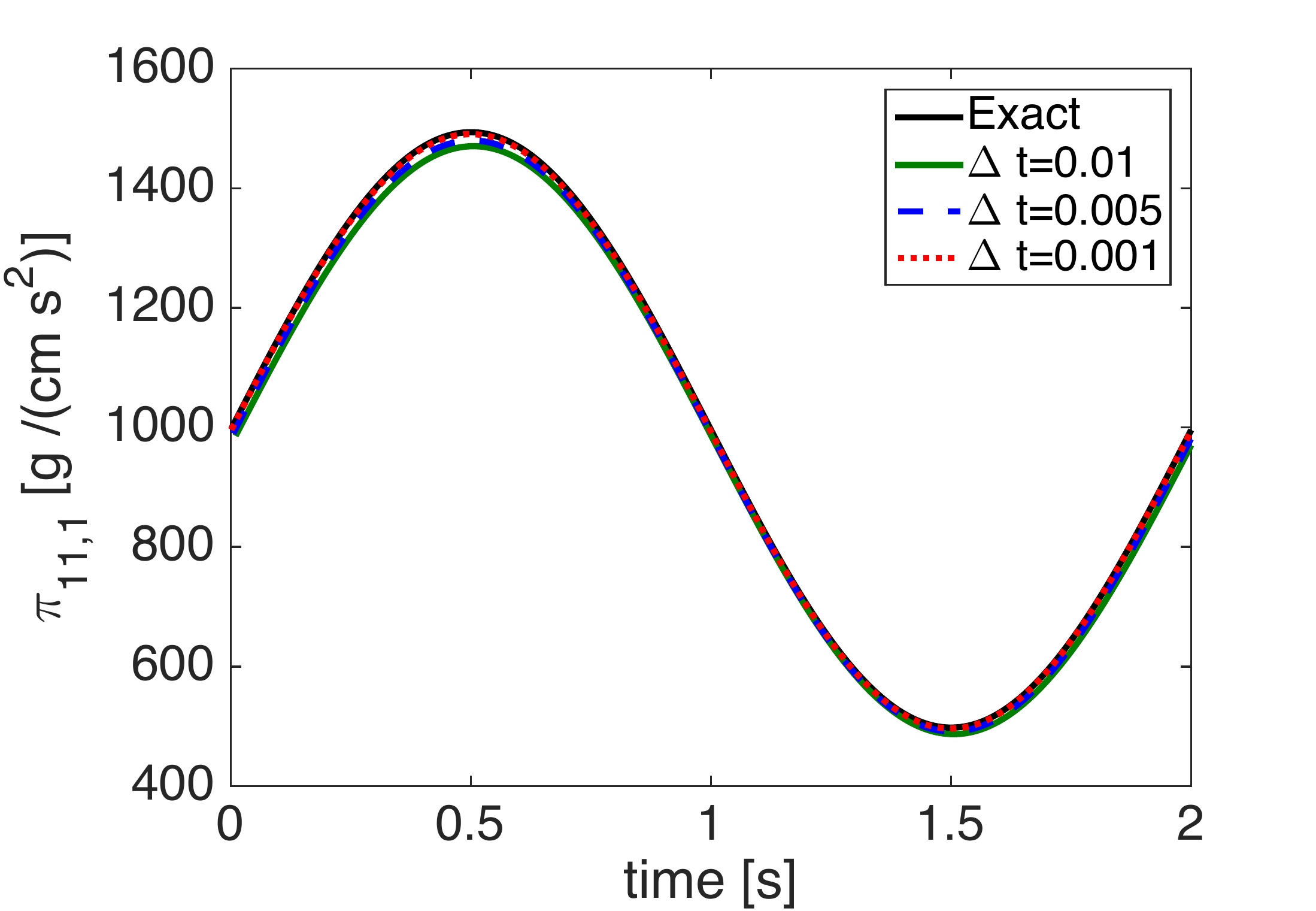}};
\node(label) at (-1.0,-0.8) {\includegraphics[width=0.16\linewidth]{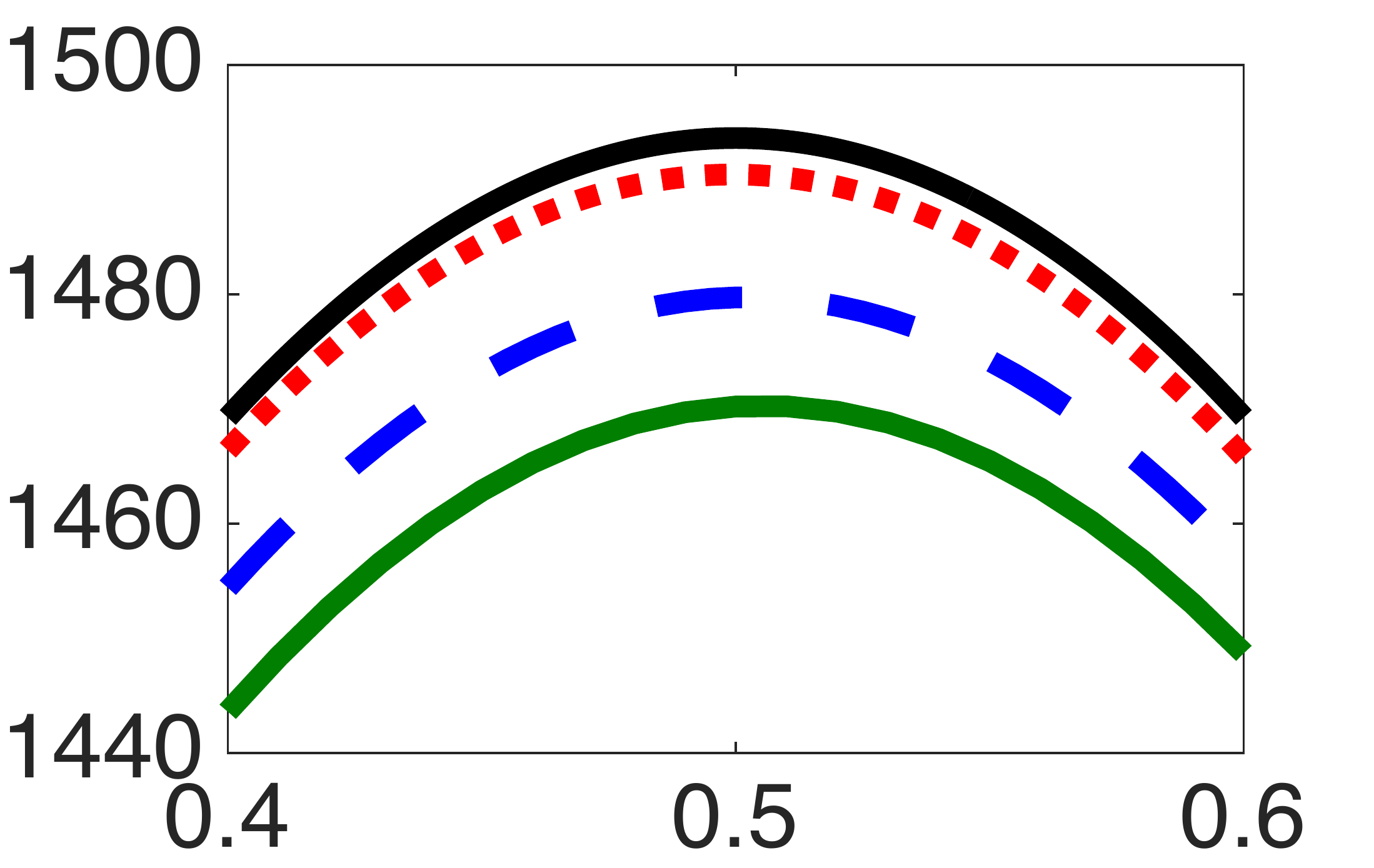}};
\begin{scope}[shift={(-1.2,1.33)},yscale=1.1]
\draw[draw=gray] (0,0) rectangle (0.5,0.4);
 \draw[gray] (0,0) -- (-0.45,-1.462);
  \draw[gray] (0.5,0) -- (0.98,-1.462);
\end{scope}
\node(label) at (7,-0) {\includegraphics[width=0.54\linewidth]{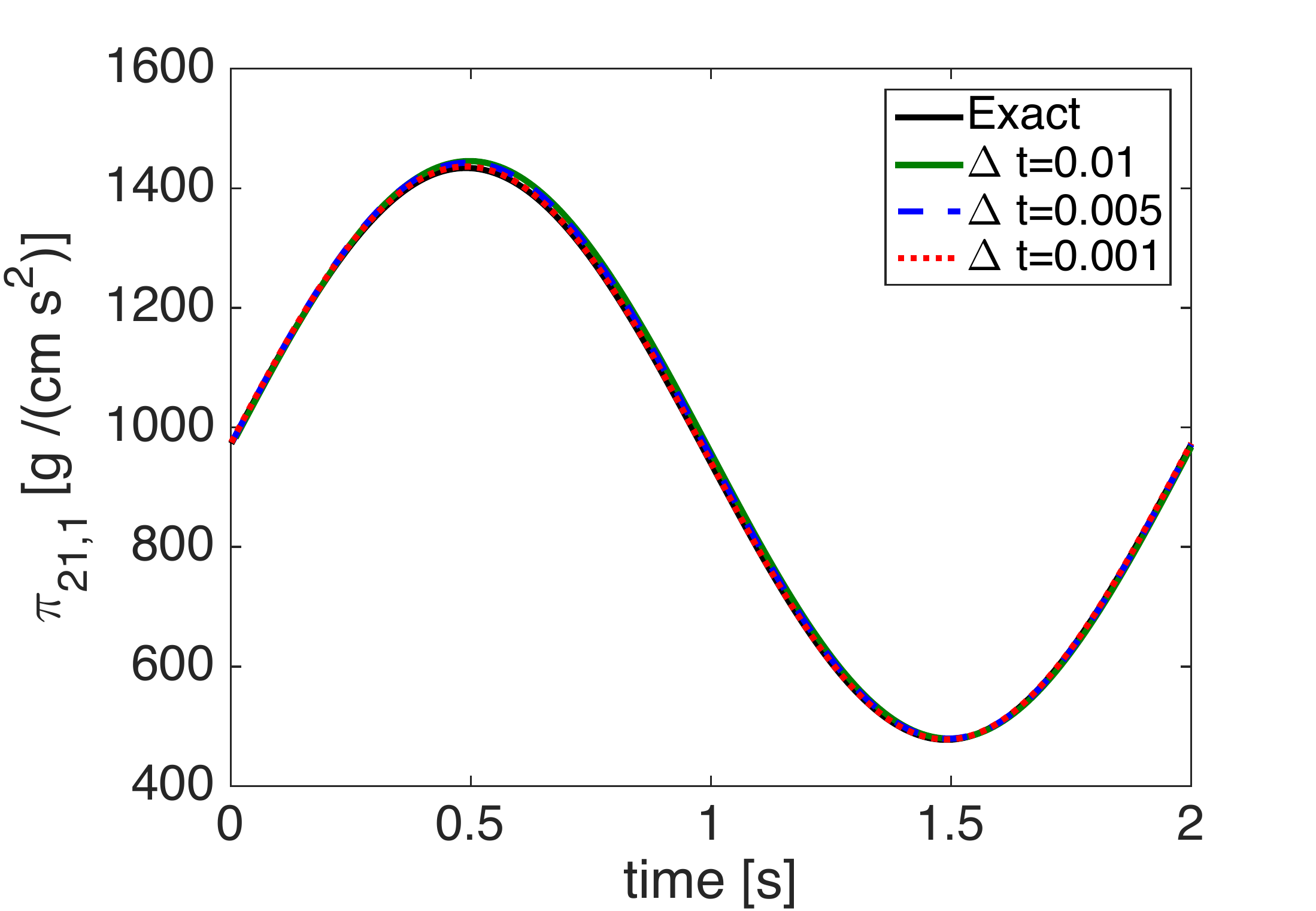}};
\node(label) at (6,-0.8) {\includegraphics[width=0.16\linewidth]{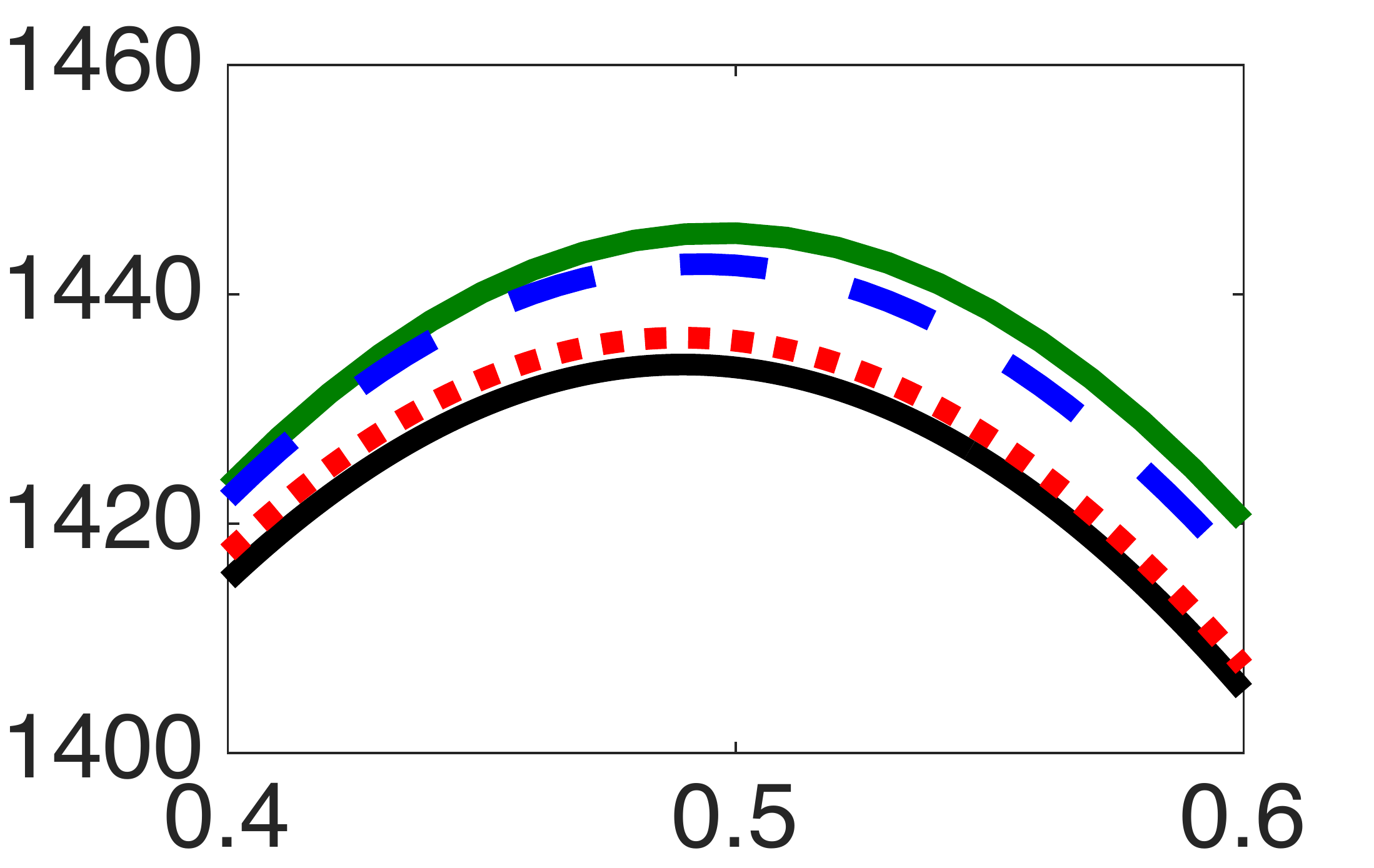}};
\begin{scope}[shift={(5.8,1.22)},yscale=1.1]
\draw[draw=gray] (0,0) rectangle (0.5,0.4);
 \draw[gray] (0,0) -- (-0.45,-1.36);
  \draw[gray] (0.5,0) -- (0.99,-1.36);
\end{scope}

\begin{scope}[shift={(3.7,0.5)}]
\node(label) at (0,-5) {\includegraphics[width=0.54\linewidth]{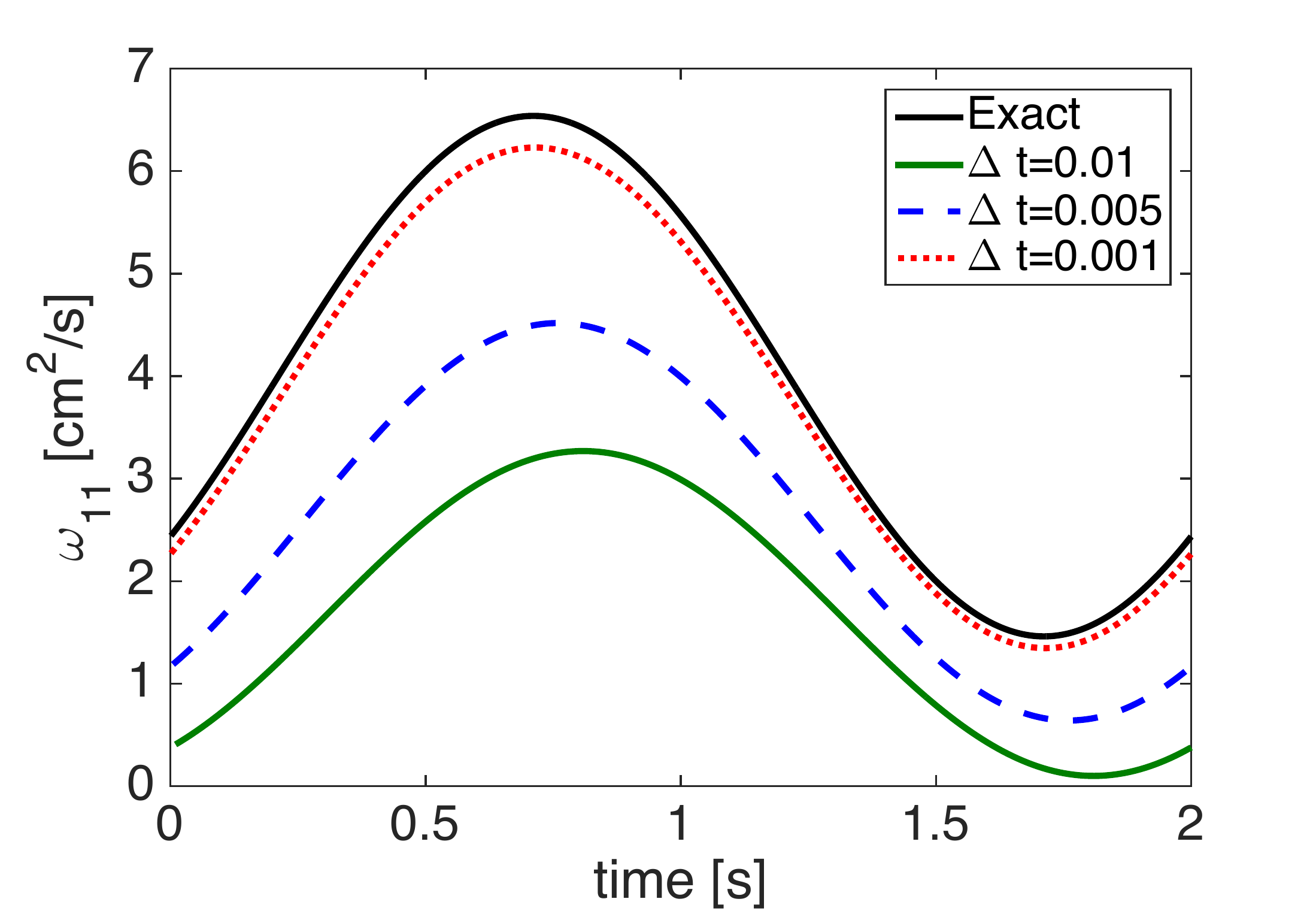}};
\end{scope}

\end{scope}

\end{tikzpicture}}
    \caption{\textit{Example 2.} Comparison between the exact solution and the corresponding numerical approximation for 
    interface quantities and \textsc{0d} unknowns, for three time steps $\Delta t=0.01,\ 0.005,\ 0.001$, over one period once the periodicity is reached.}
  \label{fig:num_result_ex2}
\end{figure}

\textbf{Example 3.}
Figure~\ref{fig:num_result_ex3} displays a comparison between the exact solution and
 pressures and flow rates at the Stokes-circuit interfaces, namely  $P_{11,1}$ and $Q_{11,1}$ at the interface $S_{11,1}$ and $P_{11,2}$ and $Q_{11,2}$ at the interface $S_{11,2}$ (upper panel) and the state variables $\pi_{11,1}$, $\pi_{11,2}$ and $\omega_{11}$ pertaining to the \textsc{0d} circuit $\Upsilon_1$ (lower panel). 
Similarly to Examples 1 and 2, the comparison shows very good agreement in the pressures, both for the interface pressures $P_{11,1}$ and $P_{11,2}$ and for the nodal pressures $\pi_{11,1}$ and $\pi_{11,2}$. The approximation of the interface flow rates $Q_{11,1}$ and $Q_{11,2}$ improves as $\Delta t$ decreases, capturing periodicity and peaks. Conversely, the numerical approximation of the flow rate $\omega_{11}$ shows a very good agreement even for $\Delta t=0.01$. Note that, in contrast with Examples 1 and 2, the Stokes problem in $\Omega_1$ does not include any external forcing term. Thus, the fully coupled problem is driven only by the time evolution of the pressures $\widetilde{p}_a$ and $\widetilde{p}_b$ imposed by the voltage generators within the \textsc{0d} circuit.\\
\begin{figure}[htbp]
\centering
\scalebox{0.88}{
  \begin{tikzpicture}[xscale=1]
    \coordinate [label=above:\textcolor{black}{\bf Interface quantities}] (E) at (3.7,2.2);  
\node(label) at (0,0) {\includegraphics[width=0.54\linewidth]{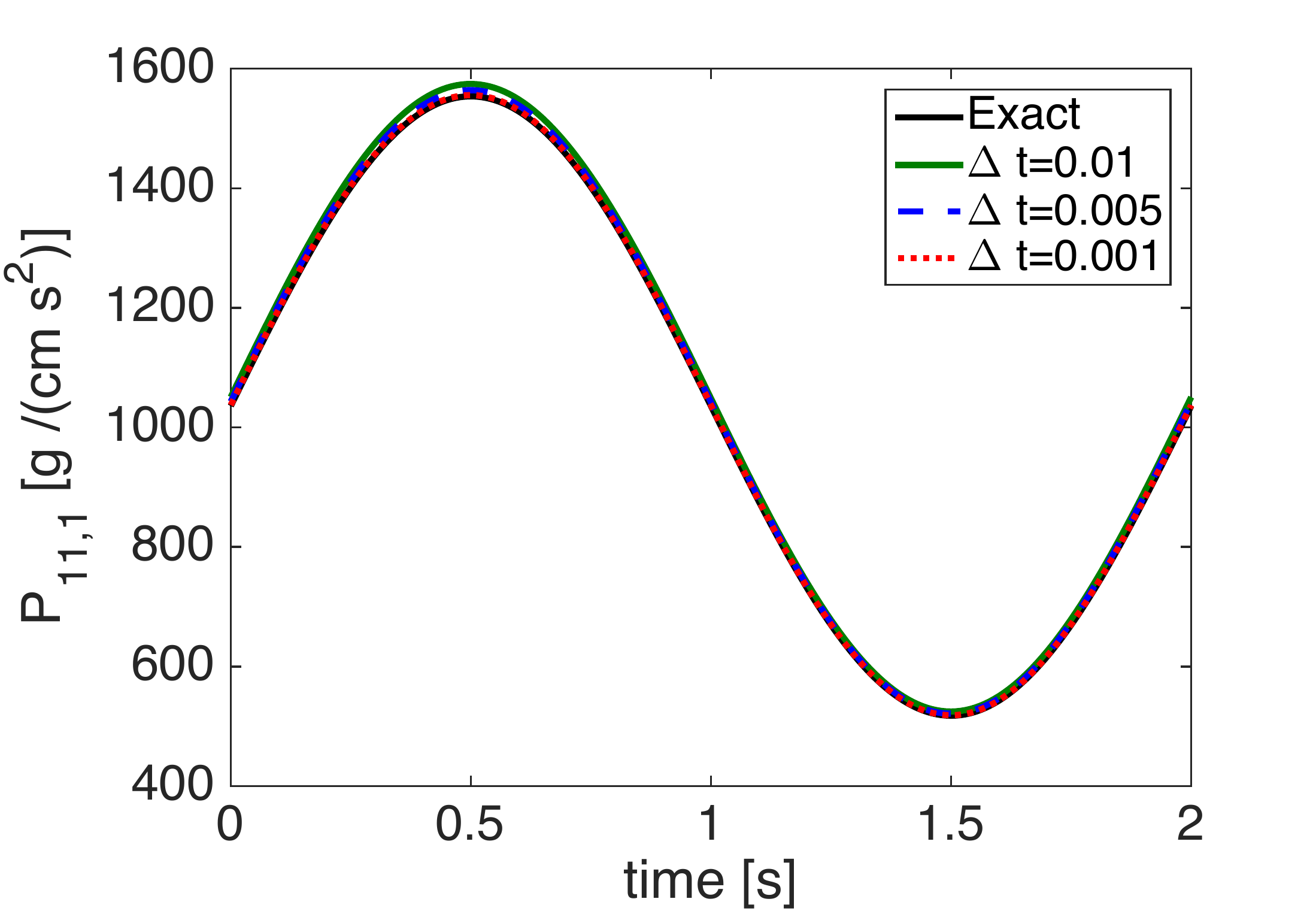}};
\node(label) at (-1.0,-0.8) {\includegraphics[width=0.16\linewidth]{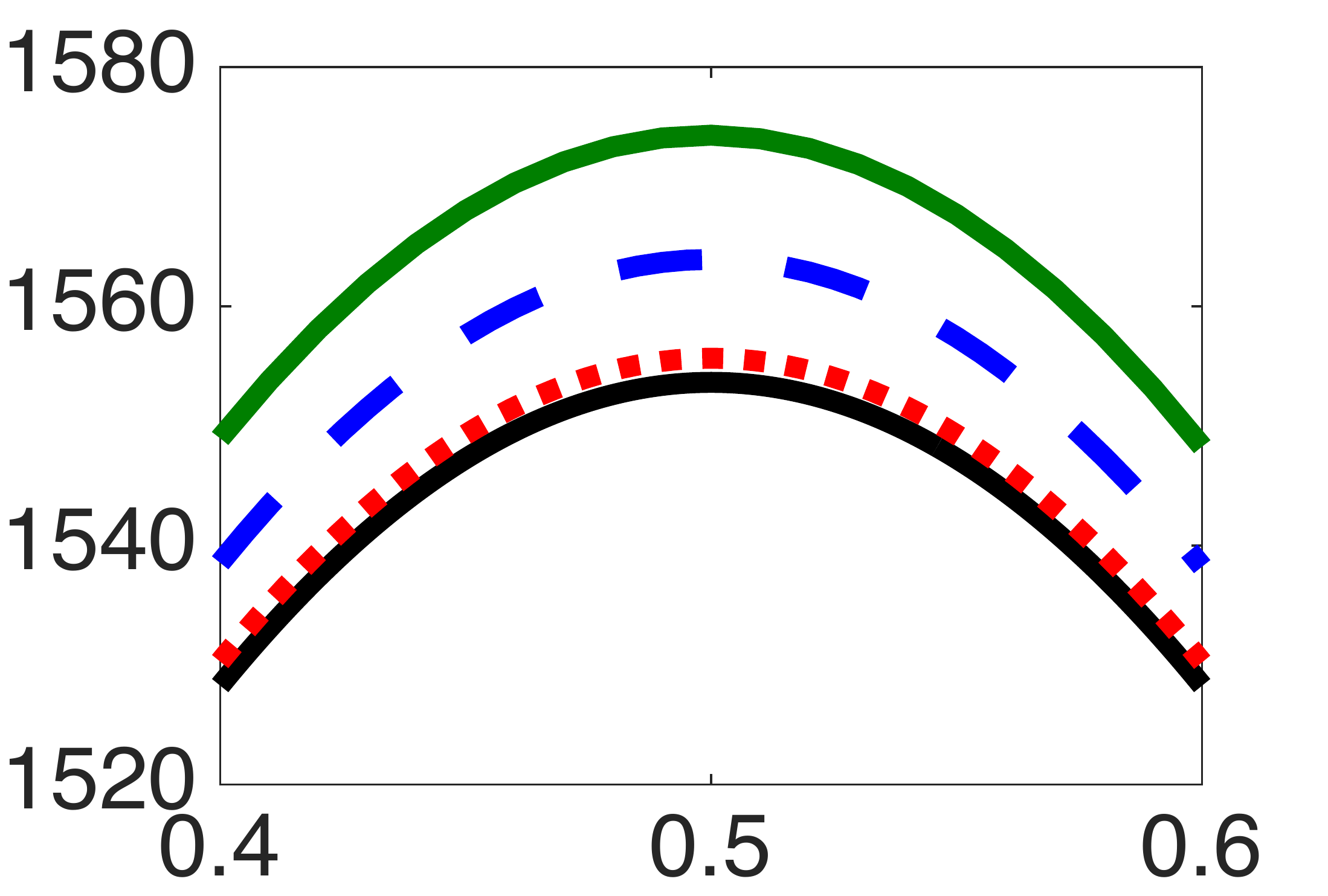}};
\begin{scope}[shift={(-1.2,1.55)},yscale=1.1]
\draw[draw=gray] (0,0) rectangle (0.5,0.4);
 \draw[gray] (0,0) -- (-0.45,-1.63);
  \draw[gray] (0.5,0) -- (0.98,-1.63);
\end{scope}
\node(label) at (7,0) {\includegraphics[width=0.54\linewidth]{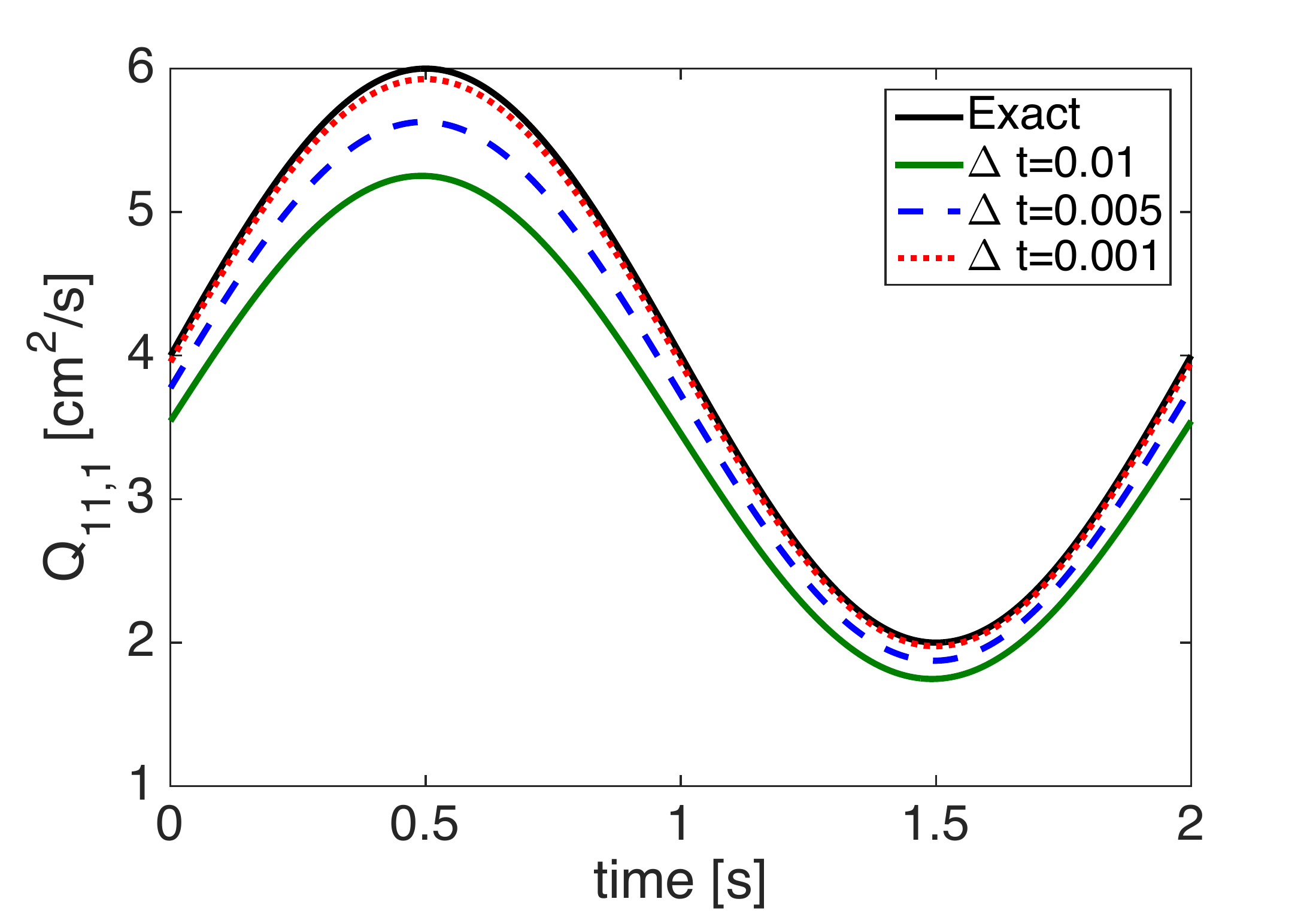}};    

\begin{scope}[shift={(0,0.5)}]
\node(label) at (0,-5) {\includegraphics[width=0.54\linewidth]{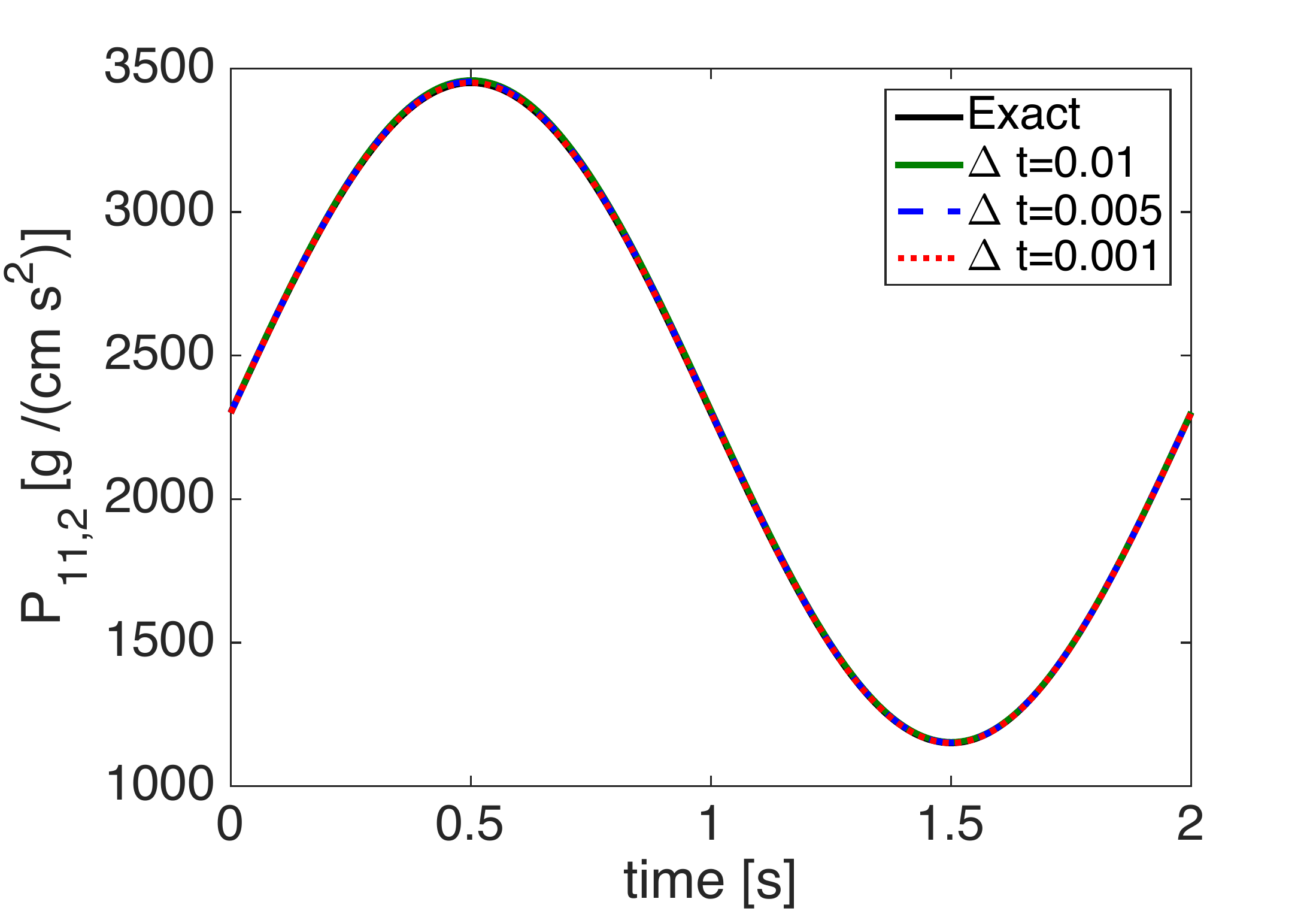}};
\node(label) at (-1.0,-5.8) {\includegraphics[width=0.16\linewidth]{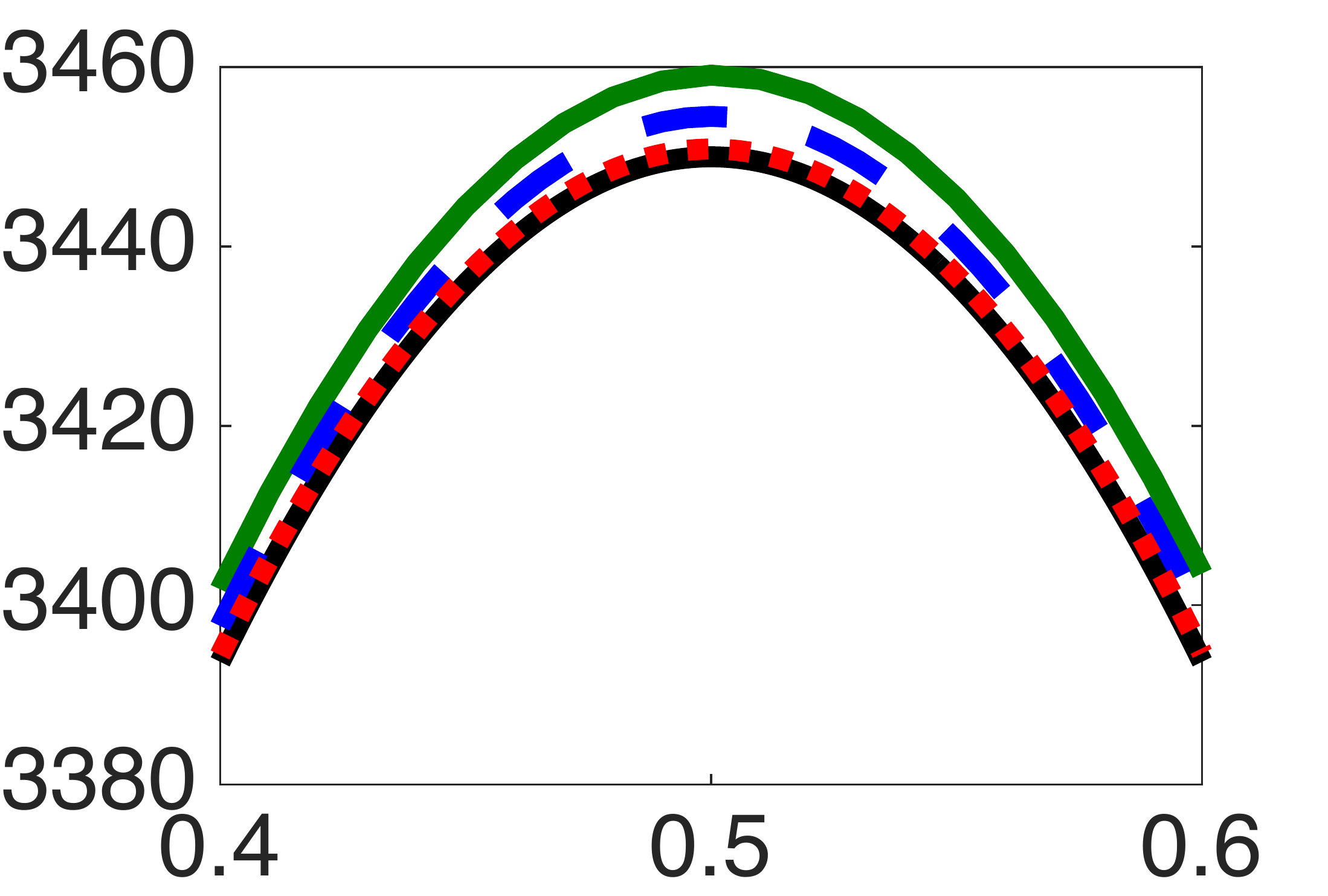}};
\begin{scope}[shift={(-1.196,-3.4)},yscale=1.1]
\draw[draw=gray] (0,0) rectangle (0.5,0.4);
 \draw[gray] (0,0) -- (-0.45,-1.68);
  \draw[gray] (0.5,0) -- (0.98,-1.68);
\end{scope}
\node(label) at (7,-5) {\includegraphics[width=0.54\linewidth]{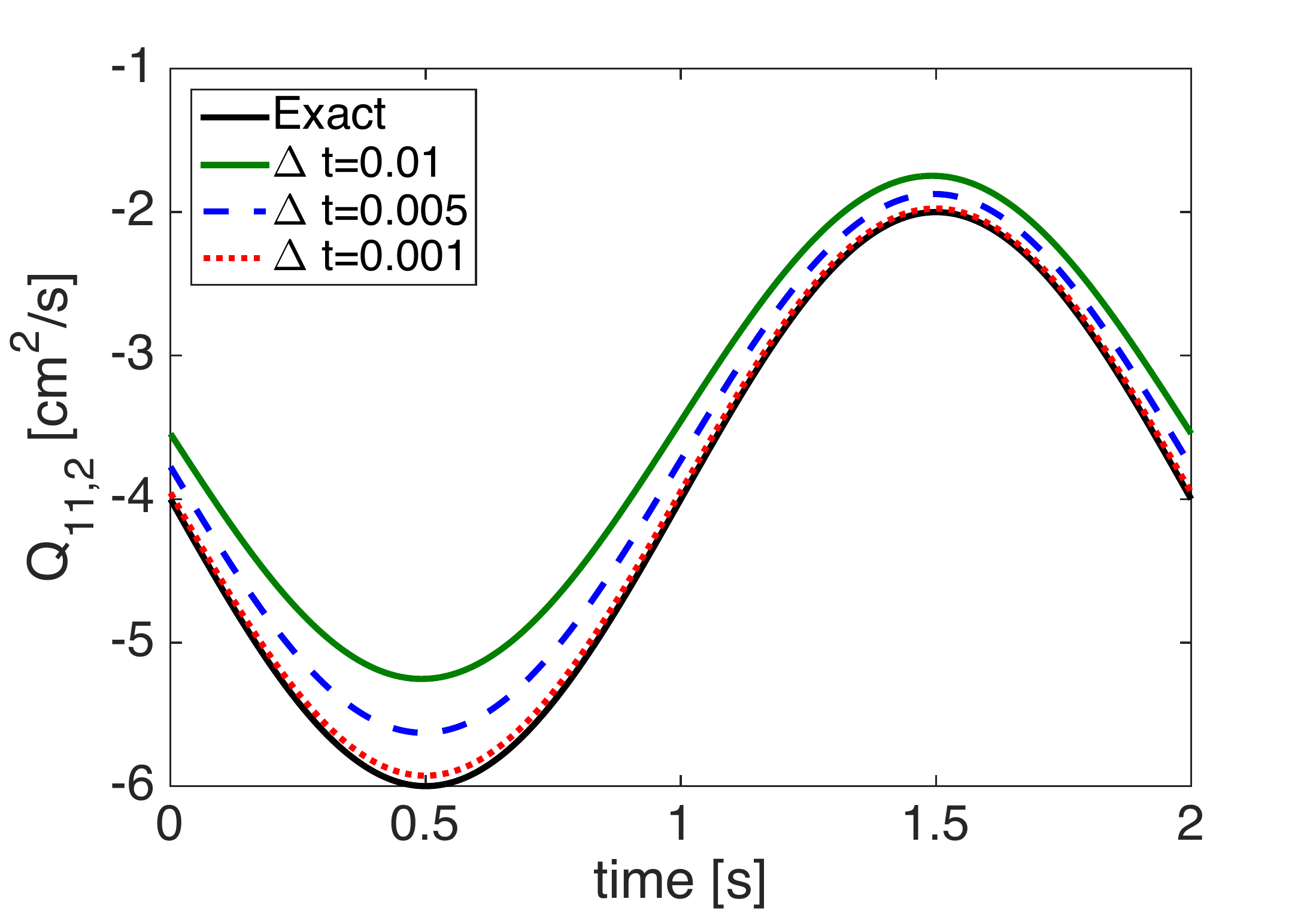}};
\end{scope}

\begin{scope}[shift={(0,-9.5)}]
  \coordinate [label=above:\textcolor{black}{\bf \textsc{0d} unknowns}] (E) at (3.7,2.2);  
\node(label) at (0,0) {\includegraphics[width=0.54\linewidth]{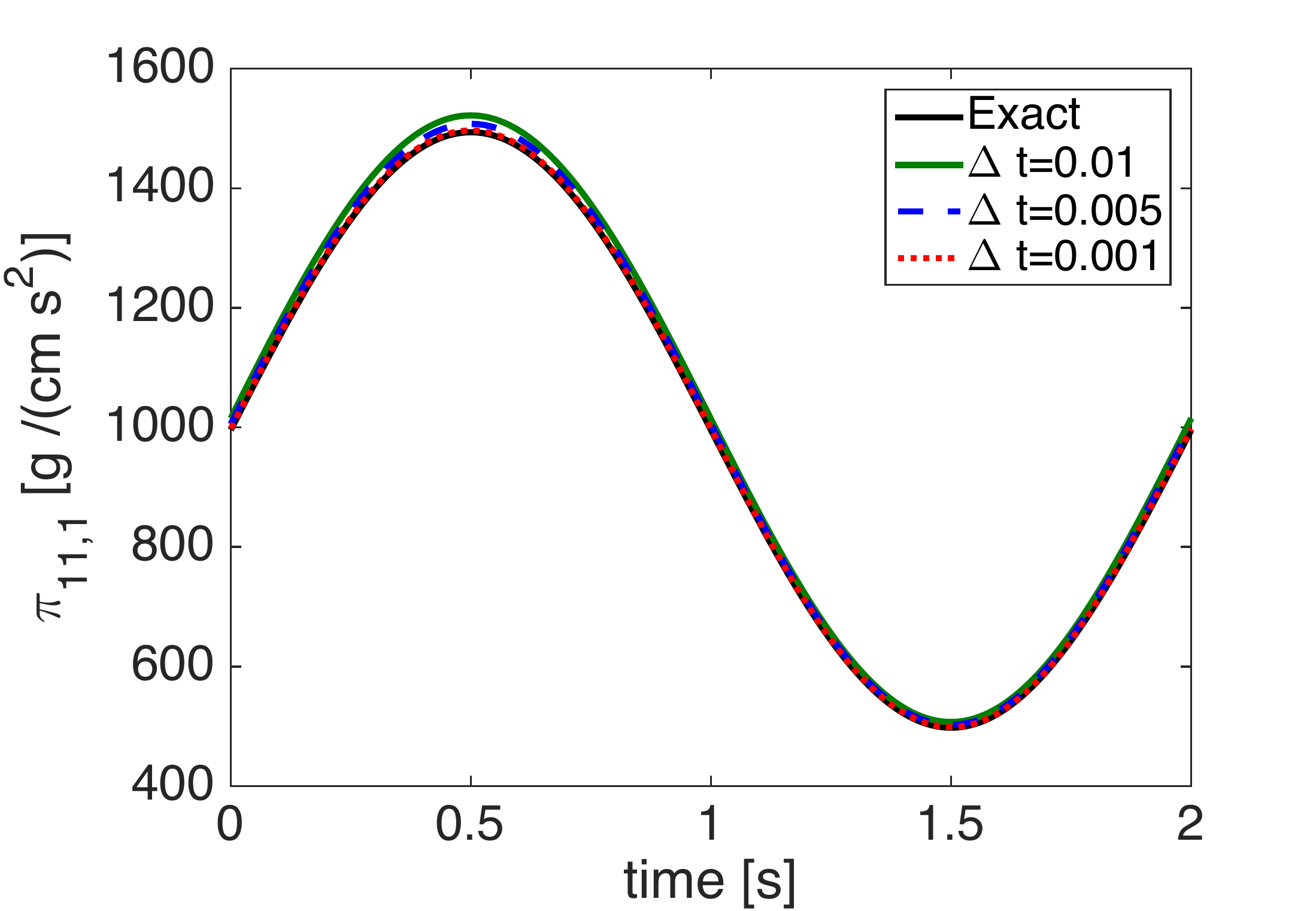}};
\node(label) at (-1.0,-0.8) {\includegraphics[width=0.16\linewidth]{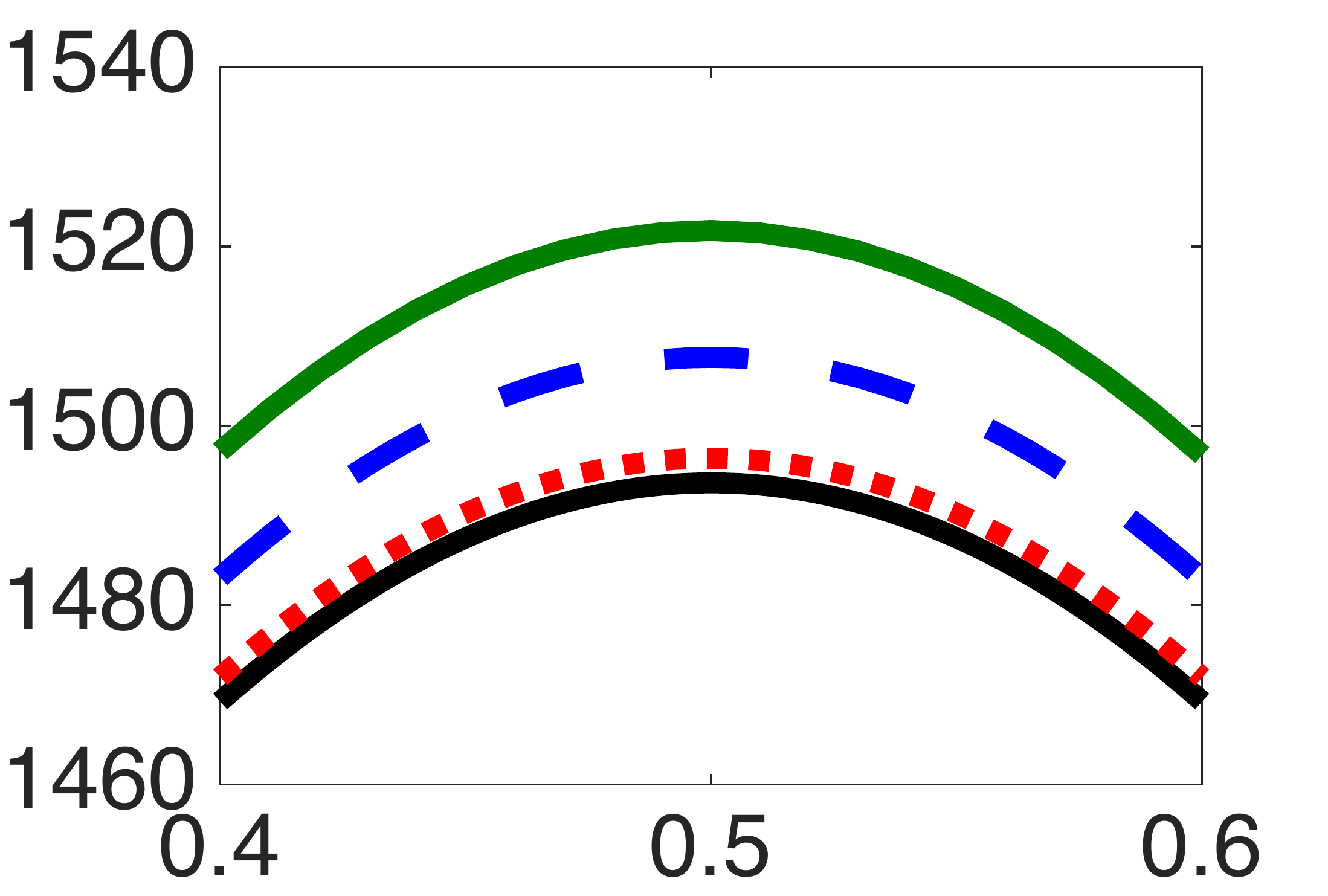}};
\begin{scope}[shift={(-1.2,1.37)},yscale=1.1]
\draw[draw=gray] (0,0) rectangle (0.5,0.4);
 \draw[gray] (0,0) -- (-0.45,-1.462);
  \draw[gray] (0.5,0) -- (0.98,-1.462);
\end{scope}
\node(label) at (7,-0) {\includegraphics[width=0.54\linewidth]{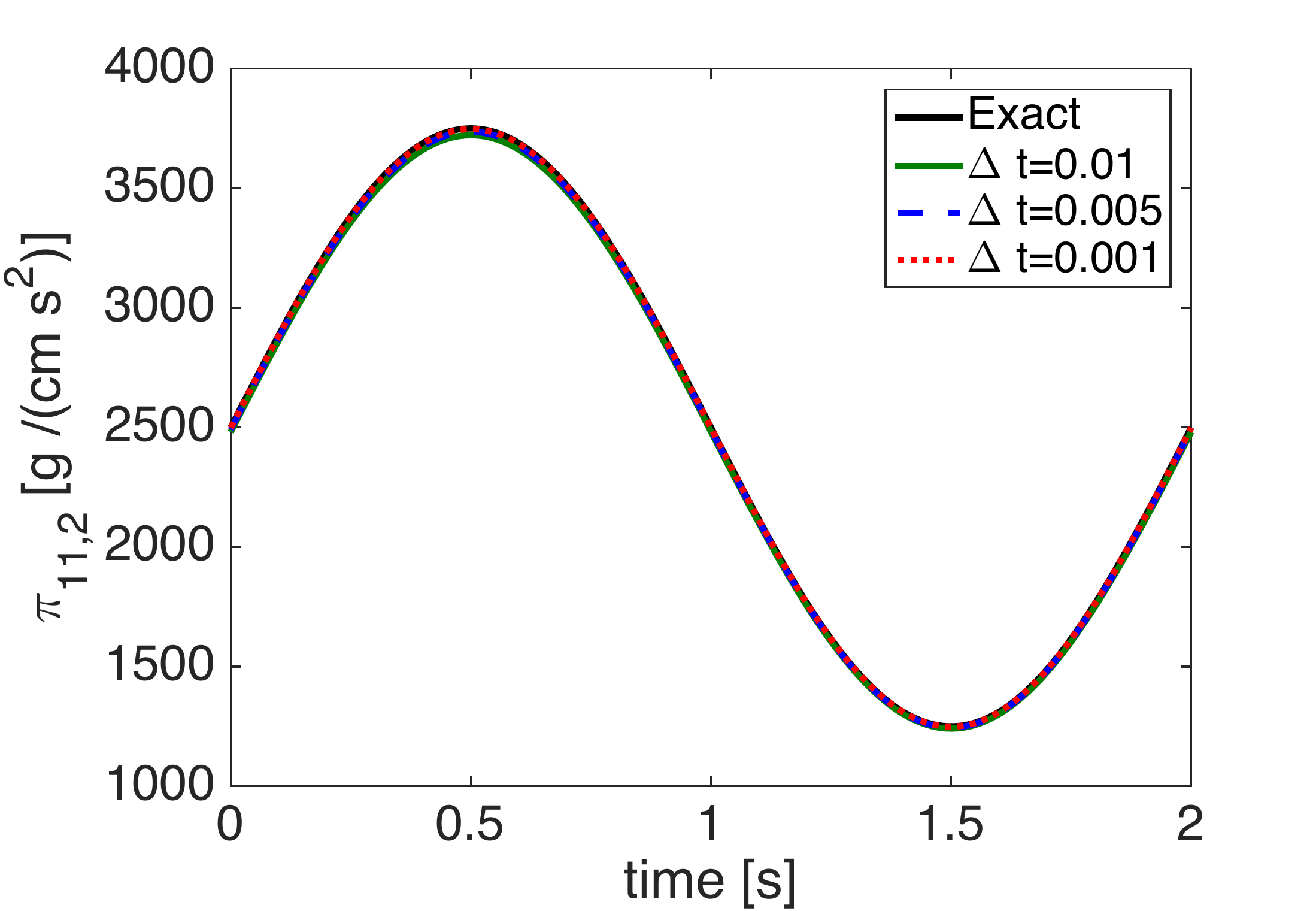}};
\node(label) at (6,-0.8) {\includegraphics[width=0.16\linewidth]{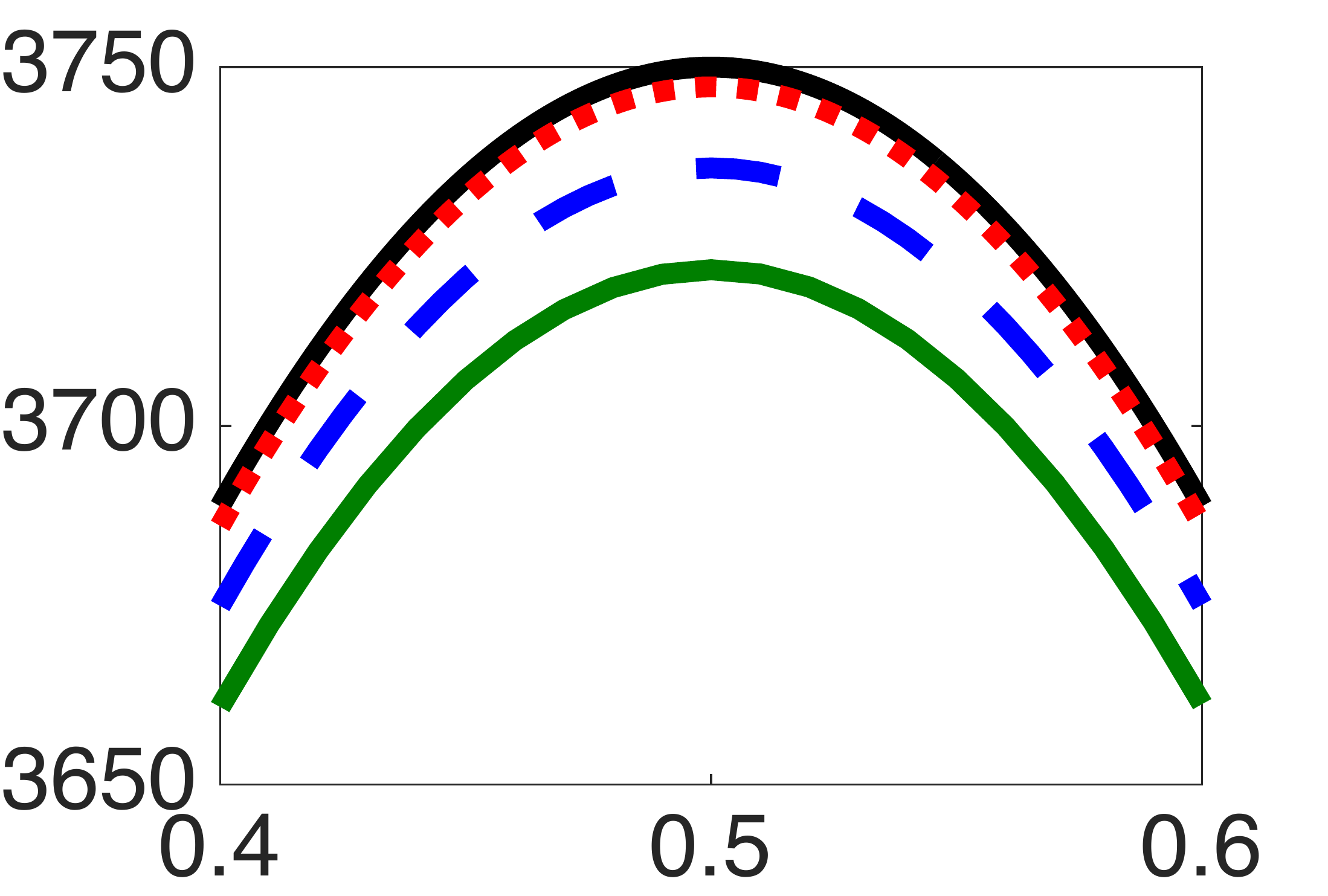}};
\begin{scope}[shift={(5.8,1.35)},yscale=1.1]
\draw[draw=gray] (0,0) rectangle (0.5,0.4);
 \draw[gray] (0,0) -- (-0.45,-1.45);
  \draw[gray] (0.5,0) -- (0.985,-1.45);
\end{scope}

\begin{scope}[shift={(3.7,0.5)}]
\node(label) at (0,-5) {\includegraphics[width=0.54\linewidth]{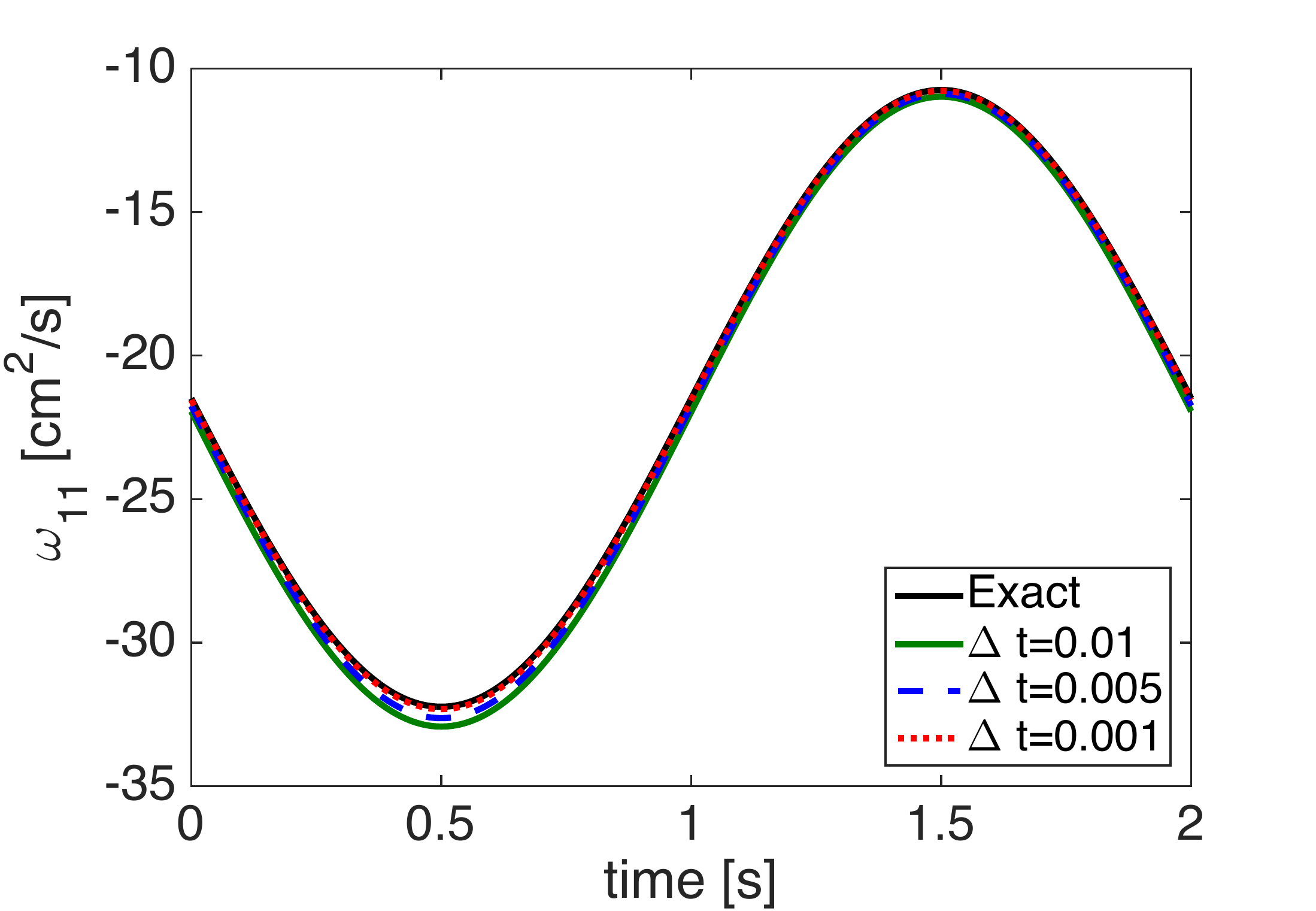}};
\node(label) at (-1.15,-3.8) {\includegraphics[width=0.16\linewidth]{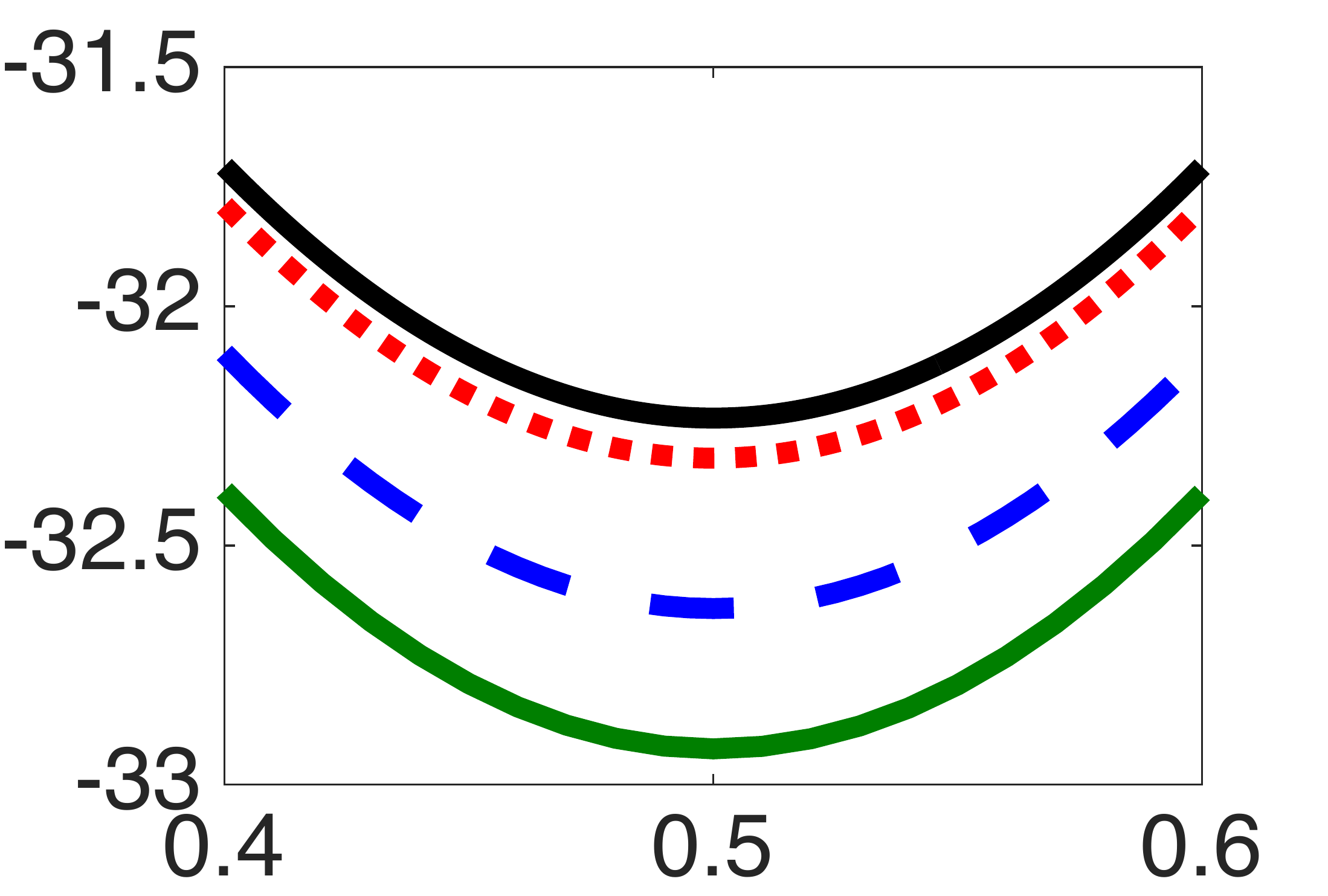}};
\begin{scope}[shift={(-1.33,-6.05)},yscale=-1.1]
\draw[draw=gray] (0,0) rectangle (0.5,0.4);
 \draw[gray] (0,0) -- (-0.455,-1.6);
  \draw[gray] (0.5,0) -- (0.968,-1.6);
\end{scope}
\end{scope}

\end{scope}
\end{tikzpicture}}
    \caption{\textit{Example 3.} Comparison between the exact solution and the corresponding numerical approximation for 
    interface quantities and \textsc{0d} unknowns, for three time steps $\Delta t=0.01,\ 0.005,\ 0.001$, over one time period once the numerical periodicity has been reached.}
  \label{fig:num_result_ex3}
\end{figure}

\noindent{\bf Order of convergence in time.}
The time-discretization scheme presented in this article is based on a first order operator splitting technique, which is known to be first-order
accurate in time. We performed a standard time refinement study for Example 1, Example 2 and Example 3, using the errors defined in~\eqref{eq:err_comput_v}-\eqref{eq:err_comput_y}, that confirmed this property, as shown in Figure~\ref{fig:err_ex123}.
\begin{figure}[htbp]
\centering
      \begin{subfigure}[b]{0.325\textwidth}
        \centering
        \hspace{-0.4cm}
    \includegraphics[width=1\linewidth]{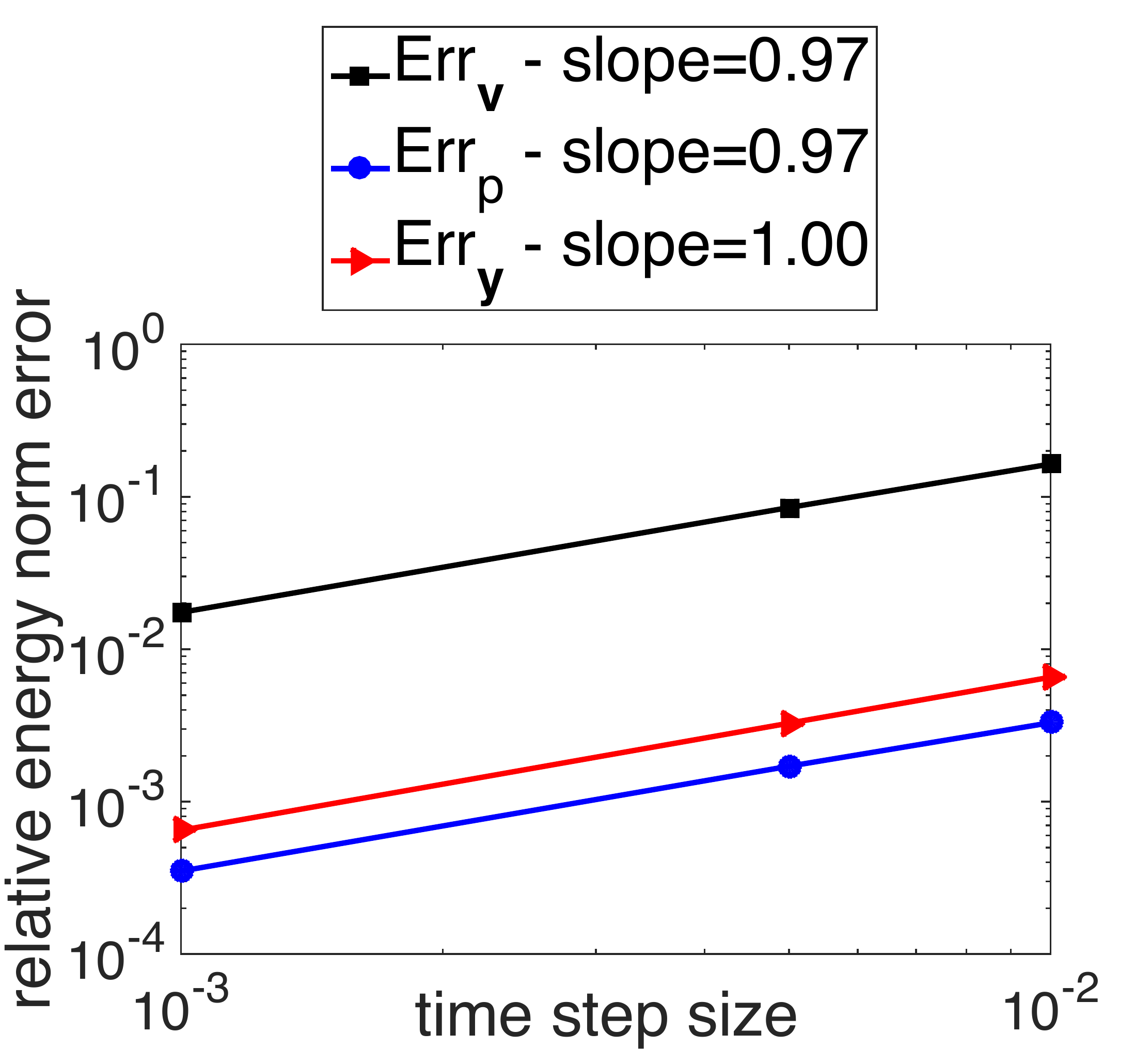}
        \caption{\textit{Example 1}}
        \label{fig:err_ex1}
    \end{subfigure}
      \begin{subfigure}[b]{0.32\textwidth}
        \centering
    \includegraphics[width=1\linewidth]{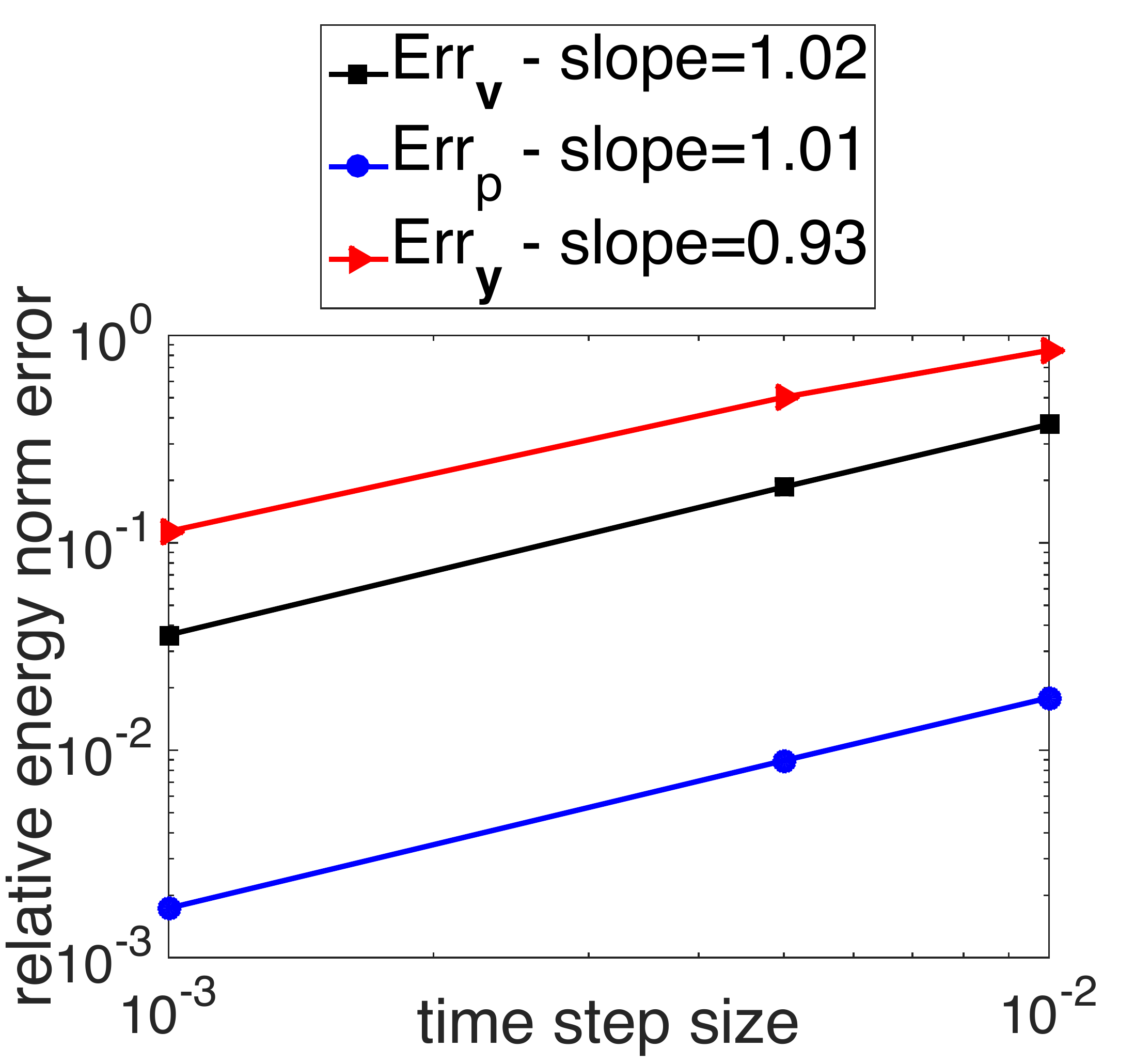}
        \caption{\textit{Example 2}}
        \label{fig:err_ex2}
    \end{subfigure}
      \begin{subfigure}[b]{0.32\textwidth}
        \centering
    \includegraphics[width=1\linewidth]{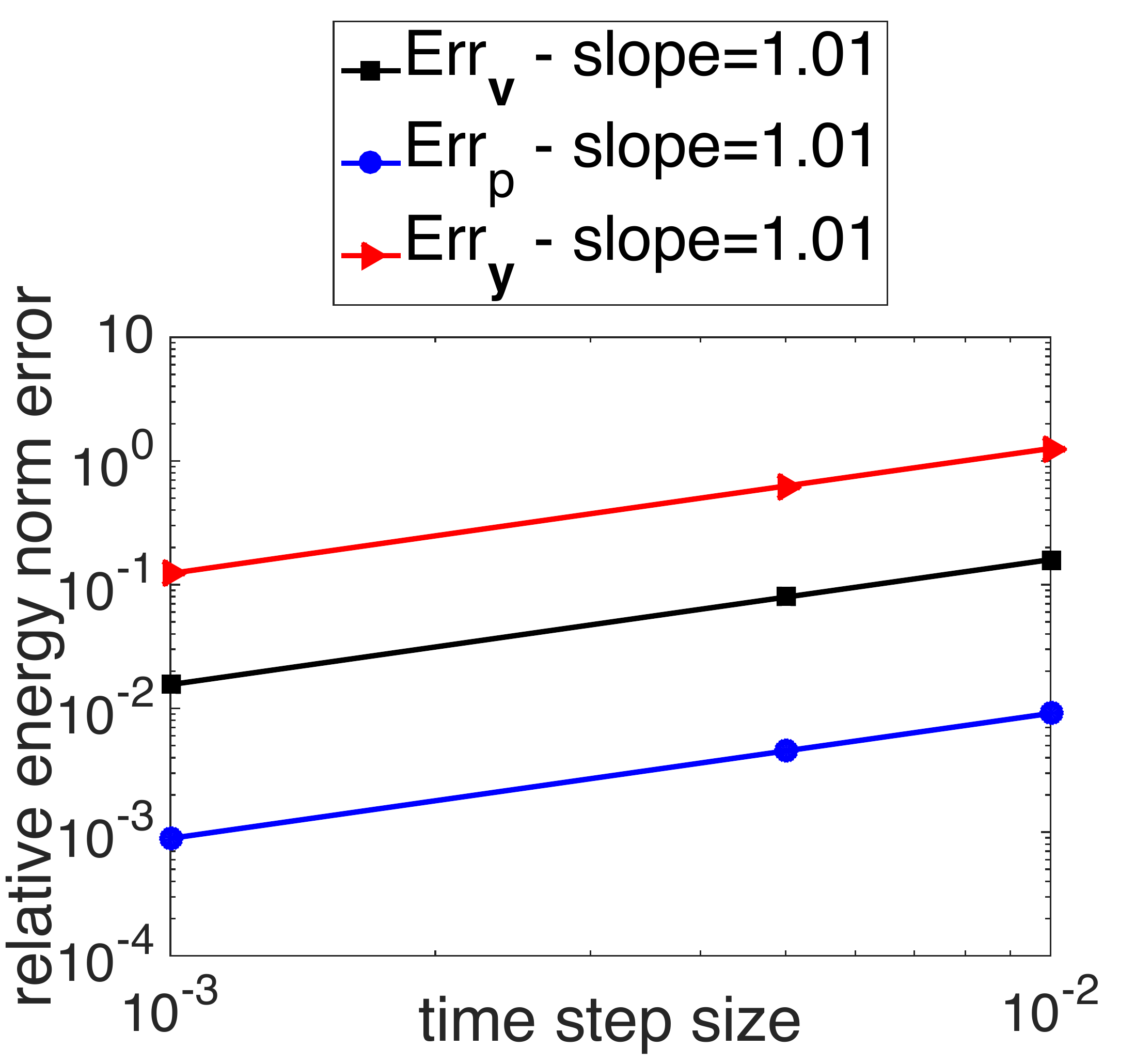}
        \caption{\textit{Example 3}}
        \label{fig:err_ex3}
    \end{subfigure}    
 \caption{Plot of the energy norm errors defined in~\eqref{eq:err_comput_v}-\eqref{eq:err_comput_y} in logarithmic scale as a function of the global time step $\Delta t=0.01,0.005,0.001$ for the three examples considered.}
 \label{fig:err_ex123}
\end{figure}

\section{Conclusions and Future Perspectives}
\label{sec:extensions}

This work presents a novel energy-based operator splitting approach for the numerical solution of multiscale problems involving the coupling between Stokes equations and ODE systems. 
Unconditional stability with respect to the time step choice is ensured by the implicit treatment of interface conditions within the Stokes substeps, whereas the coupling between 
Stokes and ODE substeps is enforced via appropriate initial conditions for each substep. Notably, the splitting design has been driven by the rationale of ensuring that 
the physical energy balance is maintained at the discrete level and, as a consequence, unconditional stability is attained without the need of subiterating between substeps. 
Results were presented in the case of \textsc{2d}-\textsc{0d} coupling, for the sake of simplicity. However, the \textsc{3d}-\textsc{0d} case does not present any additional conceptual issue from the splitting viewpoint. Upon comparison with exact solutions, numerical results show a better agreement for pressures than for flow rates at the Stokes-circuit interfaces. This might be a consequence of the simple finite element approach adopted here, where pressures are primal variables but flow rates are not.

In this work, the splitting algorithm is presented in its simplest version, which is at most first-order accurate in time. It has the
advantages of easy implementation, low computational cost and good stability and robustness properties. In addition, the framework presented here could serve as a basis for extensions
to 
\textit{(i)} \textit{higher order time-discretization methods} via, for example, symmetrization  \cite{glowinski2003} or time-extrapolation \cite{Fernandez2013, Fernandez2015};
\textit{(ii)} \textit{different PDE systems} in each $\Omega_l$ region deriving, for example, from Navier-Stokes equations~\cite{glowinski2006}, fluid-structure 
interactions~\cite{guidoboni2009,bukac2016}, non-Newtonian fluid flows~\cite{dean2007} or porous media flows~\cite{prada2016thesis,sala2017}; and \textit{(iii)} \textit{different spatial discretization methods} for each PDE problem including, for example, higher order 
finite element methods~\cite{karniadakis2013} or hybridizable discontinuous Galerkin methods~\cite{prada2016thesis,cockburn2009,bociu2016}. The extensions mentioned above are 
relatively straight-forward, since the algorithms described in the corresponding references are directly implementable on that presented in this work. On the other hand, 
extensions to more general coupling architectures 
\revA{require further investigation, possibly involving defective interface conditions, as discussed in~\cite{formaggia2002numerical}, or special numerical strategies depending on the 
physics of the multiscale connections, as discussed in~\cite{moghadam2013}.}

\revA{The case involving the Navier-Stokes equations deserves particular mention for its relevance to the multiscale modeling of blood flow. In this case, the mathematical problem to be 
solved in Step 1 would also include nonlinearites due to advection which, if not treated properly, might disrupt the physical energy balance at the multiscale 
interface~\cite{formaggia2013physical}. Following the approach described in~\cite{glowinski2003} and implemented in~\cite{glowinski2006,guidoboni2009}, we would again adopt the operator 
splitting viewpoint and modify the overall algorithm by introducing an additional step where the advection is treated separately using energy-preserving schemes, such as the wave-like 
method described in~\cite{glowinski2003}. By doing so, the physical energy balance is maintained from the continuous to the discrete level also in the case of the Navier-Stokes equations, 
thereby providing stability for the overall algorithm.}

\section*{Appendix: Exact solutions of numerical examples}
\noindent{\bf Example 1.}
A direct computation shows that the coupled problem~\eqref{eq:st_div}-\eqref{eq:st_ic_q} with~\eqref{eq:test1_A}-\eqref{eq:flow_rate_11} admits the exact solution
\begin{align}
&\vector{v}_1(x_1,x_2,t)  = \left[ s(t) \exV (x_2), \, 0 \right]^T,  \qquad \qquad p_1(x_1,x_2,t)= s(t) \exP(x_1), \\
&\pi_{11,1} (t)  = P_{11,1}(t) - R_{11,1} Q_{11,1}(t), \\
& \omega_{11} (t) = \frac{1}{2\gamma_1} {\Bigg\{ } -1 \label{eq:y12_ex1}\\
& \qquad + \left. \sqrt{1+4\gamma_1\overline{C}_{a}\left[\pi_{11,1}(t) - R_{a}(\vector{y}_1,t) \left( Q_{11,1}(t)- C_{11,1} \frac{d\pi_{11,1}}{dt}(t) \right) \right]}\;  \right\},  \nonumber
\end{align}
where
\begin{align}
&s(t) = s_0+s_1 \sin(\omega\,t), & &\exV(x_2) = V_0 \cos^2\left(\frac{\pi x_2}{H}\right), \\
&\exP(x_1)  = a_0 + a_1 \exp(-k x_1), & & P_{11,1}(t) = s(t) \exP(L), \\
&Q_{11,1}(t) = \frac{V_0H}{2} s(t), & & R_{a}(\vector{y}_1,t)  = \overline{R}_{a} + \frac{\alpha_0}{1 + \alpha_1 e^{ - \alpha_2 \pi_{11,1}(t) }}, \label{eq:test1:NLR} \\
 &C_{a} (\vector{y}_1,t) = \frac{\overline{C}_{a}}{1+\gamma_1 \omega_{11}(t)} \label{eq:test1:NLC},
\end{align}
provided that the forcing terms are given by
\begin{align}
&
\vector{f}_1(x_1,x_2,t) = 
\left[
\frac{ds}{dt}(t) \exV(x_2) -\frac{\mu}{\rho} s(t) \frac{d^2\exV}{dx_2^2}(x_2) + \frac{s(t)}{\rho} \frac{d\exP}{dx_1}(x_1) , \;
0
\right]^T,\label{eq:test1_force}\\
& \overline{p}_1(t) = s(t) \exP(0),\\
& \widetilde{p} (t)= R_{b}\frac{d \omega_{11}}{dt}(t) - \frac{R_{b}}{R_{a}(\vector{y}_1,t)} \pi_{11,1}(t)+ \frac{R_{b}}{C_{a}(\vector{y}_1,t)}\left(\frac{1}{R_{a}(\vector{y}_1,t)}+\frac{1}{R_{b}} \right) \omega_{11}(t),
\end{align}
and the initial conditions are equal to the exact solution evaluated at $t=0$.

It may be useful to notice that the expression for the volume $\omega_{11}$ in equation \eqref{eq:y12_ex1} results from enforcing the volume-pressure relationship 
\begin{equation}
\omega_{11}(t) = C_{a}(\vector{y}_1,t) \left[ \pi_{11,1}(t) - R_{a}(\vector{y}_1,t) \left( Q_{11,1} (t)- C_{11,1} \frac{d\pi_{11,1}}{dt}(t) \right) \right],
\end{equation}
where $C_{a}$ is given by \eqref{eq:test1:NLC}.
All the parameter values involved in the exact solution  are listed in Table~\ref{tab:param_test123}. Since this example is set in a \textsc{2d} context, 
flow rates should be interpreted per unit of length and the units of resistance and capacitance should scale accordingly.
We remark that the case of constant resistance and capacitance, namely $R_{a}(\vector{y}_1,t)=\overline{R}_{a}$ and $C_{a}(\vector{y}_1,t)=\overline{C}_{a}$, 
corresponds to the choice of $\alpha_0 = 0$ and $\gamma_1=0$.  In this case though, the expressions for $\omega_{11}(t)$ in \eqref{eq:y12_ex1} simplifies to 
\begin{equation}
\omega_{11}(t)=\overline{C}_{a}\left[\pi_{11,1}(t) - \overline{R}_{a}\left( Q_{11,1}(t)- C_{11,1} \frac{d\pi_{11,1}}{dt}(t) \right) \right].
\end{equation}

\noindent{\bf Example 2.}
Using a similar approach to the one developed for Example 1, it can be shown that the coupled problem~\eqref{eq:st_div}-\eqref{eq:st_ic_q} with~\eqref{eq:test2_A}-\eqref{eq:test2_flow_rate} admits the following exact solution
\begin{align}
 &\vector{v}_l(x_1,x_2,t)  = \left[s_l(t) \exV_l (x_2),\, 0 \right]^T, \quad  p_l(x_1,x_2,t) = s_l(t) \exP_l(x_1)\quad l=1,2, \\
&\pi_{11,1} (t) = P_{11,1}(t) - R_{11,1} Q_{11,1}(t),\\
&\pi_{21,1} (t) = \pi_{11,1}(t)- R_{a} \omega_{11} (t) - L_{a} \frac{d\omega_{11}}{dt}(t),\\
&\omega_{11} (t) = Q_{11,1}(t) - C_{11,1} \frac{d\pi_{11,1}}{dt}(t),
\end{align}
where
\begin{align}
 &\exV_l(x_2) = V_0 \cos^2\left(\frac{\pi x_2}{H}\right), & & \exP_l(x_1)  = a_{0l} + a_{1l} \exp(-k x_1),\\
 &P_{l1,1}(t) = s_l(t) \exP_l(L) & & l=1,2,\\
 &s_1(t) = s_0+s_1 \sin(\omega\,t), & & s_2(t) = \frac{1}{a_{02} + a_{12} + R_{21,1}V_0 \frac{H}{2}} \pi_{21,1}(t), \\
 &Q_{11,1}(t) = \frac{V_0H}{2} s_1(t), & & Q_{21,1}(t) = - \frac{V_0H}{2} s_2(t),
\end{align}
provided that the forcing terms are given by
\begin{align}
&
\vector{f}_l(x_1,x_2,t) = 
\left[
\frac{ds_l}{dt}(t) \exV_l(x_2) -\frac{\mu}{\rho} s_l(t) \frac{d^2\exV_l}{dx_2^2}(x_2) + \frac{s_l(t)}{\rho} \frac{d\exP_l}{dx_1}(x_1),\,
0\right]^T, \\
& \overline{p}_l(t) = s_l(t) \exP_l(0), \\
& \widetilde{p}(t) = -R_{b} \omega_{11}(t) + \pi_{21,1}(t) - R_{b} Q_{21,1}(t) + R_{b} C_{21,1} \frac{d \pi_{21,1}}{dt} (t),
\end{align}
for $l=1,2$, and the initial conditions are equal to the exact solution evaluated at $t = 0$.

All the parameter values involved in the exact solution  are listed in Table~\ref{tab:param_test123}, with the same convention on the units used in Example 1.\\
\noindent{\bf Example 3.}
Proceeding as in the previous examples, it can be shown that the coupled problem~\eqref{eq:st_div}-\eqref{eq:st_ic_q} with~\eqref{eq:test3_A}-\eqref{eq:test3_flow_rate} admits the exact solution
\begin{align}
& \vector{v}_1(x_1,x_2,t)  = \left[s(t) \exV (x_2), \, 0\right]^T, \qquad \; \; p_1(x_1,x_2,t) = s(t) \exP(x_1), \\
&  \pi_{11,1} (t)  = P_{11,1}(t) - R_{11,1} Q_{11,1}(t), \quad \pi_{11,2} (t)  = P_{11,2}(t) - R_{11,2} Q_{11,2}(t), \\
& \omega_{11} (t)  = \Lambda_1 + \Lambda_2\sin(\omega t) + \Lambda_3\cos(\omega t),  \label{eq:test3_exy13}
\end{align}
where
\begin{align}
&s(t) = s_0+s_1 \sin(\omega\,t), &&\exV(x_2) = V_0 \cos^2\left(\frac{\pi x_2}{H}\right),\\
&\exP(x_1)  = a_0 + a_1 \exp(-k x_1),&& P_{11,1}(t) = s(t) \exP(L), \\ 
&P_{11,2}(t) = s(t) \exP(0), && Q_{11,1}(t) = -  Q_{11,2}(t) = \frac{V_0H}{2} s(t),\\
& \Lambda_1 = s_0 \exC, && \Lambda_2 = s_1\exC \left[ \left(\Frac{\omega L_{c}}{R_{c}}\right)^2 +1 \right]^{-1},\\
& \Lambda_3 = -\Frac{\omega \Lambda_2 L_{c}}{R_{c}}, && \exC = \Frac{ \exP(L)  - \exP(0) - \left(R_{11,1} + R_{11,2}\right)  \frac{V_0H}{2} }{R_{c}},
\end{align}
provided that the forcing terms are $\vector{f}_1(x_1,x_2,t)$ as in~\eqref{eq:test1_force} and
\begin{align}
& \widetilde{p}_a (t)= R_{a}C_{11,1}\frac{d \pi_{11,1}}{dt}(t) +\pi_{11,1}(t) + \left( \omega_{11}(t) - Q_{11,1}(t)\right) R_{a},\\
& \widetilde{p}_b (t)= R_{b}C_{11,2}\frac{d \pi_{11,2}}{dt}(t) +\pi_{11,2}(t) - \left( \omega_{11}(t) + Q_{11,2}(t)\right) R_{b},
\end{align}
and the initial conditions are equal to the exact solution evaluated at time $t=0$.

Note that in~\eqref{eq:test3_exy13}, for simplicity, we consider only the particular solution of the ODE corresponding to $\omega_{11}$, setting to zero the coefficient corresponding to the homogeneous solution. All the parameter values involved in the exact solution are listed in Table~\ref{tab:param_test123}, with the same convention on the units used in the previous examples. 
\begin{table}
\scalebox{0.88}{
\centering
\begin{tabular}{llllll}
\toprule
\multicolumn{6}{c}{\bf Common parameters}\\
\toprule
{\textit{Param.}}	&  {\textit{Value}} & {\textit{Param.}}	&  {\textit{Value}} & {\textit{Param.}}	&  {\textit{Value}}\\
\midrule
$H$ 			& 2 cm 					& $L$ 				& 10 cm 				& $\rho$ 					& 1 g cm$^{-3}$\\
$\mu$ 		& 1 g cm$^{-1}$ s$^{-1}$		&  $V_0$		                 & 2 cm s$^{-1}$                 & $\omega$			         & $\pi$ s$^{-1}$\\
$k$			& 0.1 cm$^{-1}$                         & $s_0$				& 2 					& $s_1$					& 1  \\
\toprule
\multicolumn{6}{c}{\bf Example 1 parameters}\\
\toprule
{\textit{Param.}}	&  {\textit{Value}} & {\textit{Param.}}	&  {\textit{Value}} & {\textit{Param.}}	&  {\textit{Value}}\\
\midrule
$R_{11,1}$ 	& 10 g cm$^{-3}$ s$^{-1}$		& $\overline{R}_{a}$ 	& 10 g cm$^{-3}$ s$^{-1}$	 & $R_{b}$ 				& 10 g cm$^{-3}$ s$^{-1}$	\\
$\alpha_0$ 	& 10 g cm$^{-3}$ s$^{-1}$	 	& $\alpha_1$ 			& 1					& $\alpha_2$ 				& 0.001 g$^{-1}$ cm s$^2$  \\
$C_{11,1}$ 	& 0.001 g$^{-1}$ cm$^3$ s$^{2}$	& $\overline{C}_{a}$ 	& 0.01 g$^{-1}$ cm$^3$ s$^{2}$					& $\gamma_1$				& 1 cm$^{-2}$ \\
 $a_0$				& 150 g cm$^{-1}$ s$^{-2}$	& $a_1$					& 1000 g cm$^{-1}$ s$^{-2}$\\
\toprule
\multicolumn{6}{c}{\bf Example 2 parameters}\\
\toprule
{\textit{Param.}}	&  {\textit{Value}} & {\textit{Param.}}	&  {\textit{Value}} & {\textit{Param.}}	&  {\textit{Value}}\\
\midrule
$R_{11,1}$ 	& 10 g cm$^{-3}$ s$^{-1}$	    & $R_{a}$  	& 10 g cm$^{-3}$ s$^{-1}$	& $R_{21,1}$ 	 & 10 g cm$^{-3}$ s$^{-1}$	\\
$R_{b}$ 	        & 10 g cm$^{-3}$ s$^{-1}$	    & $C_{11,1}$ 	& 0.001 g$^{-1}$ cm$^3$ s$^{2}$	    & $ C_{21,1}$		& 0.001 g$^{-1}$ cm$^3$ s$^{2}$  \\
 $L_{a}$ 	 & 0.003 g cm$^{-3}$ 	    & $a_{01}$ 	        & 150 g cm$^{-1}$ s$^{-2}$	& $a_{11}$	 & 1000 g cm$^{-1}$ s$^{-2}$\\
$a_{02}$		        & 75  g cm$^{-1}$ s$^{-2}$	& $a_{12}$	 & 500 g cm$^{-1}$ s$^{-2}$\\
\toprule
\multicolumn{6}{c}{\bf Example 3 parameters }\\
\toprule
{\textit{Param.}}	&  {\textit{Value}} & {\textit{Param.}}	&  {\textit{Value}} & {\textit{Param.}}	&  {\textit{Value}}\\
\midrule
$R_{11,1}$ 	& 10 g cm$^{-3}$ s$^{-1}$	    & $R_{11,2}$ 	        & 50 g cm$^{-3}$ s$^{-1}$	& $R_{a}$ 	 & 10 g cm$^{-3}$ s$^{-1}$	\\
$ R_{b}$		& 10 g cm$^{-3}$ s$^{-1}$ & $R_{c}$  	& 70 g cm$^{-3}$ s$^{-1}$	    & $L_{c}$ 	 & 0.003 g cm$^{-3}$    \\
$C_{11,1}$ 	& 0.001 g$^{-1}$ cm$^3$ s$^{2}$	    & $C_{11,2}$ 	& 0.001 g$^{-1}$ cm$^3$ s$^{2}$	& $a_0$				& 150 g cm$^{-1}$ s$^{-2}$	\\
$a_1$					& 1000 g cm$^{-1}$ s$^{-2}$\\
\bottomrule
\end{tabular}
}
\caption{Parameter values for Examples 1, 2 ad 3. Examples are set in a two-dimensional context, hence flow rates are per unit of length and the units of resistance and capacitance scale accordingly.}
\label{tab:param_test123}
\end{table}

\section*{Acknowledgements}
The present work has been partially supported by the National Science Foundation grant DMS-1224195, Chair Gutenberg research funds awarded by the Cercle Gutenberg (France),
a STEM Chateaubriand Fellowship awarded by the Embassy of France in the United States and research funds awarded by the LabEx IRMIA (University of Strasbourg, France).

 
\bibliography{ms}

\end{document}